\documentclass{article}

\usepackage{mathrsfs, bm}

\usepackage[colorlinks,linkcolor={blue},citecolor={red}]{hyperref}
\everymath{\displaystyle}

\usepackage{amssymb,amsthm,amsmath,graphicx, mathabx}
\usepackage{color} 
\usepackage{enumerate}

\numberwithin{equation}{section}
\usepackage{rotating}
\usepackage{ wasysym }

\newcommand{\vp}{ {\scriptsize \normalfont \hspace{1.7pt}  \rotatebox{90}{\Bowtie}}}
\newcommand{\hp}{{\scriptsize \normalfont \Bowtie}}

\newcommand{\bigvp}{ { \hspace{1.7pt} \rotatebox{90}{\Bowtie}}}
\newcommand{\bighp}{{\Bowtie}}

\newcommand{\lin}{\mathrm{lin}}
\newcommand{\Fp}{\mathcal{F}}
\newcommand{\PL}{{\mathbb{P}_{\mathrm{L}}}}
\newcommand{\g}{\vec{g}}
\newcommand{\U}{\mathcal{U}}

\newcommand{\mua}{{\bar{\mu}}}

\newcommand{\wh}{\widehat}
\newcommand{\wt}{\widetilde}

\newcommand{\Z}{\mathbb{Z}}
\newcommand{\R}{\mathbb{R}}
\newcommand{\vr}{\varrho}
\renewcommand{\r}{\rho}
\newcommand{\z}{\zeta}
\renewcommand{\th}{\theta}
\newcommand{\Lmd}{\Lambda}
\newcommand{\lmd}{\lambda}
\newcommand{\F}{\mathscr{F}}
\newcommand{\supp}{\mathrm{supp}}
\newcommand{\Q}{\mathcal{Q}}
\renewcommand{\L}{\mathcal{L}}

\newcommand{\m}{\mathfrak{m}}
\newcommand{\res}{\mathrm{res}}
\newcommand{\nr}{\mathrm{nr}}

\newcommand{\f}{{ \vec{f}}}
\newcommand{\ue}{\vec{u}}

\newcommand{\rot}{\mathrm{rot}}

\newtheorem{theorem}{Theorem}[section]
\newtheorem{lemma}[theorem]{Lemma}
\newtheorem{corollary}[theorem]{Corollary}
\newtheorem{proposition}[theorem]{Proposition}
\newtheorem{Lemma A.1}{Lemma A.1}
\theoremstyle{definition}

\theoremstyle{remark}
\newtheorem{remark}[theorem]{Remark}

\allowdisplaybreaks[3]

\begin{document}
	\title{Global solutions to the Euler-Coriolis system}
	\author{Xiao Ren\footnotemark[1] \and Gang Tian\footnotemark[1]}
	
	\renewcommand{\thefootnote}{\fnsymbol{footnote}}
	\footnotetext[1]{Beijing International Center for Mathematical Research, Peking University, Beijing 100871, P.R.China. 
		
	\ \ \ Emails: renx@pku.edu.cn, gtian@math.pku.edu.cn}

	\maketitle

	\abstract{We prove the global well-posedness and scattering for the 3D incompressible Euler-Coriolis system with sufficiently small, regular and suitably localized initial data. Equivalently, we obtain the asymptotic stability for ``rigid body" rotational solutions to the pure Euler equations. This extends the recent work of Guo, Pausader and Widmayer \cite{GuoInvent} to the general non-axisymmetric setting. }

\setcounter{tocdepth}{1}
	\tableofcontents
	
	\section{Introduction}
	
	Consider the 3D incompressible Euler-Coriolis equations
	\begin{equation} 
		\left\{\begin{array}{l}
			\partial_t \ue + \ue \cdot \nabla \ue + \vec{e}_3 \times \ue + \nabla p = 0,\\
			\nabla \cdot \ue = 0,
		\end{array}\right. \label{eq:EC}
	\end{equation}
	where $\ue, p$ are the velocity and pressure fields respectively, and $\vec{e}_3$ is the unit vector $(0,0,1)$. The classical Coriolis force term $\vec{e}_3 \times \ue =  (-u_2, u_1, 0)$ in \eqref{eq:EC} 
	arises as a inertial force felt by the fluid due to a rotating frame of reference. The pressure $p$ can be recovered from $\ue$ through the elliptic equation $\Delta p = \partial_1 u_2 - \partial_2 u_1 - \mathrm{div} (\ue \cdot \nabla \ue)$.  From an alternative viewpoint, the system \eqref{eq:EC} describes Euler solutions (without Coriolis force) near a background flow given by the uniform ``rigid body" rotation. More precisely, we define the background velocity and pressure fields as
	$$\vec{U}_{\rot} = \left(-\frac{x_2}{2}, \frac{x_1}{2}, 0\right), \quad P_{\rot} = \frac{|x_h|^2}{4},$$ 
	which constitute an infinite energy stationary solution to the 3D Euler system  in full space. Let 
	$$\vec{U}(t,x) = \vec{U}_{\rot}(x) + R(t) \ue(t, R(-t) x), \quad P(t,x) = P_{\rot}(x) + p(R(-t) x),$$ 
	where we denote
	\begin{equation}
		R(t) = \left(\begin{matrix}
			\cos (t/2) & -\sin (t/2) & \\
			\sin (t/2) & \cos (t/2) & \\
			& & 1
		\end{matrix}\right).
	\end{equation}
	Then, $(\vec{U}, P)$ solves the pure Euler system
		\begin{equation} 
		\left\{\begin{array}{l}
			\partial_t \vec{U} + \vec{U} \cdot \nabla \vec{U} + \nabla P = 0,\\
			\nabla \cdot \vec{U} = 0
		\end{array}\right. \label{eq:E}
	\end{equation}
	if and only if $(\ue,p)$ solves the Euler-Coriolis system \eqref{eq:EC}. In \cite{GuoInvent}, Guo, Pausader and Widmayer were able to prove the global well-posedness and scattering result for \eqref{eq:EC} in the axisymmetric setting. The goal of this paper is to extend the result to the general non-axisymmetric case.
	
	\smallskip

	Denote $S = x \cdot \nabla$ and $\Omega = x_h^\perp \cdot \nabla_h = -x_2 \partial_1 + x_1 \partial_2$ 	(our convention is that $(a_1, a_2)^\perp = (-a_2, a_1)$), and define the Lie derivative $\bar{\Omega}$ acting on vector fields as
	\[ \bar{\Omega} \ue = \Omega \ue - \ue_h^{\perp} = x_h^\perp \cdot \nabla_h \ue - \ue_h^\perp. \]

	  \begin{theorem} \label{thm:main}
	  	There exist  constants $M_1, M_2 \ge 1$ and $\beta_0, {\varepsilon}_0 > 0$ such that if  $\ue_{0}$ is divergence free and satisfies
	  	\begin{equation} \label{eq:59-a}
	  		\|\ue_{0}\|_{H^{M_1}} + \sum_{a+b \le M_2} \|S^a \bar{\Omega}^b \ue_{0}\|_{L^2} + \||x|^{1+\beta_0}  \ue_{0} \|_{L^2} \le \varepsilon
	     \end{equation}
	     for some $0 <\varepsilon < {\varepsilon}_0$, then there exists a unique global solution $\ue \in C([0,+\infty); C^2(\R^3))$ to \eqref{eq:EC} with initial data $\ue_0$, and thus also a global solution $\vec{U}$ to \eqref{eq:E} with initial data $\vec{U}_{0} =\vec{U}_{\rot} + \ue_{0}$. Moreover, $\ue(t)$ decays over time with the optimal rate	 $\|\ue(t)\|_{L^\infty} \lesssim \varepsilon \langle t \rangle^{-1}$     and scatters linearly in $L^2$, that is, 
	     \begin{equation}
	     	\|\ue(t) - \ue_\lin(t)\|_{L^2} \to 0, \quad  t \to +\infty,
	     \end{equation}
	     for some solution $\ue_\lin$ to the linearization of \eqref{eq:EC} given by
	     	\begin{equation} 
	     	\left\{\begin{array}{l}
	     		\partial_t \ue_\lin  + \vec{e}_3 \times \ue_\lin + \nabla p_\lin = 0,\\
	     		\nabla \cdot \ue_\lin = 0.
	     	\end{array}\right. \label{eq:EC-lin-1}
	     \end{equation}
	  \end{theorem}

	  Let us make a few remarks on the main result. 
	  
	 \begin{enumerate}[(1)]
	 	\item Our theorem is inspired by the remarkable work \cite{GuoInvent} in which the axisymmetric case of our result was proved. As we know, substantial difficulties may arise from extending global existence of Euler-type equations in the axisymmetric case to the general case. Indeed, it was pointed out in \cite[Section 2.2.2]{GuoInvent} that such an extension is highly nontrivial. The new difficulty is that the nonlinearity of \eqref{eq:EC} contains a resonant subsystem of 2D Euler type which fails to satisfy the null type condition used in \cite{GuoInvent}. Further, in view of the fast norm growth of 2D Euler solutions \cite{Kiselev} and the notoriously unstable nature of 3D Euler system (see, e.g., \cite{Elgindi, Chen}), it was unclear whether such nonlinearity could lead to more complicated dynamics or even blow-up.

	 	
	 	Our proof of Theorem \ref{thm:main} is based on the discovery of a  hidden null structure in \eqref{eq:EC} without axisymmetry assumption. Roughly speaking, we observe an inherent cancellation in \eqref{eq:EC} when the bilinear inputs are symmetrized and share  the same modulus of frequency.  We are able to fully control the 2D Euler type subsystem and maintain global estimates on the dynamics based on this new null type structure.
	 	
	 	
	 	\item The dispersive effect of rotation in fluid flows has been studied in the literature from various perspectives, see, \emph{e.g.}, geophysical flows \cite{Galla, McWilliams, Ped}, the $\beta$-plane model \cite{EW, PW,  LWZZ, WZZ, BFMT, Gren}, life span and asymptotic behaviour in the case of fast rotations \cite{Angulo, Chemin1,  Dut, KLT, WanChen} and almost global stability for \eqref{eq:EC} in the axisymmetric case \cite{GuoCpam}. See also \cite{EWsiam, Takada} for studies on the stratified Boussinesq flows and dispersive SQG, where linear dispersions of a similar nature are in effect. The dispersion relation underlying the linear system \eqref{eq:EC-lin-1} is given by
	 	\begin{equation}
	 		\Lmd(\xi) = \frac{\xi_3}{|\xi|}.
	 	\end{equation}
	 This dispersive mechanism exhibits  $O(t^{-1})$ decay rate in $L^\infty$ norm (see \cite{Rensq} for a direct proof), and is strongly anisotropic and degenerate.  Despite the existing results, it is still surprising that, in the absence of viscosity, the Coriolis force alone  is sufficient to stabilize the solutions globally in time in the full 3D setting.

	 	\item Inherent null type structures (with respect to certain linear dispersion) have been observed and studied in various incompressible hydrodynamical systems, \emph{e.g.}, the strong null structure in elastodynamics \cite{Lei, Cai2} and MHD equations \cite{Cai1}, the double null form in the $\beta$-plane equation \cite{PW}. The last one is quite similar to ours as it is also revealed via  symmetrization between the two inputs. Note that the effective use of null type structures is highly nontrivial in many cases, in particular, it depends strongly on the underlying dispersion and the geometry of the resonant set. In our proof, a step-by-step multiscale bilinear decomposition (see \eqref{eq:tele}) is designed to capture the inherent cancellation between spatially coherent waves for interactions at various scales. Moreover, a delicate orthogonality argument (see \eqref{eq:51-31}) plays a crucial role in gluing the local small-scale estimates into a global one. 
	 	
	 	\item Theorem \ref{thm:main} establishes the global asymptotic stability of the uniform rotating solution $\vec{U}_{\mathrm{rot}}$, hence it naturally connects to other studies on the stability of stationary Euler flows (in 2D or 3D), \emph{e.g.}, shear flows \cite{Mas,  Wei, WZZcmp, Jia2, Jiacmp}, vortex axisymmetrization \cite{JZV, Jia1, RWDZ} and Hill's spherical vortex \cite{Choi}, although the stability mechanisms are quite different. See also \cite{BBZD} for the study on 2D stratified Couette flow. As mentioned in \cite{GuoInvent}, $\vec{U}_{\mathrm{rot}}$ is but one example of a family of general rotating stationary flows for the 3D Euler equations, given by $\vec{U}_f = f(r) \vec{e}_\th$ with $f: \R^+ \to \R$. Even in the axisymmetric case, the global stability/instability of $\vec{U}_f$ for general $f$ is open and merits further study. 
	 	

	 	\item In our condition \eqref{eq:59-a}, the third weighted $L^2$-norm can be replaced by the  $X$ and $Y$ norms used in our main bootstrap argument \eqref{eq:BA}, which are weaker for the low frequencies (see \eqref{eq:59-b}). We choose to state the main result in this way due to the relatively simple form of \eqref{eq:59-a} and a more convenient local well-posedness theory. 
	 	
	 	\item Our result can also be stated in terms of fast rotations as follows. Given initial data $\ue_{0}$ with the norms in \eqref{eq:59-a} bounded, then the system
	 		\begin{equation} 
	 		\left\{\begin{array}{l}
	 			\partial_t \ue + \ue \cdot \nabla \ue + K \vec{e}_3 \times \ue + \nabla p = 0,\\
	 			\nabla \cdot \ue = 0,
	 		\end{array}\right. 
	 	\end{equation}
	 	is globally well-posed and scatters if the constant $K>0$ is sufficiently large. This is clearly equivalent to Theorem \ref{thm:main} using the rescaling $\ue(t,x) \mapsto K^{-1}\ue(K^{-1} t,x)$.
	 \end{enumerate}

	  Next, we highlight the key ideas and methods developed in the proof of Theorem \ref{thm:main}.
	  
	  \begin{itemize}
	  	\item \emph{Vectorial dispersive unknowns.}
	  	
	  	Due to the (non-obvious) fact that 
	  	$\PL (\vec{e}_3 \times \ue) = \partial_3 \Delta^{-1} (\nabla \times \ue)$
	  	 ($\PL$ is the Leray projector), we can diagonalize the linear system \eqref{eq:EC-lin-1} using the decomposition  
	  	 $$\ue = \ue_+ + \ue_- := \frac12 ( \ue + |\nabla|^{-1} \nabla \times \ue ) + \frac12 ( \ue - |\nabla|^{-1} \nabla \times \ue ),$$ 
	  	 which has been introduced by Lei, Lin and Zhou \cite{LeiLinZhou} to diagonalize the curl operator. Note that this decomposition is orthogonal in $L^2$ and commutes well with $S, \bar{\Omega}$ derivatives and Fourier multipliers. Then, it is natural to define the vectorial profiles as 
	  	 $$\f_\pm = e^{\mp  t |\nabla|^{-1} \partial_3} \ue_\pm.$$  
	  	 By fixing a gauge (see \eqref{eq:gauge1}--\eqref{eq:gauge2}) in phase space, $\f_\pm$ can be reduced to scalar profiles, which coincides with the choice of unknowns in \cite{GuoCpam, GuoInvent}. However,  any such gauge inevitably contains artificial singularities over the sphere according to the classical hairy ball theorem. This has led to a more complicated, carefully designed bootstrap norm in \cite{GuoInvent} in order to include non-smooth functions like ${|\xi_h|^{-1}}{\xi_h}$ arising from the gauge. We avoid such difficulty by working directly with the vectorial profiles $\f_\pm$ for vertical frequencies. This helps us to simplify considerably the choice of norms and the nonlinear analysis for frequencies with $\Lmd$ very close to $1$.

	  	\item \emph{Propagation of microlocal type norms.}
	  	
	  	The linear dispersion $e^{i t \Lmd}$ decays pointwisely with the slow rate $t^{-1}$ in two ``bad" physical regions: either close to the $x_3$-axis or close to the $x_h$-plane. On the Fourier side, they correspond to frequencies $\xi$ with $\Lmd(\xi) \approx 0$ or $\Lmd(\xi) \approx 1$ respectively. To separately handle the difficulties caused by these two types of frequencies, we introduce a basic smooth partition of unity for the phase space $1 \equiv \chi^\vp(\xi) + \chi^\hp(\xi)$, where the supports of $\chi^\vp$ and $\chi^\hp$ are contained in $|\Lmd| \ge 0.55$ (vertical frequencies) and $|\Lmd| \le 0.6$ (horizontal frequencies) respectively.  
	  	
	  	\smallskip
	  	
	  	The norms to be propagated via a standard bootstrap argument (see Proposition \ref{prop:BA}) will be $L^2$-based and incorporate a number of natural/good derivatives $\nabla, S$ and $\Omega$. The $X$ and $Y$ norms are designed for horizontal and vertical frequencies respectively. Together, they provide a  fractional  ($1+\beta$ order, $\beta \ll 1$) angular regularity over the sphere for the Fourier transform of the profiles, which is similar to the bootstrap norms used in \cite{GuoInvent} (see also \cite{PW} for a similar strategy). We emphasize that the definitions of the $X$ and $Y$ norms involve dyadic localizations both in frequency and in space. In contrast to the use of angular Littlewood-Paley projections in \cite{GuoInvent} involving spherical harmonics, here we employ simple spatial localization operators $Z_l$ and $H_l$ (see \eqref{eq:515-1}) to quantify the fractional Besov regularity on the Fourier side. This enables us to fully separate the estimates for the horizontal and vertical frequencies and effectively apply commutator estimates in many cases. Most importantly, such spatial localizations allow us to make intensive use of the sharp (pointwise) linear asymptotics (see Section \ref{sec:4})  for the nonlinear analysis.

	  	\item \emph{Handling the critical decay in time.}
	  	
	  	When propagating the $X$ norm for an $(l,k)$-piece ($l, k$ are localization parameters in space and frequency respectively) of the profile, an important feature is that the nonlinear contributions will have a critical decay rate in general. For instance, in the analysis of low-high/high-low interactions with horizontal input frequencies, the new null condition we mentioned above is useless. Applying commutator estimates and an $L^\infty$-$L^2$ type bound, we can estimate the nonlinear contribution by $t^{-1}$ in $X$ norm, which is not integrable in time in general. Nonetheless, based on the lower and upper bounds for the group velocity in the $x_3$-direction, we observe that for low-high and high-low interactions (with the two inputs  also suitably localized), the decay can be upgraded to integrable rates (cf. the term $2^{-(1+\beta)m}$ in \eqref{eq:807-o1}) \emph{outside} a resonant time interval $[t_*, C t_*]$ with $t_*$ depending on the localization parameters. Then, thanks to  the fact that
	  	\begin{equation}
	  		\int_{t_*}^{C t_*} t^{-1} dt \lesssim 1,
	  	\end{equation}
	  	we are still able to close the estimates globally (cf. the term $2^{-|m-l|}$ in \eqref{eq:809-yyy1}).  The high-high interactions also give a critical contribution even with the application of the null structure.
	  	
	  	\smallskip
	  	
	  	The critical nature of the problem makes our analysis much more delicate since we cannot afford any loss in the decay especially within the resonant time intervals. Firstly, we are forced to prove a set of sharp linear dispersive estimates (see Section \ref{sec:4}) which avoids the loss in $\beta$. Secondly, we simultaneously propagate two $X$ norms with different parameters, namely $X_{\beta}^{10}$ and $X_{\beta'}^{20,20}$ with $0<\beta'<\beta$, in order to overcome the inevitable loss in $S, \Omega$ derivatives in the dispersive estimates. Thirdly, in some cases we have to use energy type estimates instead of Duhamel's formula, in order to overcome the quasilinear nature of the equations. Note that energy type estimates are not necessarily compatible with normal forms, and we must avoid using them at the same time for many pieces of the output.

	  	\item \emph{A multiscale decomposition for bilinear forms.}
	  	
	  	In the analysis of high-high interactions for horizontal frequencies, we develop a step-by-step decomposition for bilinear forms according to the size of $|(\xi-\eta)_h|-|\eta_h|$ where $\xi-\eta$ and $\eta$ are the two input frequencies (see Section \ref{sec:742}). Roughly speaking, at the $J$-th step of the decomposition, we treat bilinear interactions with  $\big||(\xi-\eta)_h|-|\eta_h|\big| \sim 4^{-J}$. Then, by combining precised wave packet estimates for frequency localized near a circle (see Section \ref{sec:52}) with orthogonality properties of the bilinear decomposition, we are able to maximally utilize the cancellation in the null structure, which gives a crucial smallness factor for the interaction between coherent wave packets. Moreover,  an important idea in the analysis is to use (a version of) the uncertainty principle to achieve an extra smallness factor at small scales, which is important for the summability in $J$ (see \eqref{eq:51-21}--\eqref{eq:51-22}).

		\item \emph{The method of partial symmetries.}
		
		Our proof still relies heavily on the spacetime resonances framework developed in \cite{GuoInvent} called ``the method of partial symmetries", which features the use of angular dyadic decompositions (with parameters $p,q$), repeated integration by parts along $S$ and $\Omega$ vector fields as well as normal forms. The key observation behind the formalism is that, away from a resonant set (see \eqref{eq:524-3}), one can either apply integration by parts using easily accessible derivatives or perform a normal form. We closely follow many important ideas from \cite{GuoInvent}, especially in the analysis involving vertical frequencies, see Section \ref{sec:9}. Nonetheless, some simplifications are made along the way due to our simpler choice of norms and a better gain of decay for integration by parts in space. The reader may find more backgrounds on the general formalisms for spacetime resonances in, \emph{e.g.}, \cite{Klainerman, Shatah, Germain, GMS, GMS2}.
	  \end{itemize}

	  \subsection{Organization of the article}
	  
	  In Section \ref{sec:overview}, we introduce the dispersive unknowns, the functional framework and the key bootstrap assumptions  that will be used throughout the article. Moreover, in Section \ref{sec:25-boot} we give an overview on the necessary bilinear/trilinear estimates for the propagation of bootstrap norms.  
	  
	  In Section \ref{sec:3}, we establish the basic energy estimates involving natural and good derivatives, and examine the null structure and the energy structure of the system \eqref{eq:EC} at the level of bilinear multipliers. Then, the sharp dispersive estimates are presented in Section \ref{sec:4} as a foundation for all the subsequent analysis. In Section \ref{sec:dt}, we apply the dispersive estimates and $L^2$-$L^\infty$ type bounds to obtain general decay rates for $\partial_t S^a \bar{\Omega}^b \f_\pm$ in $L^2$. In particular, the case of large $l+k$ is tackled in Lemma \ref{lem:428-1}, so that in subsequent sections we can focus on the case $l+k < (1+\delta)m$. 
	  
	  The method of partial symmetries is briefly recalled in Section \ref{sec:srm}. The most delicate part of the article is contained in Section \ref{sec:8}, which focuses on the $\bighp + \bighp \to \bighp$  interactions. Finally, the other types of interactions ($\bighp + \bighp \to \bigvp$, $\bigvp+ \bigvp \to \bigvp$, $\bigvp+\bigvp\to \bighp$, $\bighp+ \bigvp$) are treated in Section \ref{sec:9}.

	  \section{Overview of the proof} \label{sec:overview}
	  
	 \paragraph{Notation.} Throughout the paper, we denote the Fourier transforms by
	 \begin{equation}
	 	(\F f) (\xi) = \widehat{f} (\xi)  = \int_{\R^3} f(x) e^{-i\xi \cdot x} dx, \quad
	 		(\F ^{-1}g) (x)  = \frac{1}{(2\pi)^3}\int_{\R^3} g(\xi) e^{i\xi \cdot x} d\xi,
	 \end{equation}
	 \begin{equation}
	 	(\F_h f) (\xi_h, x_3)  = \int_{\R^2} f(x) e^{-i\xi_h \cdot x_h} dx_h, \quad  	(\F_3 f) (x_h, \xi_3)  = \int_{\R} f(x) e^{-i\xi_3 \cdot x_3} dx_3.
	 \end{equation}
 	With slight abuse of notation, we will write $\Lmd$ for the dispersion relation
 	\begin{equation}
 		\Lmd(\xi) = \frac{\xi_3}{|\xi|}
 	\end{equation}
 	as well as for the corresponding operator
 	\begin{equation} \label{eq:812-aa1}
 		\Lmd = -i|\nabla|^{-1} \partial_3,
 	\end{equation}
 	depending on the context.
	For a given multiplier function $\m(\xi, \eta)$ and a multi-index $\mua = (\mu, \mu_1, \mu_2) \in \{+, -\}^3$, define the corresponding bilinear form $\Q_\mua[\m](\cdot, \cdot)$ at time $t$ as
	\begin{equation} \label{eq:510-52}
		\left\{\F \Q_\mua[\m](f_1, f_2)\right\}(\xi)=  \frac{1}{(2\pi)^3}\int_{\R^3} e^{it\Phi_\mua(\xi, \eta)} \m(\xi, \eta) \wh{f_1}(\xi -\eta) \wh{f_2}(\eta) d\eta,
	\end{equation}
	with the phase function given by
	\begin{equation}
		\Phi_\mua(\xi, \eta) = -\mu \Lmd(\xi) + \mu_1 \Lmd (\xi-\eta) + \mu_2 \Lmd (\eta),
	\end{equation}
	and we also denote that
	\begin{equation}
		\L_\mua[\m](f_1, f_2) = e^{\mu i  t \Lmd } \Q_\mua[\m](f_1, f_2).
	\end{equation} 
	
	\medskip
	
	\medskip
	
	\paragraph{Important constants.} Throughout the paper, we shall take the constants $M_1, M_2, \delta_0, \delta, \beta_0$ such that 
	$$M_1 = M_2 \gg \delta_0^{-1}$$ 
	and 
	$$0< \delta_0 \ll  \delta \ll \beta_0 \ll 1.$$ 
	Moreover, we will define
	$$\beta = \frac23\beta_0, \quad \beta' = \frac23 \beta.$$
	In practice, it is sufficient to take
	\begin{equation*}
		\beta_0 = 10^{-2}, \ \delta = 10^{-3}, \ \delta_0 = 10^{-4}, M_1=M_2=10^6.
	\end{equation*}
	We have not striven to optimize the number of vector fields and derivatives used in the proof.

	\subsection{The dispersive unknowns} \label{sec:12}
	We denote the Leray projector by
	\begin{equation}
		\PL  = \mathrm{Id} - \nabla \Delta^{-1} \nabla \cdot .
	\end{equation}
	The linear part of the Euler-Coriolis system, namely,
	\begin{equation}  \label{eq:EC-lin}
		\partial_t \ue  + \PL \, (\vec{e}_3 \times \ue) = 0,
	\end{equation}
	can be diagonalized in the following way. Let ${\vec \omega} = \nabla \times \ue$ be the vorticity, and define
	\begin{equation} \label{eq:57-2}
		\ue_+ = \frac12 ( \ue + |\nabla|^{-1} \vec\omega ), \quad \ue_- = \frac12 ( \ue - |\nabla|^{-1} \vec\omega ).
	\end{equation}
	Notice that
	\begin{equation} \label{eq:58-1}
		\ue = \ue_+ + \ue_-, \quad \nabla \times \ue_+ = |\nabla| \ue_+, \quad \nabla \times \ue_- = - |\nabla| \ue_-.
	\end{equation}
	Such a decomposition of velocity was introduced by Lei, Lin and Zhou \cite{LeiLinZhou} to exploit the structure of helicity and construct a class of large global solutions to the 3D Navier-Stokes equations. Using \eqref{eq:EC-lin}, we obtain
	\begin{align}
		\partial_t \ue_+ &= \frac12 (\partial_t \ue + |\nabla|^{-1} \nabla \times \partial_t \ue) \\
		&= -\frac12 ( \PL(\vec{e}_3 \times \ue) + |\nabla|^{-1} \nabla \times \PL (\vec{e}_3 \times \ue)) \\
		&= -\frac12 (\vec{e}_3 \times \ue + \nabla \Delta^{-1} \omega_3 -|\nabla|^{-1} \partial_3 \ue ).
	\end{align}
	Using tha fact that
	\begin{equation}
		\vec{e}_3 \times \Delta \ue + \nabla \omega_3 = -\vec{e}_3 \times (\nabla \times \vec \omega)  + \nabla \omega_3 = \partial_3 {\vec \omega},
	\end{equation}
	we get
	\begin{equation} \label{eq:57-1}
		\partial_t \ue_+ = -\frac12 (\partial_3 \Delta^{-1} \vec \omega - |\nabla|^{-1} \partial_3 \ue) = |\nabla|^{-1} \partial_3 \ue_+.
	\end{equation}
	 With the notation \eqref{eq:812-aa1}, \eqref{eq:57-1} can be simply written as
	\begin{equation}
		\partial_t \ue_+ = i \Lmd \ue_+.
	\end{equation}
	Similarly, we also have
	\begin{equation}
		\partial_t \ue_- = -i \Lmd \ue_-.
	\end{equation}
	It is interesting that the decomposition \eqref{eq:57-2} simultaneously diagonalize the operators $\nabla \times \cdot$ and $\PL(\vec{e}_3 \times \cdot)$ acting on divergence free vector fields. 
	

	To exploit the above dispersive mechanism, it is natural to define the dispersive profiles for the nonlinear problem \eqref{eq:EC} as
	\begin{equation} \label{eq:58-01}
		\f_+ = e^{-it \Lmd}\ue_+, \quad \f_- = e^{i t\Lmd} \ue_-.
	\end{equation}
	Denote 
	\begin{equation}
		\mathbb{P}_+ = \frac12 (\mathrm{Id} + |\nabla|^{-1} \nabla \times),\quad \mathbb{P}_- = \frac12 (\mathrm{Id} - |\nabla|^{-1} \nabla \times).
	\end{equation}
	The nonlinear evolution of $\f_\pm$ is then given by
	\begin{align} \label{eq:58-2}
		\partial_t \f_\pm &= e^{\mp i t \Lmd} (\partial_t \ue_\pm \mp i \Lmd \ue_\pm) \nonumber \\
	    &=  -e^{\mp i t \Lmd} \mathbb{P}_{\pm} \, \PL \nabla \cdot (\ue \otimes \ue) \nonumber \\
		&=  -e^{\mp i t \Lmd} \mathbb{P}_{\pm} \, \PL \nabla \cdot (e^{it\Lmd} \f_+ \otimes e^{it\Lmd} \f_+ + e^{it\Lmd} \f_+ \otimes e^{-it\Lmd} \f_- \nonumber \\
		&\qquad \qquad \qquad +e^{-it\Lmd} \f_- \otimes e^{it\Lmd} \f_+ +e^{-it\Lmd} \f_- \otimes e^{-it\Lmd} \f_-),
	\end{align}
	where the conventions $({\vec a} \otimes {\vec b})_{ij} = a_i b_j $ and $\nabla \cdot \mathbb{M} = \sum_i \partial_i M_{ij}$ are used.

	We emphasize that $\f_+$ and $\f_-$ are vector-valued functions. However, due to the two constraints $\nabla \cdot \f_{\pm} = 0$ and $\nabla \times \f_{\pm} = \pm |\nabla| \f_{\pm}$, they  can be reduced to scalar profiles via an appropriate choice  of basis vector fields on the Fourier side (cf. \cite{Zhou}). Indeed, taking the Fourier transform of the constraints, we obtain that
	\begin{equation}
		\xi \cdot \F {\f_\pm}(\xi) = 0, \quad i\xi \times \F {\f_\pm}(\xi) = \pm |\xi| \F {\f_\pm}(\xi).
	\end{equation}
	We will choose a vector field $\vec\Gamma_1(\xi)$ which is 0th order homogeneous, \emph{i.e.}, $\vec\Gamma_1(s\xi)=\vec\Gamma_1(\xi), \ \forall s>0$. Moreover, we require that $\xi \cdot \vec\Gamma_1 = 0$ and $|\vec\Gamma_1| > 0$ over the unit sphere. Let $\vec\Gamma_2 = \frac{\xi}{|\xi|} \times \vec\Gamma_1$.   Then, there exist scalar functions $a_\pm(\xi)$ such that
	\begin{equation}
		\F {\f_\pm}(\xi) =  i \, a_\pm(\xi) \cdot (\vec\Gamma_1(\xi) \pm i \vec\Gamma_2(\xi)).
	\end{equation}
	Note that, due to the well-known hairy ball theorem, $\vec\Gamma_1$ cannot be smooth over the unit sphere. To fix the gauge, in this paper we take
	\begin{equation} \label{eq:gauge1}
		\vec{\Gamma}_1 = \left( \frac{-\xi_2}{|\xi_h|}, \frac{\xi_1}{|\xi_h|}, 0 \right)
	\end{equation}
	and
	\begin{equation} \label{eq:gauge2}
		\vec{\Gamma}_2 = \frac{\xi}{|\xi|} \times \vec{\Gamma}_1 = \left( -\frac{\xi_1 \xi_3}{|\xi_h||\xi|}, -\frac{\xi_2 \xi_3}{|\xi_h| |\xi|}, \frac{|\xi_h|}{|\xi|} \right).
	\end{equation}
	Then, the corresponding scalar profiles coincide with the choice of unknowns in \cite{GuoInvent}, as we explain below. 
	
	Let us define two vector-to-scalar operators $R_\pm$ as  
	\begin{equation}
		R_\pm \ue = \frac{1}{2} \vec{e}_3 \cdot (| \nabla_h |^{- 1} \nabla \times \ue \pm | \nabla | | \nabla_h |^{- 1} \ue)
	\end{equation}
	and denote
	\begin{equation}
	\U_+ = R_+ \ue, \quad		\U_- = R_- \ue,
		\end{equation}
	\begin{equation} \label{eq:def-Fp}
		\Fp_+ = e^{- i t \Lambda} \U_+ = R_+ \f_+, \quad \Fp_- = e^{i t \Lambda} \U_- = R_- \f_-.
	\end{equation} 
	Note that $R_\pm$ satisfy the properties
	\begin{equation}
		R_\pm \ue = R_\pm \ue_\pm, \quad R_\pm \ue_\mp = 0.
	\end{equation}
	By construction, $\ue_\pm, \f_\pm$ can be uniquely determined by $\U_\pm, \Fp_\pm$, and we formally write
	\begin{equation}
		\ue_\pm = R_\pm^{-1} \U_\pm := (- \nabla_h^{\perp} | \nabla_h |^{- 1} \pm i \Lambda \nabla_h |
		\nabla_h |^{- 1} ) \, \U_\pm \pm \sqrt{1 - \Lambda^2} \,  \U_\pm \vec{e}_3,
	\end{equation}
	and
	\begin{equation} \label{eq:525-1}
		\f_\pm = R_\pm^{-1} \Fp_\pm := (- \nabla_h^{\perp} | \nabla_h |^{- 1} \pm i \Lambda \nabla_h |\nabla_h |^{- 1} ) \, \Fp_\pm \pm \sqrt{1 - \Lambda^2} \, \Fp_\pm \vec{e}_3.
	\end{equation}
	On the Fourier side, we indeed have 
	\begin{equation} \label{eq:57-01}
		\F {\ue_\pm} = -i \, (\F {\U_\pm}) (\vec{\Gamma}_1 \pm i \vec{\Gamma}_2), \quad \F{\f_\pm} = -i \, (\F {\Fp_\pm}) (\vec{\Gamma}_1 \pm i \vec{\Gamma}_2).
	\end{equation}
	We emphasize that $\vec{\Gamma}_{1,2}$ are not smooth  at the $\xi_3$-axis, which made the analysis in \cite{GuoInvent} more involved. In particular, the key bootstrap norm there is designed carefully  to include functions like $\frac{\xi_h}{|\xi_h|}$. In this paper, we avoid such difficulty by working directly with the vectorial profiles $\f_\pm$ for $\Lmd$ close to $1$. 
	 

	\smallskip
	
	We point out that the inherent null type structure of the nonlinear system \eqref{eq:EC} is already visible at the vectorial level.  Using the identity
	\begin{equation}
		\PL (\ue \cdot \nabla \vec{v}) = -\PL (\ue \times (\nabla \times \vec{v}))
	\end{equation}
	and the property \eqref{eq:58-1}, we can further write \eqref{eq:58-2} as
	\begin{align} 
		\partial_t \f_\pm &=  e^{\mp i t \Lmd} \mathbb{P}_{\pm} \, \PL  (e^{it\Lmd} \f_+ \times |\nabla| e^{it\Lmd} \f_+ - e^{it\Lmd} \f_+ \times |\nabla| e^{-it\Lmd} \f_- \nonumber \\
		&\qquad \qquad \qquad +e^{-it\Lmd} \f_- \times |\nabla| e^{it\Lmd} \f_+ - e^{-it\Lmd} \f_- \times |\nabla| e^{-it\Lmd} \f_-). 
	\end{align}
	By direct computation, the group velocity of the dispersion $e^{\pm i t \Lmd}$ in the $x_3$-direction is given by
	\begin{equation}
		\mp \partial_3 \Lmd = \mp \frac{|\xi_h|^2}{|\xi|^3}.
	\end{equation}
	The mixed terms $- e^{it\Lmd} \f_+ \times |\nabla| e^{-it\Lmd} \f_- +e^{-it\Lmd} \f_- \times |\nabla| e^{it\Lmd} \f_+$ are space nonresonant for input frequencies with $\Lmd \approx 0$, due to an opposite sign in the group velocities. The remaining two terms exhibit a null type structure via symmetrization, \emph{e.g.},
	\begin{equation} \label{eq:58-3}
		e^{it\Lmd} \f_+ \times |\nabla| e^{it\Lmd} \f_+  = \frac12 (e^{it\Lmd} \f_+ \times |\nabla| e^{it\Lmd} \f_+ - |\nabla| e^{it\Lmd} \f_+ \times  e^{it\Lmd} \f_+ ).
	\end{equation}
	The expression \eqref{eq:58-3} induces a crucial factor $|\xi - \eta| - |\eta|$ on the Fourier side ($\xi-\eta$ and $\eta$ are the frequencies of the two inputs), which vanishes on the key  resonant set
	\begin{equation} \label{eq:524-3}
		\{(\xi, \eta) \in \R^3 \times \R^3 : (\partial_3 \Lmd)(\xi-\eta) = (\partial_3 \Lmd)(\eta), \quad  \Lmd(\xi-\eta) = \Lmd(\eta)=0\}.
	\end{equation}
	This is closely related to the classical fact that the Beltrami flows are stationary solutions to the Euler system (cf. \cite{LeiLinZhou}). In Section \ref{sec:null}, we will examine such null type structure in more details at the scalar level, see Remark \ref{rem:null-sym}.

	\subsection{Localizations} \label{sec:515-22}
	Let $\psi \in C^{\infty} (\mathbb{R}, [0, 1])$ be an even non-increasing smooth bump
	function supported on $\left[ - 2, 2 \right]$ with $\psi
	\left|_{\left[ - 1.8, 1.8 \right]} \equiv 1 \right.$ and set
	$\varphi (x) = \psi (x) - \psi (2 x)$. By construction $\varphi$ is supported on $[-2, -0.9] \cup [0.9, 2]$. In general, we write $\wt{\psi}$ and
	$\wt{\varphi}$ for functions with similar support properties as $\psi$ and
	$\varphi$, which satisfy the identities $\wt{\psi} \psi =
	\psi$ and $\wt{\varphi} \varphi = \varphi$.
	
	\smallskip
	
	On the Fourier side, we shall make intensive use of the partition of unity
	\begin{equation} \chi^\hp (\xi) = \phi \big(\Lambda (\xi)\big), \quad \chi^\vp(\xi) = 1 -  \chi^\hp(\xi) = 1-\phi\big(\Lambda(\xi)\big), \end{equation}
	where the smooth cut-off function $\phi$ is chosen such that $\chi^\hp$ is supported on $\left\{\xi \in \R^3 :  | \Lambda(\xi) |
	\leqslant 0.6 \right\}$ and $\chi^\vp$ is supported on $\left\{\xi \in \R^3 : |
	\Lambda(\xi) | \ge 0.55 \right\}$. 
	For technical reasons we introduce another pair of cut-off functions $\widetilde{\chi}^\hp$ and $\widetilde{\chi}^\vp$ which also depend on $\xi$ through $\Lambda(\xi)$ and satisfies
	\begin{equation}\widetilde{\chi}^\hp(\xi) = \begin{cases} 0 \quad \text{for} \quad  \Lambda(\xi)>0.62 \\ 1 \quad \mbox{for}\quad \Lambda(\xi)<0.61  \end{cases} \mbox{and} \quad \widetilde{\chi}^\vp(\xi) = \begin{cases} 0 \quad \text{for} \quad \Lambda(\xi)<0.53 \\ 1 \quad \text{for} \quad \Lambda(\xi)>0.54 \end{cases}\end{equation}
	so that
	\begin{equation}\widetilde{\chi}^\hp \chi^\hp = \chi^\hp, \quad \widetilde{\chi}^\vp \chi^\vp = \chi^\vp.\end{equation}
	The associated projection type operators on the physical side will be denoted by
	\begin{equation}
		P^\hp f = \F^{-1} \left( \chi^\hp\widehat{f} \right), \quad P^\vp f = \F^{-1} \left( \chi^\vp\widehat{f} \right),
	\end{equation}
and
	\begin{equation}
		\wt{P}^\hp f = \F^{-1} \left( \wt{\chi}^\hp\widehat{f} \right), \quad \wt{P}^\vp f = \F^{-1} \left( \wt{\chi}^\vp\widehat{f} \right).
	\end{equation}

\medskip

	Similar to the standard Littlewood-Paley
	decompositions, we make the following dyadic localizations according to the modulus of horizontal/vertical frequency:
	\begin{equation} \label{eq:811-s1}
		\chi^\hp_k (\xi) = \chi^\hp(\xi) \varphi (2^{- k} | \xi_h |), \quad \chi^\vp_k
		(\xi) = \chi^\vp(\xi) \varphi (2^{- k} | \xi_3 |)
	\end{equation}
	\begin{equation}
		\chi_k = \chi^\hp_k + \chi^\vp_k.
	\end{equation}
	Sometimes we have to localize   according the size of $\Lambda$ and $\sqrt{1 - \Lambda^2}$ as in \cite{GuoInvent} to handle the resonances in regions where $\Lambda \approx 0$ or $\sqrt{1-\Lmd^2} \approx 0$. Given integer parameters $p, q \le -2$, we introduce
	\[ \chi^{\hp,q} (\xi) = \varphi (2^{- q} \Lambda (\xi)), \quad
	\chi^{\vp,p} (\xi) = \varphi \left( 2^{- p} \sqrt{1 -
		\Lambda^2 (\xi)} \right). \]
	For technical reasons, we also take smooth cut-off functions $\chi^{\hp, -1}$ and $\chi^{\vp,-1}$ depending on $\xi$ through $\Lambda(\xi)$ such that, for $\iota \in \{\bighp, \bigvp\}$,
	\begin{equation}
		\chi^{\iota} \chi^{\iota, \le -1}  = \chi^{\iota}, \quad \wt{\chi}^{\iota} \chi^{\iota, \le -1} = \chi^{\iota, \le -1}.
	\end{equation}
	Here, we have used the notations
	\begin{equation}\chi^{\hp, \le -1} = \sum_{q \le -1} \chi^{\hp, q}, \quad \chi^{\vp, \le -1} = \sum_{p \le -1} \chi^{\vp, p}.\end{equation}
	We shall always assume that $p, q \le -1$. Similar to \eqref{eq:811-s1} we define
	\begin{equation} \chi^{\hp,q}_k (\xi) = \chi^{\hp,q}(\xi) \varphi (2^{- k} | \xi_h |), \quad
	\chi^{\vp,p}_k (\xi) = \chi^{\vp,p}(\xi) \varphi (2^{- k} | \xi_3 |). \end{equation}
	The associated projection type operators on the physical side will be denoted by $P_k$, $P^\hp_k$, $P^{\hp,q}_k$, \emph{etc.} For instance, we have that $\F (P_k^\hp f) = \chi_k^\hp \wh{f}$ and $\F (P_k^{\vp,p} f) = \chi_k^{\vp,p} \wh{f}$. 
	
	Furthermore, we need two operators $Z_l, H_l$ for localizations in physical variables, which are simply defined as
	\begin{equation} \label{eq:515-1}
	Z_l f (x) = \varphi (2^{- l} | x_3 |) f(x), \quad  H_l f (x) = \varphi (2^{- l} | x_h |) f(x) . 
	\end{equation} 
	Due to the uncertainty principle, the commutators $[P^\hp_k, Z_l]$ and
	$[P^\vp_k, H_l]$ can be controlled (up to low order terms) only when $l + k$
	is not too negative. For such reason we also introduce the modified versions
	\begin{equation} Z_l^{(k)} = \left\{\begin{array}{l}
		Z_l, \hspace{4em} l \geqslant - k + 1,\\
		Z_{\leqslant -k}, \hspace{3em} l = - k,\\
		0, \hspace{4em} l \leqslant - k - 1,
	\end{array}\right. \end{equation}
	and
	\begin{equation} H_l^{(k)} = \left\{\begin{array}{l}
		H_l, \hspace{4em} l \geqslant - k + 1,\\
		H_{\leqslant -k}, \hspace{3em} l = - k,\\
		0, \hspace{4em} l \leqslant - k - 1.
	\end{array}\right. \end{equation}
	These operators are useful for quantifying the vertical and horizontal regularity of $\wh{f}(\xi)$ respectively. In particular, we have the Bernstein type properties
	\begin{equation}
		\|\partial_{\xi_3} \F(Z_{\le l} f)\|_{L^2}  \lesssim 2^{l} \| Z_{\le l} f\|_{L^2} 
	\end{equation}
	and
	\begin{equation}
		\|\nabla_{\xi_h} \F(H_{\le l} f)\|_{L^2}  \lesssim 2^{l} \| H_{\le l} f\|_{L^2}.
	\end{equation}
	Sometimes, we also need
	\begin{equation}
		Z^+_lf(x) =  \mathbf{1}_{x_3 > 0} \cdot Z_l f(x), \quad Z^-_lf(x) =  \mathbf{1}_{x_3 < 0} \cdot Z_l f(x).
	\end{equation}

	\paragraph{Conventions.} We will often use the slightly less formal relations $ \lesssim, \sim, \ll$, \emph{etc.} when the constants involved are independent of other key parameters and their exact values are unimportant for the proof. With slight abuse of notation, given a dyadic scale $2^p$, we shall write $p \lesssim 0$ for the additive bound $p \le C$, and also write $2^p \lesssim 1$ for the multiplicative bound $2^p \le C$.
	
	In the proof, we often localize according to a range of parameters instead of a single value. For example, $P^\hp_{[a,b]} = \sum_{k\in[a,b]} P^\hp_k$, $P^\hp_{[a,b]^c} = \sum_{k\notin[a,b]} P^\hp_k$, $P^\hp_{\le k} = P^{\hp}_{(-\infty, k]}$ and $P^\hp_{\sim k} = P^\hp_{[k-C_1,k+C_2]}$ for some constants $C_1, C_2 \in [1,10]$ whose exact values are unimportant for the proof (unless otherwise specified).

\subsection{Local charts and vector fields} \label{sec:22}

On the Fourier side, in standard cylindrical coordinates $(|\xi_h|, \xi_3, \th)$, we have
\begin{equation}
	\xi_1 = |\xi_h| \cos \th, \quad \xi_2 = |\xi_h| \sin \th.
\end{equation}
For our purpose, it is convenient to introduce two modified cylindrical coordinate systems on the support of $\chi^\hp$ and $\chi^\vp$ respectively, which are more compatible with the dispersion relation $\Lambda = \xi_3 / |\xi|$, and induce various useful (commuting) vector fields. 

\smallskip 

On the support of $\chi^\hp$, we often use $(\r, \z, \th)$ with
\begin{equation}
	\r = |\xi_h|, \quad \z = \frac{\xi_3}{|\xi_h|} = \frac{\Lmd}{\sqrt{1-\Lmd^2}}.
\end{equation}
Note that
\begin{equation}
	\xi_1 = \r \cos \th, \quad \xi_2 = \r \sin \th, \quad  \xi_3 = \r \z.
\end{equation}
Partial derivatives taken in the modified chart $(\r, \z, \th)$ induce three vector fields which commute with each other, \emph{i.e.},
\begin{equation} \label{eq:413-1}
	\partial_{\ln \r} = \r \partial_\r = S_\xi = \xi \cdot \nabla_\xi,
\end{equation}
\begin{equation} \label{eq:413-2}
	\partial_\z = \r \, \vec{e}_3 \cdot\nabla_\xi,
\end{equation}
\begin{equation} \label{eq:413-3}
	\partial_\th = \Omega_\xi = \xi^\perp_h \cdot \nabla_\xi.
\end{equation}
The standard volume element can be written as 
\begin{equation}
	d\xi = \r^2 d \r d \z d \th.
\end{equation}
Analogously, on the support of $\chi^\vp$, we introduce $(\vr, z, \th)$ with
\begin{equation}
	\vr = \frac{|\xi_h|}{|\xi_3|} = \frac{\sqrt{1-\Lmd^2}}{|\Lmd|}, \quad z = \xi_3,
\end{equation}
so that
\begin{equation}
	\xi_1 = \vr\, |z| \cos \th, \quad \xi_2 = \vr \,|z| \sin \th, \quad  \xi_3 = z.
\end{equation}
This new chart also naturally induces three commuting vector fields:
\begin{equation} \label{eq:425-1}
	\partial_{\ln |z|} = z \partial_{z} = S_\xi = \xi \cdot \nabla_\xi,
\end{equation}
\begin{equation} \label{eq:425-2}
	\partial_\vr = \vr^{-1} \xi_h \cdot \nabla_\xi,
\end{equation}
and again,
\begin{equation} \label{eq:425-3}
	\partial_\th = \Omega_\xi = \xi^\perp_h \cdot \nabla_\xi.
\end{equation}
The standard volume element can be written as 
\begin{equation}
	d\xi = \vr z^2 d z d \vr d \th.
\end{equation}
Note that $S_\xi \Lmd(\xi) = \Omega_\xi \Lmd(\xi) = 0$ and
\begin{equation}
	\F (S f) = -(S_\xi+3) \wh{f}, \quad \F(\Omega f) = \Omega_\xi \wh{f}.
\end{equation}
We remark that $\partial_\z$ and $\partial_\vr$ do \emph{not} commute with each other as they are taken in different coordinate charts. More importantly, they also do \emph{not} commute with the linear dispersion $e^{i t \Lmd}$. As a consequence, it is much harder to propagate regularity in $\partial_\z$, $\partial_\vr$ than in $S_\xi$, $\Omega_\xi$.

\subsection{Choice of norms} \label{sec:24}

With the notation $k^+ = \max\{0,k\}$, we introduce a collection of key (semi-)norms, which are of $L^2$-based Besov type and incorporate a given number of normal derivatives or $S, \Omega$ derivatives. The $X$ norms, parameterized by $n_1, n_2, \beta$, are designed to capture the regularity of $\chi^\hp \wh{f}$ in the vector field $\partial_\z$ (in the vertical direction):
\begin{equation} \label{eq:46-11}
	\|f\|_{X^{n_1,n_2}_\beta} = \sup_{k,l\in\mathbb{Z}} \ \sup_{a+b\le n_2}  2^{n_1 k^+ +(1+\beta)(l+k)} \| Z_l^{(k)} S^a \Omega^b P^\hp_k f\|_{L^2} 
\end{equation}
Here, $n_1, n_2 \in \mathbb{N}$ and $\beta \ge 0$ are parameters to be chosen.  We then define
\begin{equation} \label{eq:46-12}
	\|f\|_{X^n_\beta} = \sup_{n_1+n_2 \le n} \|f\|_{X^{n_1, n_2}_\beta}.
\end{equation}
Analogously, the $Y$ norms below capture the regularity of $\chi^\vp \wh{f}$ in the horizontal direction (in particular, in the vector field $\partial_\vr$):
\begin{equation} \label{eq:46-13}
	\|f\|_{Y^{n_1,n_2}_\beta} = \sup_{k,l\in\mathbb{Z}} \ \sup_{a+b\le n_2}  2^{n_1 k^+ +(1+\beta)(l+k)} \| H_l^{(k)} S^a \Omega^b P^\vp_k f\|_{L^2}, 
\end{equation}
\begin{equation} \label{eq:46-14}
	\|f\|_{Y^n_\beta} = \sup_{n_1+n_2 \le n} \|f\|_{Y^{n_1, n_2}_\beta}.
\end{equation}
Without surprise, for a vector-valued function $\f$, the $X, Y$ norms of $\f$ are defined as the sum of the norms of each component $f_i, i =1,2,3$ of $\f$.


\paragraph{Convention.} To be precise, the notation $\sup_{a+b \le n_2}$ stands for supremum taken over all $a, b \in \mathbb{N}$ such that $a+b \le n_2$. Sometimes, we also write $\{S, \Omega\}^{\le n}$ as abbreviation for the supremum over $S^{n_1} \Omega^{n_2}$ with $n_1 + n_2 \le n$.

\medskip

\medskip

Notice that the commutators $[S, Z_l]$ and $[S, P^{\hp/\vp}_k]$ have the similar support properties (on the physical or Fourier side) as $Z_l$ and $P^{\hp/\vp}_k$ respectively. Also, observe that $[\Omega, S] = [\Omega, Z_l] = [\Omega, P^{\hp/\vp}_k] = 0$. Hence, in the definitions \eqref{eq:46-11} and \eqref{eq:46-13}, one is allowed to change the order of $Z_l^{k}$ and $S^a\Omega^b$, or the order of $S^a\Omega^b$ and $P^{\hp/\vp}_k$, in the sense that the resulting norms will be equivalent to the original ones. Moreover, using the modified charts defined in Section \ref{sec:22}, we also have the following equivalent definitions:
\begin{equation} \label{eq:729-001}
	\|f\|_{X^{n_1,n_2}_\beta} \sim   \sup_{k \in \Z} 2^{n_1 k^+ + \frac32 k} \| \{\partial_{\r}, \partial_\th\}^{\le n_2} (\chi^\hp_0(\cdot) \wh{f}(2^{k} \cdot))\|_{L^2_{h}B^{1+\beta}_{2, \infty, \zeta}},
\end{equation}
\begin{equation} \label{eq:729-002}
	\|f\|_{Y^{n_1,n_2}_\beta} \sim   \sup_{k \in \Z} 2^{n_1 k^+ + \frac32 k} \| \{\partial_{z}, \partial_\th\}^{\le n_2} (\chi^\vp_0(\cdot) \wh{f}(2^{k} \cdot))\|_{L^2_{z}B^{1+\beta}_{2, \infty, h}},
\end{equation}
where $B^{1+\beta}_{2, \infty, \z}$ and $B^{1+\beta}_{2, \infty, h}$ stand for the usual inhomogeneous Besov norm in the $\z$ direction and in the horizontal directions respectively.


 Working in the modified charts and applying Sobolev embedding, it is easy to see that the $L^\infty$ norm of $\wh{f}$ is under control of the $X$ and $Y$ norms,
\begin{equation} \label{eq:56-1}
	\|\chi^\hp_k \wh{f}\|_{L^\infty} \lesssim 2^{-\frac32 k} \|f\|_{X^{0,2}_0}, \quad \| \chi^\vp_k \wh{f}\|_{L^\infty} \lesssim_\beta 2^{-\frac32 k} \| f\|_{Y^{0,2}_\beta} \ \ (\beta>0).
\end{equation}
In particular, \eqref{eq:56-1} implies that, for $p, q \le -2$,
\begin{equation} \label{eq:56-1001}
	\|\chi^{\hp,q}_k \wh{f}\|_{L^2} \lesssim 2^{\frac{q}{2}} \|f\|_{X^{0,2}_0}, \quad \| \chi^{\vp,p}_k \wh{f}\|_{L^2} \lesssim_\beta 2^{p} \| f\|_{Y^{0,2}_\beta} \ \ (\beta>0).
\end{equation}
It is also convenient to interpolate between the $X, Y$ norms and basic energy norms. There holds that
\begin{align} \label{eq:724-01}
	\|f\|_{X^{\wt{n_1}, \wt{n_2}}_{\wt{\beta}}} \lesssim \|f\|_{X^{{n_1}, {n_2}}_{{\beta}}}^{1-\vartheta} \|\{S, \Omega\}^{\le N_2} f \|_{H^{N_1}}^{\vartheta}, 	
\end{align}
and
\begin{align} \label{eq:724-02}
	\|f\|_{Y^{\wt{n_1}, \wt{n_2}}_{\wt{\beta}}} \lesssim \|f\|_{Y^{{n_1}, {n_2}}_{{\beta}}}^{1-\vartheta} \|\{S, \Omega\}^{\le N_2} f \|_{H^{N_1}}^{\vartheta} 
\end{align}
with $0 <\vartheta < 1$ and
\begin{align}
	\wt{n_1} = n_1 (1-\vartheta) + N_1 \vartheta, \quad \wt{n_2} = n_2 (1-\vartheta) + N_2 \vartheta, \quad \wt{\beta} = \beta - \vartheta - \beta \vartheta.
\end{align}
Another useful property is that
\begin{equation} \label{eq:59-b}
	\|P_k f\|_{X^0_\beta} + \|P_k f\|_{Y^0_\beta} \lesssim 2^{(1+\beta)k} \||x|^{1+\beta}f\|_{L^2},
\end{equation}
which follows from the potential estimate
\begin{equation}
	\|\langle|x|\rangle^{1+\beta} P_0 f\|_{L^2} \lesssim \||x|^{1+\beta} f\|_{L^2}
\end{equation}
as well as scaling.


\subsection{The key bootstrap argument} \label{sec:25-boot}

Theorem \ref{thm:main} is a direct corollary of the following bootstrap estimates, in which we propagate the $X$ and $Y$ norms for the profiles $\f_\pm$ globally in time.  We fix $\beta = \frac{2}{3} \beta_0$ and $\beta' = \frac23 \beta$ and will take $\beta_0>0$ sufficiently small.

\smallskip

\begin{proposition} \label{prop:BA}
	Let $u \in C([0,T], C^2(\mathbb{R}^3; \R^3))$ be a solution to \eqref{eq:EC} for some $T>0$ with dispersive profiles $\f_\pm$ defined in \eqref{eq:58-01}, and suppose that the initial data verifies the assumption of Theorem \ref{thm:main}. 	If for all $t \in [0,T]$ there holds that
	\begin{equation} \label{eq:BA}
		\|\f_\pm(t)\|_{X^{10}_\beta} + \|\f_\pm(t)\|_{X^{20,20}_{\beta'}} + \| \f_\pm(t)\|_{{Y}^{20, 20}_{\beta}} \le \varepsilon_1
	\end{equation}
	for some $\varepsilon_1 \le \varepsilon^\frac23$, then in fact we have the improved bounds
	\begin{equation} \label{eq:BA-improv}
		\|\f_\pm(t)\|_{X^{10}_\beta} + \|\f_\pm(t)\|_{X^{20,20}_{\beta'}}  + \| \f_\pm(t)\|_{{Y}^{20, 20}_{\beta}} \lesssim \varepsilon + \varepsilon_1^\frac32,
	\end{equation}
	and for some universal constant $C>0$ there holds that
	\begin{equation} \label{eq:slow-growth}
		\|\f_\pm(t)\|_{H^{M_1}} + \sum_{a+b \le M_2} \|S^a \bar{\Omega}^b \f_\pm(t)\|_{L^2} \lesssim \varepsilon \langle t \rangle^{C \varepsilon_1}.
	\end{equation}
\end{proposition}

\begin{proof}  [Proof of Theorem \ref{thm:main} via Proposition \ref{prop:BA}]
	By the energy estimates  \eqref{eq:16-c}--\eqref{eq:16-3a} as well as \eqref{eq:59-c}, the norms in \eqref{eq:59-a} can be propagated at least locally in time. Due to \eqref{eq:59-b} and interpolation \eqref{eq:724-01}--\eqref{eq:724-02}, for some $T>0$, there holds that
	\begin{equation} \label{eq:59-b1}
		\|\ue(t)\|_{X^{10}_\beta} + \|\ue(t)\|_{X^{20,20}_{\beta'}} + \| \ue(t)\|_{{Y}^{20, 20}_{\beta}} \lesssim \varepsilon, \quad t \in [0,T].
	\end{equation}
	Hence, \eqref{eq:BA} is satisfied locally in time. Using Proposition \ref{prop:BA}, we can then extend the solution globally in time. The  scattering result directly follows from the fast decay of $\partial_t \f_\pm$ in $L^2$, see Lemma \ref{lem:528-01}.
\end{proof}


\medskip

Hence, proving Proposition \ref{prop:BA} will be the central goal for the rest of the paper. Next, we give an overview of the strategy and reduce the proof of Proposition \ref{prop:BA} to certain bilinear and trilinear estimates in Sections \ref{sec:dt}--\ref{sec:9}.

\begin{proof}[Outlined proof of Proposition \ref{prop:BA}]
	Under the bootstrap assumption \eqref{eq:BA}, it follows from the linear dispersive estimate in Corollary \ref{cor:59-1} that 
\begin{equation}
	\|\{1, S, \bar{\Omega}, \nabla\} \ue\|_{L^\infty} \lesssim \varepsilon_1 \langle t \rangle^{-1}.
\end{equation}
Hence, the slow growth of the energies in \eqref{eq:slow-growth} follows from a standard energy estimate for the system \eqref{eq:EC}, see Lemma \ref{lem:ee1} below. By interpolation, there also holds that
\begin{equation}
 	\|\{S, \bar{\Omega}\}^{\le \frac{M_2}{2}} \f\|_{H^{\frac{M_1}{2}}} =	\|\{S, \bar{\Omega}\}^{\le \frac{M_2}{2}} \ue\|_{H^{\frac{M_1}{2}}} \lesssim \varepsilon \langle t \rangle^{C \varepsilon_1}.
\end{equation}
As we have mentioned in \eqref{eq:59-b1},  \eqref{eq:BA-improv} is satisfied at least locally in time. Hence, to prove Proposition \ref{prop:BA}, it suffices to establish \eqref{eq:BA-improv} for $t \gg 1$, and we proceed as follows.

\smallskip

Write the nonlinear evolution  of the profiles \eqref{eq:58-2} more compactly as
\begin{equation}
	\partial_t \f_\mu = \sum_{\mu_1, \mu_2 = \pm} \mathcal{N}_\mua (\f_{\mu_1}, \f_{\mu_2}),
\end{equation}
where we denote $\mua = (\mu, \mu_1, \mu_2) \in \{\pm\}^3$ and
\begin{equation} \label{eq:510-51}
	\mathcal{N}_\mua(\f_{\mu_1}, \f_{\mu_2}) = - e^{-\mu i t \Lmd} \mathbb{P}_\mu \PL \nabla \cdot (e^{\mu_1 i t \Lmd} \f_{\mu_1} \otimes e^{\mu_2 i t \Lmd} \f_{\mu_2}).
\end{equation}
Applying a number of good derivatives $S^a \bar{\Omega}^b$, we get
\begin{equation} \label{eq:510-1}
	\partial_t S^a \bar{\Omega}^b \f_\mu = \sum_{\substack{a_1 + a_2 \le a \\ b_1 + b_2 \le b}} C_{\bar{a}, \bar{b}} \sum_{\mu_1, \mu_2 = \pm} \mathcal{N}_\mua (S^{a_1} \bar{\Omega}^{b_1} \f_{\mu_1}, S^{a_2} \bar{\Omega}^{b_2} \f_{\mu_2}), 
\end{equation}
where $\bar{a} = (a, a_1, a_2), \bar{b} = (b, b_1, b_2)$, and $C_{\bar{a}, \bar{b}}$ are constants symmetric in $a_1, a_2$ and in $b_1, b_2$. For simplicity, we shall denote
\begin{equation} \label{eq:807-y1}
	{\sum}^* = \sum_{\substack{a_1 + a_2 \le a \\ b_1 + b_2 \le b}} C_{\bar{a}, \bar{b}} \sum_{\mu_1, \mu_2 = \pm}.
\end{equation}
 Then, we decompose the inputs and the output into their $\bighp$ and $\bigvp$ parts, and apply different strategies for different types of interactions.

\medskip
\underline{\emph{Horizontal frequency output}}
\smallskip

For the $\bighp$-part of the output, we prefer to work with the scalar profiles $\Fp_\pm$ as in \cite{GuoInvent}.\footnote{It is also possible to work only with vectorial profiles. In that case, one has to keep track of the vectorial structure of the nonlinearity throughout the proof.}  Recall that
\begin{equation} \label{eq:808-i11}
	\Fp_\pm = R_\pm \f_\pm = \frac12 \, \vec{e}_3 \cdot (|\nabla_h|^{-1} \nabla \times \f_{\pm} \pm |\nabla| |\nabla_h|^{-1}  \f_\pm).
\end{equation}
Note that the vector-to-scalar operators $R_\pm$ commute with Littlewood-Paley type projections, and moreover,
\begin{equation}
	R_\pm S^a \bar{\Omega}^b \f_\mu = S^a {\Omega}^b \Fp_\pm.
\end{equation}
We emphasize that the $X$ norms of $\f_\pm$ and $\Fp_\pm$ are equivalent since the symbol of $R_\pm$ is smooth on the support of $\chi^\hp$. Acting $Z_l^{(k)} P_k^\hp R_\mu$ on \eqref{eq:510-1} and decomposing the inputs into $\bighp$ and $\bigvp$-parts, we receive that
\begin{align} 
	\partial_t Z_l^{(k)} P^\hp_k S^a {\Omega}^b \Fp_\mu &= 	{\sum}^* Z_l^{(k)} P^\hp_k R_\mu \mathcal{N}_\mua ( S^{a_1} \bar{\Omega}^{b_1} \f_{\mu_1},  S^{a_2} \bar{\Omega}^{b_2} \f_{\mu_2}) \label{eq:510-3-001} \\
	&= 	{\sum}^* \sum_{\iota_1, \iota_2 \in \{\hp, \vp\}} Z_l^{(k)} P^\hp_k R_\mu \mathcal{N}_\mua (  P^{\iota_1} S^{a_1} \bar{\Omega}^{b_1} \f_{\mu_1},  P^{\iota_2} S^{a_2} \bar{\Omega}^{b_2}  \f_{\mu_2}), \label{eq:510-3}
\end{align}
for $\mu \in \{\pm\}$. Using Proposition \ref{prop:nonlinear}, the term in \eqref{eq:510-3} with $\iota_1 = \iota_2 = \hp$ can be written as
\begin{align} \label{eq:807-000}
		{\sum}^* Z_l^{(k)} P^\hp_k \Q_\mua [\m_\mua] (P^\hp S^{a_1} \Omega^{b_1} \Fp_{\mu_1}, P^\hp S^{a_2} \Omega^{b_2} \Fp_{\mu_2}),
\end{align}
where the family of scalar bilinear multipliers $\m_{\mua}(\xi, \eta)$ exhibit the null type structure mentioned in Section \ref{sec:12} as well as the energy structure. 

\smallskip

Let $\mathcal{I}_m = [2^m, 2^{m+1}) \cap [1,t]$ with $m \ge 1$ and $t \in [1, T]$. To propagate the $X^{10}_\beta$ norm of $\Fp_\mu$ on the time interval $\mathcal{I}_m$, we assume that $a+b \le 10$ and  consider the following three cases  depending on the size of $l+k$. The propagation of the auxilliary $X^{20,20}_{\beta'}$ norm is based on a modification of the methods below, see Section \ref{sec:Xbetap}.

\medskip
\emph{Case 1.}  $l+k \ge (1+\delta)m$.
\smallskip

In this case, we  use  a dispersive estimate at far field (Lemma \ref{lem:largel}) to obtain fast decay in $L^2$ norm for 
$$Z_{\sim l}\{(e^{\mu_1 i t \Lmd} P_{k_1} S^{a_1} {\bar \Omega}^{b_1} \f_{\mu_1}) \otimes (e^{\mu_2 i t \Lmd} P_{k_2} S^{a_2} {\bar \Omega}^{b_2} \f_{\mu_2})\}.$$
Then, by \eqref{eq:510-3-001} and the commutator type estimate in Lemma \ref{lem:a1-1}, we deduce an integrable decay rate for  
$$2^{30k^+ + (1+\beta)(l+k)}\|\partial_t Z_l P^\hp_k S^a \Omega^b \Fp_\mu\|_{L^2}.$$
Moreover, here we can relax the assumption $a+b \le 10$ to $a+b \le 20$, hence the result also applies to the propagation of the $X^{20,20}_{\beta'}$ norm.  This is carried out in detail in Lemma \ref{lem:428-1}. The method here avoids the technical difficulties in \cite[Section 8.1]{GuoInvent}, essentially due to our simpler choice of norms.

\medskip
\emph{Case 2.}  $l+k \le m+4$.
\smallskip

In this case, we propagate the $X^{10}_\beta$ norm of $\Fp_\mu$ based on Duhamel's formula, \emph{i.e.}, 
\begin{align} \label{eq:510-12}
	\left\|Z_l^{(k)} P^\hp_k S^a {\Omega}^b \Fp_\mu\right\|_{L^2}(t_2) \le \left\|Z_l^{(k)} P^\hp_k S^a {\Omega}^b \Fp_\mu\right\|_{L^2}(t_1) 	+ \left\|\int_{t_1}^{t_2} \partial_s Z_l^{(k)} P^\hp_k S^a {\Omega}^b \Fp_\mu ds \right\|_{L^2}.
\end{align}
Using \eqref{eq:510-3} and \eqref{eq:807-000}, we separate the analysis according to the choice of $(\iota_1, \iota_2) \in \{\bighp, \bigvp\}^2$.

\medskip
\emph{Subcase 2.1.}  $\iota_1=\iota_2=\bighp$.
\smallskip

The proof of this case is carried out in detail in Section \ref{sec:8}. We will bound the $L^2$ norm of the bilinear form
\begin{equation} \label{eq:807-901}
	\int_{\mathcal{I}_m} {\sum}^* Z_l^{(k)} P^\hp_k \Q_\mua [\m_\mua] (P^\hp S^{a_1} \Omega^{b_1} \Fp_{\mu_1}, P^\hp S^{a_2} \Omega^{b_2} \Fp_{\mu_2}) \, ds.
\end{equation}
by
\begin{equation} \label{eq:807-o1}
	2^{-(10-a-b) k^+ -(1+\beta)m} \,  \varepsilon_1^2.
\end{equation}
Note that due to $m \ge l+k-4$, \eqref{eq:807-o1} gives an acceptable contribution  to the propagation of the $X^{10}_\beta$  norm of $\Fp_\mu$. Our strategy to prove \eqref{eq:807-901}--\eqref{eq:807-o1} is as follows. 

Since $Z_l^{(k)}$ is bounded from $L^2$ to $L^2$, one can simply ignore it in \eqref{eq:807-901}.  Further localizing the inputs into $P^{\hp}_{k_1}$ and $P^{\hp}_{k_2}$ pieces, we study the  high-low ($k_1 \gg k_2$), low-high ($k_1 \ll k_2$) and high-high ($k_1 \sim k_2$) interactions respectively. 

Concerning the high-low and the low-high interactions, we observe that the group velocities in the $x_3$-direction for the two inputs are different. Based on the precise linear asymptotics in Lemma \ref{lem:slab}, we show that the $L^2$ norm of $\Q_\mua [\m_\mua] (P^\hp_{k_1} S^{a_1} \Omega^{b_1} \Fp_{\mu_1}, P^\hp_{k_2} S^{a_2} \Omega^{b_2} \Fp_{\mu_2})$ decays with rate $2^{-(2+\beta)m}$. 

The high-high interactions with $\mu_1 = - \mu_2$ can be treated similarly as in the last paragraph. However, the case $\mu_1 = \mu_2$ is more involved as a direct $L^2$-$L^\infty$ or $L^\infty$-$L^2$ bound is no longer sufficient. We further decompose the inputs into two parts using
$$\mbox{Id}=P^{\hp, \le -m\alpha} + (\mbox{Id}-P^{\hp, \le -m\alpha}),$$
with $\alpha = 0.35$ (see \eqref{eq:614-aa1} for the choice of $\alpha$).
For the input frequencies with $|\Lmd|\gtrsim 2^{-m \alpha}$, we are away from the resonant set \eqref{eq:524-3}, hence a spacetime resonances approach based on the formalism of \cite{GuoInvent} can be applied. If both inputs are supported on frequencies with $|\Lmd| \lesssim 2^{-m\alpha}$, we have to exploit the null type structure of $\m_\mua$ using the so-called multiscale decomposition and the precised wave packet estimates in Section \ref{sec:52}. See Section \ref{sec:high-high} for the detailed arguments.

\medskip
\emph{Subcase 2.2.}  $(\iota_1, \iota_2)  \neq  (\bighp, \bighp)$.
\smallskip

These types of interactions are treated in Sections \ref{sec:93}-\ref{sec:84} using the spacetime resonances approach of \cite{GuoInvent}. We will bound the $L^2$ norm of
\begin{equation}
	 \int_{\mathcal{I}_m}	 Z_l^{(k)} P^\hp_k R_\mu \mathcal{N}_\mua ( P^{\iota_1} S^{a_1} \bar{\Omega}^{b_1}  \f_{\mu_1}, P^{\iota_2} S^{a_2} \bar{\Omega}^{b_2}  \f_{\mu_2}) ds,
\end{equation}
by
\begin{equation} \label{eq:807-u1}
	2^{-20 k^+ -(1+\beta)(l+k) - \delta m} \varepsilon_1^2.
\end{equation}
In fact, the stronger bilinear $L^2$ estimate
\begin{equation} \label{eq:809-ab1}
	\left\|\int_{\mathcal{I}_m}	  P^\hp_k R_\mu \mathcal{N}_\mua ( P^{\iota_1} S^{a_1} \bar{\Omega}^{b_1}  \f_{\mu_1}, P^{\iota_2} S^{a_2} \bar{\Omega}^{b_2}  \f_{\mu_2}) ds \right\|_{L^2} \lesssim 2^{-20 k^+  - (1+\beta+3\delta) m} \varepsilon_1^2.
\end{equation}
will be established. Moreover, here we can relax the assumption $a+b \le 10$ to $a+b \le 20$, hence \eqref{eq:807-u1} also gives an acceptable contribution to the $X^{20,20}_{\beta'}$ norm.

\medskip
\emph{Case 3.}  $m+5 \le l+k <(1+\delta) m.$
\smallskip

In this case, we propagate the $X^{10}_\beta$ norm of $\Fp_+$ and $\Fp_-$ together based on energy method, \emph{i.e.}, 
\begin{align} \label{eq:510-11}
	&\quad \sum_{\mu=\pm} \frac12  \|Z_l^{(k)} P^\hp_k S^a {\Omega}^b \Fp_\mu\|^2_{L^2}(t_2) - \sum_{\mu=\pm} \frac12 \|Z_l^{(k)} P^\hp_k S^a {\Omega}^b \Fp_\mu\|^2_{L^2}(t_1) \nonumber \\
	&\qquad = \int_{t_1}^{t_2} \sum_{\mu=\pm} \langle \partial_s Z_l^{(k)} P^\hp_k S^a {\Omega}^b \Fp_\mu,  Z_l^{(k)} P^\hp_k S^a {\Omega}^b \Fp_\mu \rangle_{L^2} ds.
\end{align}
As in Case 2, we separate the analysis according to the choice of $(\iota_1, \iota_2) \in \{\bighp, \bigvp\}^2$.

\medskip
\emph{Subcase 3.1.}  $\iota_1=\iota_2=\bighp$.
\smallskip

In Section \ref{sec:8}, we will bound the trilinear form
\begin{equation} \label{eq:807-9011}
	\int_{\mathcal{I}_m} \sum_{\mu=\pm} {\sum}^* \langle Z_l  P^\hp_k R_\mu \mathcal{N}_\mua (P^{\iota_1} S^{a_1} \bar{\Omega}^{b_1} \f_{\mu_1}, P^{\iota_2} S^{a_2} \bar{\Omega}^{b_2} \f_{\mu_2}), Z_l P_k^\hp S^a \Omega^b \Fp_\mu \rangle_{L^2} \, ds.
\end{equation}
by
\begin{equation}\label{eq:809-yyy1}
	2^{-2(10-a-b)k^+ - 2(1+\beta)(l+k)} \left(2^{-\delta m} + 2^{- |m-l| } + \mathbf{1}_{l+k \le m+15}\right) \,  \varepsilon_1^3.
\end{equation}
Similar to Subcase 2.1, we  further localize the inputs into $P^{\hp}_{k_1}$ and $P^{\hp}_{k_2}$ pieces and employ different methods for the  high-low, low-high and high-high  interactions. 

For the high-low and low-high interactions, using the commutator type estimate in Lemma \ref{lem:a1-1}, we are able to insert the operator $Z_{\sim l}$ after the multipliers $P_k^\hp R_\mu e^{-\mu i s \Lmd} \mathbb{P}_\mu \PL$ within \eqref{eq:807-9011} up to an admissible remainder.   Then, we use $L^2$-$L^\infty$ or $L^\infty$-$L^2$ bounds depending on the parameters of localization (\emph{i.e.}, $l, m, k, k_1, k_2$) and invoke dispersive bounds from Lemma \ref{lem:slab}. For the low-high interactions, there is an additional difficulty caused by the quasilinear nature of \eqref{eq:EC}. To tackle this, we use the energy structure of $\m_\mua$ and commutator estimates to exploit cancellations in the trilinear form. 

For the high-high interactions, we distinguish the following two cases. If $l+k \ge m+16$, we can apply far field linear estimates to deduce an integrable bound. If $l+k \le m+15$, then we use the same method as in Subcase 2.1 based on the multiscale decomposition.

\medskip
\emph{Subcase 3.2.}  $(\iota_1, \iota_2)  \neq  (\bighp, \bighp)$.
\smallskip

This is similar to Subcase 2.2, also treated in Sections \ref{sec:93}-\ref{sec:84} using the spacetime resonances approach. We can also relax the assumption $a+b \le 10$ to $a+b \le 20$, and the resulting estimate reads that 
\begin{equation}
	|\eqref{eq:807-9011}| \lesssim 2^{-40 k^+ - (2+2\beta+5\delta)m} \varepsilon_1^3  \lesssim 2^{-40 k^+ - 2(1+\beta)(l+k) - \delta m} \varepsilon_1^3.
\end{equation}



\medskip
\underline{\emph{Vertical frequency output}}
\smallskip

For the $\bigvp$-part of the output, it is more convenient to work with the vectorial profiles $\f_\pm$. Acting $H_l^{(k)} P^{\vp}_k$ on \eqref{eq:510-1}, we obtain that
\begin{equation} \label{eq:59-c1}
	\partial_t H_l^{(k)} P^\vp_k S^a \bar{\Omega}^b \f_\mu ={\sum}^* H_l^{(k)} P^\vp_k \mathcal{N}_\mua ( S^{a_1} \bar{\Omega}^{b_1} \f_{\mu_1},  S^{a_2} \bar{\Omega}^{b_2} \f_{\mu_2}).
\end{equation}  
Assume that $a+b \le 20$. We rely on Duhamel's formula to propagate the $Y^{20,20}_{\beta}$ norm of $\f_\mu$.

\medskip
\emph{Case 1.}  $l+k \ge (1+\delta) m.$
\smallskip

This is similar to Case 1 for horizontal frequency output. By the same method, we deduce an integrable decay rate for  
$$2^{30k^+ + (1+\beta)(l+k)}\|\partial_t H_l P^\vp_k S^a \bar{\Omega}^b \f_\mu\|_{L^2},$$
see Lemma \ref{lem:428-1}.

\medskip
\emph{Case 2.}  $l+k < (1+\delta) m.$
\smallskip

This case is treated in Sections \ref{sec:8-1}, \ref{sec:9-2} and \ref{sec:84}. Following the spacetime resonances approach, we will bound the $L^2$ norm of
\begin{equation}
	\int_{\mathcal{I}_m} P^\vp_k \mathcal{N}_\mua ( S^{a_1} \bar{\Omega}^{b_1} \f_{\mu_1},  S^{a_2} \bar{\Omega}^{b_2} \f_{\mu_2}) ds
\end{equation}
by the right hand side of  \eqref{eq:809-ab1}, hence also by \eqref{eq:807-u1}.

\medskip
In conclusion, the bootstrap norms in \eqref{eq:BA} can be propagated up to time $T$, satisfying the improved estimate \eqref{eq:BA-improv}.

\end{proof}

\section{Structure of the nonlinearity} \label{sec:3}

\subsection{Basic energy estimates}

\begin{lemma} \label{lem:ee}
	Assume that $\ue$ solves \eqref{eq:EC} on $0 \leqslant t \leqslant T$, then
	for any given $n \in \mathbb{N}$ there holds that
	\begin{equation} \label{eq:16-c}
		\frac{d}{d t} \| \ue \|_{H^n}^2 \lesssim \| \nabla \ue \|_{L^{\infty}} \| \ue
		\|_{H^n}^2, 
	\end{equation}
	\begin{equation} \label{eq:16-d}
		\frac{d}{d t} \| S^n \ue \|_{L^2}^2 \lesssim (\| S \ue \|_{L^{\infty}} + \|
		\nabla \ue \|_{L^{\infty}}) \left( \| \ue \|_{H^n} + \sum_{a = 1}^n \| S^a \ue
		\|_{L^2} \right)^2, 
	\end{equation}
	\begin{equation}\label{eq:16-3a}
		\frac{d}{d t} \| \bar{\Omega}^n \ue \|_{L^2}^2 \lesssim (\| \bar{\Omega} \ue
		\|_{L^{\infty}} + \| \nabla \ue \|_{L^{\infty}}) \left( \| \ue \|_{H^n} +
		\sum_{a = 1}^n \| \bar{\Omega}^a \ue \|_{L^2} \right)^2. 
	\end{equation}
\end{lemma}

\begin{proof}
	These estimates extend {\cite[Proposition 5.1]{GuoCpam}} to the
	non-axisymmetric case. The estimate \eqref{eq:16-c} is classical. To prove
	\eqref{eq:16-d}, we apply $S^n = (x \cdot \nabla)^n$ to the system \eqref{eq:EC} and get
	\begin{equation} \label{eq:524-1}
		\partial_t S^n \ue + S^n (\ue \cdot \nabla \ue) + \vec{e}_3 \times S^n \ue + \nabla (S-1)^n p =
		0.
	\end{equation} 
	Testing \eqref{eq:524-1} by $S^n \ue$ gives
	\[ \frac{1}{2}  \frac{d}{d t} \| S^n \ue \|_{L^2}^2 = - \int_{\R^3} S^n (\ue \cdot
	\nabla \ue) \cdot S^n \ue \]
	\[ = -\sum_{k = 0}^n  \big({\large \substack{n \\ k}}\big)   \int_{\R^3} (S^k \ue \cdot \nabla) S^{n - k} \ue \cdot S^n \ue d x
	+ l.o.t. \]
	Here, $l.o.t.$ stands for lower order terms coming from the commutator
	$[\nabla, S] = \nabla$, which are easier to estimate. For $k = 0$ or $n$,
	the estimates are straightforward:
	\[ \int_{\R^3} (\ue \cdot \nabla) S^n \ue \cdot S^n \ue dx = 0, \]
	\[ \left| \int_{\R^3} (S^n \ue \cdot \nabla) \ue \cdot S^n \ue d x \right| \lesssim \|
	\nabla \ue \|_{L^{\infty}} \| S^n \ue \|_{L^2}^2 . \]
	For $2 \leqslant k \leqslant n - 1$, using integration by parts there holds
	that
	\[ - \int_{\R^3} (S^k \ue \cdot \nabla S^{n - k} \ue) \cdot S^n \ue d x = \int_{\R^3} (S^k \ue
	\cdot \nabla S^n \ue) \cdot S^{n - k} \ue d x \]
	\begin{equation}  \label{eq:524-2}
		\quad = - \int_{\R^3} (S^{k + 1} \ue \cdot \nabla S^{n - 1} \ue) \cdot S^{n - k} \ue d x -
	\int_{\R^3} (S^k \ue \cdot \nabla S^{n - 1} \ue) \cdot S^{n - k + 1} \ue d x + l.o.t.
	\end{equation}
	The first integral in \eqref{eq:524-2} can be estimated using Holder's inequality and the interpolation estimates
	\eqref{eq:16-a} and \eqref{eq:16-2b},
	\[ \left| \int_{\R^3} (S^{k + 1} \ue \cdot \nabla S^{n - 1} \ue) \cdot S^{n - k} \ue d x
	\right| \lesssim \| S^{k + 1} \ue \|_{L^{\frac{2 (n - 1)}{k}}}  \| S^{n -
		k} \ue \|_{L^{\frac{2 (n - 1)}{n - k - 1}}} \| \nabla S^{n - 1} \ue \|_{L^2}
	\]
	\[ \lesssim \| S \ue \|_{L^{\infty}}  \left( \sum_{a = 1}^n \| S^a \ue \|_{L^2}
	\right) \left( \| \ue \|_{H^n} + \sum_{a = 0}^n \| S^a \ue \|_{L^2} \right)
	\]
	and second integral in \eqref{eq:524-2} can be estimated in a similar way. This finishes the
	proof of \eqref{eq:16-d}.
	
	To show \eqref{eq:16-3a}, we first observe that for $\ue$ divergence free, $\bar{\Omega} \ue$ is also
	divergence free. Moreover, there hold that
	\[ \bar{\Omega} \nabla p = \nabla \Omega p, \]
	\[ \bar{\Omega} (\vec{e}_3 \times \ue) = \vec{e}_3 \times \bar{\Omega} \ue, \]
	and
	\[ \bar{\Omega} (\ue \cdot \nabla \ue) = \bar{\Omega} \ue \cdot \nabla \ue + \ue \cdot
	\nabla \bar{\Omega} \ue. \]
	Applying a number of $\bar{\Omega}$ operators to the Euler-Coriolis system
	gives
	\[ \partial_t \bar{\Omega}^n \ue + \sum_{k = 0}^n \big({\large \substack{n \\ k}}\big)    \, \bar{\Omega}^k \ue \cdot
	\nabla \bar{\Omega}^{n - k} \ue + \vec{e}_3 \times \bar{\Omega}^n \ue + \nabla
	\Omega^n p = 0. \]
	Testing by $\bar{\Omega}^n \ue$ gives
	\[ \frac{1}{2}  \frac{d}{{dt}} \| \bar{\Omega}^n \ue \|_{L^2}^2 = -
	\sum_{k = 0}^n  \big(\substack{n \\ k}\big) \, \int_{\R^3} (\bar{\Omega}^k \ue \cdot \nabla) \bar{\Omega}^{n - k}
	\ue \cdot \bar{\Omega}^n \ue d x \]
	Then, using the interpolation estimates \eqref{eq:16-b} and
	\eqref{eq:16-3b}, we can proceed using the same argument as above for $S$ derivatives.
	
\end{proof}

\begin{lemma} \label{lem:ee1}
	Under the bootstrap assumption \eqref{eq:BA}, there holds that
	\begin{equation} \label{eq:46-1}
		\| \ue (t) \|_{H^{M_1}} + \sum_{a + b \leqslant M_2}
		\| S^a \bar{\Omega}^b \ue(t) \|_{L^2} \lesssim \varepsilon
		\langle t \rangle^{C \varepsilon_1}.
	\end{equation}
	As a consequence, we also have
	\begin{equation}  \label{eq:59-02}
		\| \f_\pm (t) \|_{H^{M_1}} + \sum_{a + b \leqslant M_2}
		\| S^a \bar{\Omega}^b \f_\pm (t) \|_{L^2} \lesssim \varepsilon
		\langle t \rangle^{C \varepsilon_1},
	\end{equation}
	and
	\begin{equation} \label{eq:59-03}
		\| \Fp_\pm (t) \|_{H^{M_1}} + \sum_{a + b \leqslant M_2}
		\| S^a \Omega^b \Fp_\pm (t) \|_{L^2} \lesssim \varepsilon
		\langle t \rangle^{C \varepsilon_1}.
	\end{equation}
\end{lemma}

\begin{proof}
	Using \eqref{eq:BA} and the dispersive estimate in Corollary \ref{cor:59-1} below, we have
	\begin{equation}  \label{eq:59-01}
		\| \ue \|_{L^{\infty}} + \| \nabla \ue \|_{L^{\infty}} + \| S
		\ue \|_{L^{\infty}} + \| \bar{\Omega} \ue \|_{L^{\infty}} \lesssim
		\frac{\varepsilon_1}{\langle t \rangle}.
	\end{equation}
	By the energy estimates \eqref{eq:16-c}--\eqref{eq:16-3a}, \eqref{eq:59-01} and Gronwall's inequality, we obtain
	\begin{equation} \label{eq:46-2}
		\| \ue (t) \|_{H^{M_1}} + \sum_{a \le M_2}
		\| S^a \ue \|_{L^2} + \sum_{b \le M_2}
		\| \bar{\Omega}^b \ue \|_{L^2} \lesssim \varepsilon
		\langle t \rangle^{C \varepsilon_1}.
	\end{equation}
	Then \eqref{eq:46-1} follows from \eqref{eq:46-2} and the interpolation estimate \eqref{eq:26aa}. \eqref{eq:59-02}--\eqref{eq:59-03} follows from \eqref{eq:46-1}, the definition of $\f_\pm, \Fp_\pm$ as well as the fact that $[S, \mathbb{P}_\pm] = [\bar{\Omega}, \mathbb{P}_\pm] = [e^{i t \Lmd}, \mathbb{P}_\pm] = 0$.
\end{proof}

%

\medskip

\begin{lemma} \label{lem:wee}
	Assume that $\ue$ solves \eqref{eq:EC} on $0 \leqslant t \leqslant T$, then there holds that
	\begin{equation} \label{eq:59-c}
		\frac{d}{d t} \| \langle|x|\rangle^{1+\beta} \ue \|_{L^2}^2 \lesssim (1 + \|\ue\|_{L^\infty}) \| \langle |x| \rangle^{1+\beta} \ue \|_{L^2}^2. 
	\end{equation}
\end{lemma}
\begin{proof}
	Testing \eqref{eq:EC} by $\langle|x| \rangle^{2+2\beta}\ue$, we get
	\begin{equation}
		\frac12 \frac{d}{dt} \|\langle|x| \rangle^{1+\beta} \ue\|_{L^2}^2 =  \frac12 \int_{\R^3} |\ue|^2 \ue \cdot \nabla  \langle|x| \rangle^{2+2\beta} dx + \int_{\R^3} p \,\ue \cdot \nabla  \langle|x| \rangle^{2+2\beta} dx
	\end{equation}
	For the first term, we have
	\begin{align}
		\int_{\R^3} |\ue|^2 \ue \cdot \nabla  \langle|x| \rangle^{2+2\beta} dx \lesssim \|\ue\|_{L^\infty} \|\langle|x| \rangle^{1+\beta} \ue\|_{L^2}^2.
	\end{align} 
	Using the potential estimate
	\begin{equation}
		\|\langle|x| \rangle^{\beta} |\nabla|^{-1} f\|_{L^2} \lesssim \|\langle|x| \rangle^{1+\beta}  f\|_{L^2}
	\end{equation}
	and the fact that $\langle|x| \rangle^{\beta}$ is an $A_2$ weight (cf. \cite{Grafa}), we  deduce that
	\begin{align}
		\| \langle|x| \rangle^{\beta}  p\|_{L^2} & \le  \| \langle|x| \rangle^{\beta}  \Delta^{-1}(\omega_3 )\|_{L^2} + \| \langle|x| \rangle^{\beta}  \Delta^{-1} \nabla^2:( \ue \otimes \ue)\|_{L^2} \nonumber\\
		&\lesssim  \| \langle|x| \rangle^{1+\beta}  \ue\|_{L^2} + \| \langle|x| \rangle^{\beta}   |\ue|^2 \|_{L^2}
	\end{align}
	Hence, for the pressure term we have
	\begin{align}
		\int p \,\ue \cdot \nabla  \langle|x| \rangle^{2+2\beta} dx &\le \|\langle|x| \rangle^{1+\beta} \ue\|_{L^2} \| \langle|x| \rangle^{\beta}  p\|_{L^2} \nonumber\\
		&\lesssim (1 + \|\ue\|_{L^\infty}) \| \langle |x| \rangle^{1+\beta} \ue \|_{L^2}^2.
	\end{align}
	This finishes the proof of \eqref{eq:59-c}.
\end{proof}

\subsection{Null type structures} \label{sec:null}

In this section, we further study the special nonlinear structure of \eqref{eq:EC}. In particular, we demonstrate the null type structures at the level of bilinear multipliers, which will be important for the bilinear/trilinear estimates in Section \ref{sec:8}.  

The basic form of the nonlinear contribution to the vectorial profiles $\f_\pm$ is summarized in the following result.

\begin{proposition} \label{prop:510}
	Let $\mu \in \{\pm\}$ and  $f_\mu$ be any component of the vectorial profile $\f_\mu$ given by \eqref{eq:58-01}, then $\partial_t f_\mu$ can be linearly decomposed into bilinear terms of the form
	\begin{align} \label{eq:57-11}
		  \Q_{\mua}[ \m](f_{\mu_1}, f_{\mu_2}).
	\end{align}
	Here, $\mua = (\mu, \mu_1, \mu_2) \in \{\pm\}^3$ and for $ i=1,2$,  $f_{\mu_i}$ is some component of $\f_{\mu_i}$. Moreover, $\m = |\xi| \mathfrak{n}(\xi)$ and $\mathfrak{n}$ stands for some 0th order homogeneous multiplier which is (multiplicatively) generated by 
	\begin{equation}\label{eq:57-12}
		R_0 = \left\{\frac{\xi_1}{|\xi|}, \frac{\xi_2}{|\xi|}, \frac{\xi_3}{|\xi|} \right\}.
	\end{equation}
\end{proposition}

\begin{proof}
	The claim follows directly from the definition of $\mathcal{N}_\mua$ in \eqref{eq:510-51}, the definition of $\Q_\mua$ in \eqref{eq:510-52} and the fact that both $\mathbb{P}_\pm, \PL$ have symbols generated by $R_0$. 
\end{proof}

\begin{remark}
	 For $\vp + \vp/\hp \to \vp$ interactions there is a null type structure which was observed and used in \cite{GuoInvent} but plays no role in our analysis. We document it here for possible future applications. In the nonlinear term $\partial_i (u_i u_j)$ with $i = 3$, we can rewrite the first input using the divergence free condition as
	 \begin{align}
	 	 \wh{u_3} (\xi-\eta) &= \frac{(\xi-\eta)_1}{(\xi-\eta)_3} \wh{u_1}(\xi-\eta) + \frac{(\xi-\eta)_2}{(\xi-\eta)_3} \wh{u_2}(\xi-\eta).
	 \end{align}
	 As a consequence, each component of	$P^\vp \mathcal{N}_\mua(P^\vp \f_{\mu_1},  \f_{\mu_2})$ can be linearly decomposed into bilinear terms of the form
	\begin{align}
		P^\vp \Q_\mua[|\xi| \mathfrak{n}'](P^\vp f_{\mu_1},  f_{\mu_2}), 
	\end{align}
	where $\mathfrak{n}'$ is (multiplicatively) generated by 
	\begin{equation}
		R_0 \cup \left\{\frac{(\xi-\eta)_1}{(\xi-\eta)_3}, \frac{(\xi-\eta)_2}{(\xi-\eta)_3}\right\}
	\end{equation}
	and always contains a factor of $\frac{\xi_i}{|\xi|}$ or  $\frac{(\xi-\eta)_i}{(\xi-\eta)_3}$  with $i \in \{1,2\}$. In particular, on the support of $\chi^\vp(\xi)\chi^\vp(\xi-\eta)$, there holds either $|\mathfrak{n}'| \lesssim \sqrt{1-\Lmd^2(\xi)}$ or $|\mathfrak{n}'| \lesssim \sqrt{1-\Lmd^2(\xi-\eta)}$. This is helpful in \cite{GuoInvent} to handle the degeneracy of the linear dispersion at $\Lmd = 1$.
\end{remark}

\medskip

Next, we give a precised description on the bilinear multipliers for the nonlinear evolution of the scalar profiles $\Fp_\pm$ defined by \eqref{eq:def-Fp}. The key null structure is revealed via symmetrization between the two inputs, see Remark \ref{rem:null-sym} below. The multipliers $\m_\mua^{(1)}$ and $\m_\mua^{(2)}$  given by \eqref{eq:m1-def} and \eqref{eq:m2-def} below arise from a 2D subsystem of \eqref{eq:EC}, and their associated bilinear forms vanish in the axisymmetric case.

\begin{proposition}
	\label{prop:nonlinear} For each $\mua = (\mu, \mu_1, \mu_2) \in \{+,-\}^3$, there holds that
	\begin{equation}
		R_\mu \mathcal{N}_{\mua}(R_{\mu_1}^{-1} G_1,R_{\mu_2}^{-1} G_2) =  \Q_\mua[\m_\mua](G_1, G_2), \quad \forall \ \mathrm{scalar \ functions} \ G_1, G_2, 
	\end{equation}
	hence in particular,
	\begin{equation}
		\partial_t \Fp_{\mu} = \sum_{\mu_1, \mu_2 = \pm} R_\mu \mathcal{N}_{\mua}(R_{\mu_1}^{-1} \Fp_{\mu_1},R_{\mu_2}^{-1} \Fp_{\mu_2}) = \sum_{\mu_1, \mu_2 = \pm} \Q_\mua[\m_\mua](\Fp_{\mu_1}, \Fp_{\mu_2}).
	\end{equation}
	Here, the bilinear multipliers $\m_\mua = \m_\mua (\xi, \eta)$
	can be decomposed into three parts
	\begin{equation}
		\m_\mua = \m_\mua^{(1)} + \m_\mua^{(2)} + \m_\mua^{(3)} 
	\end{equation}
	satisfying the following statements.
	\begin{enumerate}[(i)]
		\item $\m_\mua^{(1)}$ is given by
		\begin{equation} \label{eq:m1-def}
			\m_\mua^{(1)} = - \frac{1}{2} \xi_h \cdot
			\frac{(\xi - \eta)_h^{\perp}}{| (\xi - \eta)_h |}  \frac{\xi_h}{| \xi_h
				|} \cdot \frac{\eta_h}{| \eta_h |}, \quad \forall \mua
			\in \{ +, - \}^3.
		\end{equation}
		\item $\m_\mua^{(2)}$ is given by
		\begin{equation} \label{eq:m2-def}
			 \m_\mua^{(2)} = - \frac{\mu \mu_2}{2} \xi_h \cdot
			\frac{(\xi - \eta)_h^{\perp}}{| (\xi - \eta)_h |}, \quad \forall \mua \in \{ +, - \}^3.
		\end{equation}
		\item (GPW null type structure \cite{GuoInvent}) For each $ \mua \in \{ +, - \}^3$, $\m_\mua^{(3)}$ is a linear combination of multipliers taking the form
		\begin{equation} \label{eq:GPW-null}
			| \xi | \Lambda (\zeta_1) \sqrt{1 - \Lambda^2 (\zeta_2)}  \, \mathfrak{s}
			(\xi, \eta), \quad  \zeta_{1,2} \in \{\xi, \xi - \eta, \eta \}
		\end{equation}
		where  $\mathfrak{s}$  is a 0th order homogeneous symbol generated by the 
		set $\mathfrak{S} := \mathfrak{S}_1 \cup \mathfrak{S}_2$ with
		\[ 
		\mathfrak{S}_1 = \left\{ \Lambda (\zeta), \sqrt{1 - \Lambda^2 (\zeta)},
		\left(1 + \sqrt{1 - \Lambda^2 (\zeta)}\right)^{-1} : \zeta \in \{ \xi, \xi -
		\eta, \eta \} \right\}, \]
		\[ \mathfrak{S}_2 = \left\{ \frac{\zeta \cdot \vartheta}{| \zeta | |
			\vartheta |}, \frac{\zeta \cdot \vartheta^{\perp}}{| \zeta | |
			\vartheta |} : \zeta \in \{ \xi, \xi - \eta, \eta \}, \vartheta \in \{
		\xi - \eta, \eta \} \right\} . \]
		\item (Energy structure) For each $j=1,2,3$ and each $\mua \in \{ +, - \}^3$, there holds that
		\begin{align} \label{eq:ener-struc}
			\mathfrak{m}^{(j)}_{(\mu, \mu_1, \mu_2)}(\xi,\eta) + \overline{\mathfrak{m}^{(j)}_{(\mu_2, \mu_1, \mu)}(\eta,\xi)} = 0.
		\end{align}
	\end{enumerate}
\end{proposition}

\begin{proof}
	Let us denote
	\begin{equation} \label{eq:510-71}
		\g_1 = R_{\mu_1}^{-1} e^{\mu_1 i t \Lmd} G_1 = (- \nabla_h^\perp |\nabla_h|^{-1} + \mu_1 i \Lmd \nabla_h |\nabla_h|) e^{\mu_1 i t \Lmd} G_1 + \mu_1 \sqrt{1-\Lmd^2} e^{\mu_1 i t \Lmd} G_1  \vec{e}_3
	\end{equation}
	and 
	\begin{equation} \label{eq:510-72}
		\g_2 = R_{\mu_2}^{-1} e^{\mu_2 i t \Lmd} G_2 = (- \nabla_h^\perp |\nabla_h|^{-1} + \mu_2 i \Lmd \nabla_h |\nabla_h|) e^{\mu_2 i t \Lmd} G_2 + \mu_2 \sqrt{1-\Lmd^2} e^{\mu_2 i t \Lmd} G_2 \vec{e}_3
	\end{equation}
	By the definitions of $R_\mu, \mathcal{N}_\mua$, we have
	\begin{align} \label{eq:510-61}
		e^{\mu i t \Lmd} R_\mu \mathcal{N_\mu} (\g_1, \g_2) &= - \frac12 \left(|\nabla_h|^{-1} \mathrm{curl}_3  \nabla \cdot (\g_1 \otimes \g_2) + \mu |\nabla| |\nabla_h|^{-1} \vec{e}_3 \cdot (\PL \nabla \cdot (\g_1 \otimes \g_2))\right) 
	\end{align}
	We study the two terms in \eqref{eq:510-61} separately.
	
	\medskip
	
	For the first part of \eqref{eq:510-61}, there holds that
		\begin{eqnarray} \label{eq:31-a} 
	 | \nabla_h |^{- 1} \text{curl}_3
		(\nabla \cdot (\g_1 \otimes \g_2)) =
	\sum_{i = 1}^3 ( | \nabla_h |^{- 1} \partial_1 \partial_i (g_{1,i}
		g_{2,2}) - | \nabla_h |^{- 1} \partial_2 \partial_i (g_{1,i} g_{2,1})) .
	\end{eqnarray}
	where we write $g_{i,j}$ for the $j$-th component of $\g_i$. For the terms in \eqref{eq:31-a} with $i = 1, 2$, we compute the Fourier transform as
	\begin{align} 
	& 	\F \{  | \nabla_h |^{- 1} \partial_1 \nabla_h \cdot (g_{1,h} g_{2,2}) -
	| \nabla_h |^{- 1} \partial_2 \nabla_h \cdot (g_{1,h} g_{2,1}) \} (\xi) \nonumber \\
	&\quad = -\frac{1}{(2\pi)^3}\int_{\R^3}
	\xi_h \cdot \wh{g_{1,h}} (\xi - \eta)  \frac{\xi_h^{\perp}}{| \xi_h |} \cdot
	\wh{g_{2,h}} (\eta) d \eta 
	\end{align}
	Using the expressions \eqref{eq:510-71}--\eqref{eq:510-72} and the fact that $\xi_h = | \xi | \sqrt{1 - \Lambda^2 (\xi)} 
	\frac{\xi_h}{| \xi_h |}$, we see that the terms with coefficient $\mu_1$ or $\mu_2$ all contribute to $\mathfrak{m}_{\mua}^{(3)}$ satisfying (iii), while the remaining
	one term
	\[ -\frac{1}{(2\pi)^3} \int_{\R^3} \xi_h \cdot \frac{i (\xi - \eta)_h^{\perp}}{| (\xi - \eta)_h |}
	e^{\mu_1 i t \Lmd(\xi-\eta)} \wh{G_1} (\xi - \eta)  \frac{\xi_h}{| \xi_h |} \cdot \frac{i \eta_h}{|
		\eta_h |} e^{\mu_2 i t \Lmd(\eta)} \wh{G_2} (\eta) d \eta \]
	accounts for $\mathfrak{m}_{\mua}^{(1)}$. For the term in \eqref{eq:31-a} with $i = 3$,
	using \eqref{eq:510-71}--\eqref{eq:510-72},
	we obtain a bilinear expression
	\begin{eqnarray*}
		&   \F \{  | \nabla_h |^{- 1} \partial_1 \partial_3 (g_{1,3} g_{2,2})
		- | \nabla_h |^{- 1} \partial_2 \partial_3 (g_{1,3} g_{2,1}) \} (\xi)\\
		& =  - \frac{\mu_1 }{(2\pi)^3} \int_{\R^3} {| \xi |} {\Lambda
			(\xi) \sqrt{1 - \Lambda^2 (\xi - \eta)}}  e^{\mu_1 i t \Lmd(\xi-\eta)} \wh{G_1} (\xi - \eta) \cdot \\
		&\qquad \qquad \qquad \frac{\xi_h^{\perp}}{| \xi_h |} \cdot \left( - \frac{i \eta_h^{\perp}}{|
			\eta_h |}   - \mu_2 \frac{\eta_h}{| \eta_h |} \Lambda (\eta) 
		 \right) e^{\mu_2 i t \Lmd(\eta)} \wh{G_2} (\eta) d \eta
	\end{eqnarray*}
	which contributes to $\mathfrak{m}_3$.
	
	\medskip
	
	For the second part of \eqref{eq:510-61}, there holds that
	\begin{align}
		&\qquad \qquad  |\nabla| |\nabla_h|^{-1} \vec{e}_3 \cdot (\PL \nabla \cdot (\g_1 \otimes \g_2)) \nonumber \\
		&=  | \nabla | | \nabla_h |^{- 1} \left\{  \sum_{i = 1}^3 \partial_i
		(g_{1,i} g_{2,3}) - \sum_{j, k = 1}^3 \partial_3 \Delta^{- 1} \partial_j
		\partial_k (g_{1,j} g_{2,k}) \right\} .  \label{eq:31-d}
	\end{align}
	For the case $i = j = k = 3$, notice that
	\[ | \nabla | | \nabla_h |^{- 1} \{  \partial_3 (g_{1,3} g_{2,3}) - \partial_3
	\Delta^{- 1} \partial_3 \partial_3 (g_{1,3} g_{2,3}) \} =  | \nabla | |
	\nabla_h |^{- 1} \partial_3 \left( 1 + \frac{\partial_3^2}{| \nabla
		|^2} \right) (g_{1,3} g_{2,3}), \]
	and on the Fourier side there holds that
	\begin{align*}
		\F \left\{ | \nabla | | \nabla_h |^{- 1} \partial_3 \left(
		1 + \frac{\partial_3^2}{| \nabla |^2} \right) (g_{1,3} g_{2,3}) \right\}
		= \frac{i}{(2\pi)^3} \int_{\R^3} | \xi | \Lambda (\xi) \sqrt{1 -
			\Lambda^2 (\xi)}  \wh{g_{1,3}} (\xi - \eta) \wh{g_{2,3}} (\eta) d \eta
	\end{align*}
	which gives a contribution to $\mathfrak{m}_3$. For the first summation in \eqref{eq:31-d} over $i = 1, 2$, we compute that
	\begin{align}
		& \qquad  \F \{  | \nabla | | \nabla_h |^{- 1} \nabla_h \cdot (g_{1,h}
		g_{2,3}) \} \nonumber \\
		& = - \frac{\mu_2}{(2\pi)^3} \int_{\R^3} {| \xi |} \frac{\xi_h}{| \xi_h |} \cdot
		\left( - \frac{(\xi - \eta)_h^{\perp}}{| (\xi - \eta)_h |}   + i \mu_1 {\Lambda (\xi - \eta)} \frac{(\xi -
			\eta)_h}{| (\xi - \eta)_h |}  \right)  \cdot \nonumber \\
		&\qquad \qquad \qquad \cdot e^{\mu_1 i t \Lmd(\xi-\eta)} \wh{G_1} (\xi - \eta) {\sqrt{1 - \Lambda^2 (\eta)}}  e^{\mu_2 i t \Lmd(\eta)} \wh{G_2} (\eta) d \eta.
	\end{align}
	The term with a coefficient $\mu_1$ contributes to $\mathfrak{m}_3$.
	The remaining one term can be decomposed into
	\begin{equation}
		  \frac{\mu_2}{(2\pi)^3} \int_{\R^3} \left( \frac{\sqrt{1 - \Lambda^2 (\eta)}}{\sqrt{1 - \Lambda^2
				(\xi)}} - 1 \right) \xi_h \cdot \frac{(\xi - \eta)_h^{\perp}}{| (\xi -
			\eta)_h |} e^{\mu_1 i t \Lmd(\xi-\eta)} \wh{G_1} (\xi - \eta)   e^{\mu_2 i t \Lmd(\eta)} \wh{G_2} (\eta) d
		\eta  
		\label{eq:31-e}
	\end{equation}
	and
	\begin{equation} \label{eq:510-110}
		 \frac{\mu_2}{(2\pi)^3} \int_{\R^3} \xi_h \cdot \frac{(\xi - \eta)_h^{\perp}}{| (\xi - \eta)_h |}
		e^{\mu_1 i t \Lmd(\xi-\eta)} \wh{G_1} (\xi - \eta)   e^{\mu_2 i t \Lmd(\eta)} \wh{G_2} (\eta) d \eta.
	\end{equation}
	Using the identity
	\[ \frac{\sqrt{1 - \Lambda^2 (\eta)}}{\sqrt{1 - \Lambda^2 (\xi)}} - 1 =
	\frac{\sqrt{1 - \Lambda^2 (\eta)}}{\sqrt{1 - \Lambda^2 (\xi)}} 
	\frac{\Lambda^2 (\xi)}{1 + \sqrt{1 - \Lambda^2 (\xi)}} - \frac{\Lambda^2
		(\eta)}{1 + \sqrt{1 - \Lambda^2 (\eta)}}, \]
	we see that  \eqref{eq:31-e} gives a contribution to
	$\m_\mua^{(3)}$. The integral \eqref{eq:510-110} clearly accounts
	for $\m_\mua^{(2)}$. For the second summation in \eqref{eq:31-d} over $j = 1, 2$ with $k = 3$ (or
	similarly $j = 3$, $k = 1, 2$), we compute that
	\begin{align} 
		&\qquad \F \{ -| \nabla | | \nabla_h |^{- 1} \partial_3 \Delta^{- 1}
	\nabla_h \cdot \partial_3 (g_{1,h} g_{2,3}) \} (\xi)  \nonumber \\
	&= - \frac{i}{(2\pi)^3} \int_{\R^3} {| \xi
		| \Lambda^2 (\xi)} \frac{ \xi_h}{| \xi_h |} \cdot \wh{g_{1,h}} (\xi - \eta)
	{\sqrt{1 - \Lambda^2 (\eta)}} e^{\mu_2 i t \Lmd(\eta)} \wh{G_2} (\eta) d \eta
	\end{align}
	which gives a contribution to $\m_\mua^{(3)}$. Finally, for the summation
	in \eqref{eq:31-d} over $j = 1, 2$ and $k = 1, 2$, we compute that
	\begin{align} 
	&\qquad \F \left\{ - \sum^2_{j, k = 1} | \nabla | | \nabla_h |^{- 1}
	\partial_3 \Delta^{- 1} \partial_j \partial_k (g_{1,j} g_{2,k}) \right\} (\xi) \nonumber \\
	&= -\frac{i}{(2\pi)^3} \int_{\R^3}
	{\Lambda (\xi) \sqrt{1 - \Lambda^2 (\xi)} | \xi |} 
	\frac{\xi_h}{| \xi_h |} \cdot \wh{g_{1,h}} (\xi - \eta)  \frac{\xi_h}{|
		\xi_h |} \cdot \wh{g_{2,h}} (\eta) d \eta 
		\end{align}
	which also gives a contribution to $\m_\mua^{(3)}$. 
	
	\medskip
	
	In summary, each bilinear term in
	\eqref{eq:510-61} gives a contribution to one of
	$\m_\mua^{(j)}$, $j = 1, 2, 3$, verifying (i)-(iii). 
	
	\medskip
	
	Moreover, collecting all contributions, we arrive at the following explicit formula for $\m_\mua^{(3)}$:
	\begin{align}
		\m_\mua^{(3)} &= \frac{i \mu_2}{2}  \Lambda(\eta) \xi_h \cdot  \frac{(\xi-\eta)_h^\perp}{|(\xi-\eta)_h|}  \frac{\xi_h^\perp}{| \xi_h |} \cdot  \frac{\eta_h}{| \eta_h |} + \frac{i\mu}{2}  \Lambda(\xi) \xi_h \cdot \frac{(\xi-\eta)_h^\perp}{|(\xi-\eta)_h|}  \frac{\xi_h^\perp}{|\xi_h|} \cdot \frac{\eta_h}{|\eta_h|} \nonumber \\
		&\quad - \frac{\mu_1 \mu_2}{2} {W}(\xi,\eta) \Lambda(\eta) \frac{\xi_h^\perp}{| \xi_h |} \cdot \frac{\eta_h}{| \eta_h |} - \frac{\mu\mu_1}{2} {W}(\xi,\eta) \Lambda(\xi) \frac{\xi_h^\perp}{|\xi_h|} \cdot \frac{\eta_h}{|\eta_h|} \nonumber \\
		&\quad + \frac{\mu \mu_2}{2} \left\{ - \Lambda(\xi) \Lambda(\eta) \frac{\xi_h}{|\xi_h|} \cdot \frac{ \eta_h}{|\eta_h|}  - \sqrt{1-\Lambda^2(\xi)} \sqrt{1-\Lambda^2(\eta)}  + 1 \right\} \xi_h \cdot \frac{ (\xi-\eta)^\perp_h}{|(\xi-\eta)_h|} \nonumber \\
		&\quad - \frac{i \mu \mu_1 \mu_2}{2} {W}(\xi,\eta) \left\{ \Lambda(\xi) \Lambda(\eta) \frac{\xi_h}{|\xi_h|} \cdot \frac{ \eta_h}{|\eta_h|} + \sqrt{1 -
			\Lambda^2 (\xi)} \sqrt{1 - \Lambda^2
			(\eta)} \right\} \nonumber \\
		&\quad  - \frac{i \mu_1}{2}   {W}(\xi,\eta) \frac{\xi_h}{| \xi_h |} \cdot \frac{ \eta_h}{|\eta_h|}, \label{eq:m3-explicit}
	\end{align}
	where we denote
	\begin{align} \label{eq:510-31}
		{W}(\xi,\eta) &= -\Lambda(\xi-\eta) \xi_h \cdot \frac{(\xi-\eta)_h}{|(\xi-\eta)_h|} +\xi_3 \sqrt{1-\Lambda^2(\xi-\eta)} \nonumber \\
		&= -\Lambda(\xi-\eta) \eta_h \cdot \frac{(\xi-\eta)_h}{|(\xi-\eta)_h|} +\eta_3 \sqrt{1-\Lambda^2(\xi-\eta)}.
	\end{align}
	By definition, $\m_\mua^{(1)}$ and $\m_\mua^{(2)}$ clearly satisfy \eqref{eq:ener-struc}.  A direct computation using \eqref{eq:m3-explicit}--\eqref{eq:510-31} and the fact that
	\begin{align*}
		W(\xi,\eta) = W(\eta,\xi)
	\end{align*}
	shows that $\m_\mua^{(3)}$ also satisfies \eqref{eq:ener-struc}. Hence, (iv) is proved.
\end{proof}

\smallskip

\begin{remark} \label{rem:null-sym}
	$\mathfrak{m}^{(1)}$ and $\mathfrak{m}^{(2)}$ exhibit the following hidden null type structures via symmetrization. When $\mu_1 = \mu_2$, there holds that
	\begin{align} \label{eq:510-210}
		\m_\mua^{(1), \mathrm{Sym}}&:= \m_\mua^{(1)}(\xi, \eta) + \m_\mua^{(1)}(\xi, \xi - \eta) \nonumber \\
		&= \frac{(\xi_h \cdot \eta_h^\perp) \big( \xi_h \cdot (2\eta_h - \xi_h) \big)}{2 |(\xi-\eta)_h||\xi_h||\eta_h|} \nonumber \\
		&= \frac{(\xi_h \cdot \eta_h^\perp) (|\eta_h|+|(\xi-\eta)_h|) (|\eta_h|-|(\xi-\eta)_h|)}{2 |(\xi-\eta)_h||\xi_h||\eta_h|}
	\end{align}
	and
	\begin{align} \label{eq:510-211}
		\m_\mua^{(2), \mathrm{Sym}} &:= \m_\mua^{(2)}(\xi, \eta) + \m_\mua^{(2)}(\xi, \xi - \eta) \nonumber \\
		&= \frac{\mu \mu_1}{2} \xi_h \cdot \eta_h^\perp \left(\frac{1}{|(\xi-\eta)_h|} - \frac{1}{|\eta_h|} \right)  \nonumber \\
		&= \frac{\mu \mu_1}{2} \frac{ \xi_h \cdot \eta_h^\perp}{|(\xi-\eta)_h||\eta_h|} \big(|\eta_h| - |(\xi-\eta)_h|\big).
	\end{align}
	The common factor $|\eta_h| - |(\xi-\eta)_h|$ in \eqref{eq:510-210}--\eqref{eq:510-211} is of null form since it vanishes on the key resonant set \eqref{eq:524-3}. This structure will play an essential role in the analysis of Section \ref{sec:742}.
\end{remark}

\medskip


The next lemma helps to recover the basic physical product structure of the Euler nonlinearity.

\begin{lemma} \label{lem:L-bilinear}
	For $\m_\mua$ given by Proposition \ref{prop:nonlinear},  $\Q_\mua[\m_\mua](G_1, G_2)$ can be linearly decomposed into bilinear forms of the type
	\begin{equation} \label{eq:411-1}
		|\nabla| e^{-\mu i  t \Lmd} D_0[(D_1 e^{\mu_1 i  t \Lmd} G_1)\cdot(D_2 e^{ \mu_2 i  t \Lmd} G_2)],
	\end{equation}
	where each of $D_{0,1,2}$ is a product of operators from the set
	\begin{equation} \label{eq:415-02}
		\left\{\Lmd, \sqrt{1-\Lmd^2}, \frac{\partial_1}{|\nabla_h|}, \frac{\partial_2}{|\nabla_h|} \right\}.
	\end{equation}
	Moreover, the above statement is still true if we replace \eqref{eq:411-1} by
	\begin{equation}
		  e^{-\mu i  t \Lmd} D_0[(D_1 e^{\mu_1 i  t \Lmd} G_1)\cdot(|\nabla| D_2 e^{\mu_2 i  t \Lmd} G_2)].
	\end{equation}
\end{lemma}
\begin{proof}
	By the exact form of $\m_\mua$ in \eqref{eq:m1-def}, \eqref{eq:m2-def} and \eqref{eq:m3-explicit} as well as the fact that
	\begin{equation}
		\xi_h \cdot \frac{(\xi-\eta)_h^\perp}{|(\xi-\eta)_h|} = \eta_h \cdot \frac{(\xi-\eta)_h^\perp}{|(\xi-\eta)_h|},
	\end{equation}
	we can write $\m_\mua$ as a linear combination of multipliers taking the form
	\begin{equation*}
		|\xi| \mathfrak{s}_0(\xi) \mathfrak{s}_1(\xi-\eta) \mathfrak{s}_2(\eta) \quad  \mathrm{or} \quad |\eta| \mathfrak{s}_0(\xi) \mathfrak{s}_1(\xi-\eta) \mathfrak{s}_2(\eta),
	\end{equation*} 
	where each $\mathfrak{s}_i(\zeta), \zeta \in \{\xi, \xi-\eta, \eta\}$  is a product of symbols from the set
	\begin{equation}
		\left\{\Lambda(\zeta), \sqrt{1-\Lambda^2(\zeta)}, \frac{\zeta_1}{|\zeta_h|}, \frac{\zeta_2}{|\zeta_h|}\right\}.
	\end{equation}
\end{proof}

\section{Linear dispersion}\label{sec:4}

In this section, we take any $\beta>0$ (The constants of this section will depend only on $\beta$.) and establish various sharp dispersive estimates using the $X$ and $Y$ norms defined in Section \ref{sec:24}.  

For frequencies $\xi$ from the support of $\wt{\chi}^\hp(\xi) \varphi(|\xi_h|)$ (see Section \ref{sec:515-22}), the group velocity for $e^{it\Lmd}$ in the vertical direction satisfies
\begin{equation}
	-\partial_3 \Lmd = - \frac{|\xi_h|^2}{|\xi|^3} \in \left( -\frac{3 t}{2}, -\frac{t}{10}\right).
\end{equation}
Hence, for a function $f$ with $\wh{f}$ smooth and supported on such frequencies, the $L^2$ norm of $e^{it\Lmd} f (t \gg 1)$ will concentrate in the slab  $\left\{-\frac{3 t}{2} < x_3 < -\frac{t}{10}\right\}$, while strong decay estimates are expected outside the slab. A sharp quantitative statement on the linear asymptotics  is proved in Lemma \ref{lem:slab} below for functions from  our $X$ space. 

In Section \ref{sec:52}, we introduce a family of Fourier-side localizations with support close to circles on the $\xi_h$-plane, which will be used in our  bilinear multiscale analysis. In Lemma \ref{lem:51-3}, we perform a kind of wave packet decomposition and obtain precised linear asymptotics in the physical space for each localized piece.

It has been observed in \cite[Section 4]{GuoInvent} that for frequencies away from $\Lmd = 0$ and $\Lmd = 1$, the dispersive estimates can be enhanced. We prove (essentially) the  same result in Lemma \ref{lem:hp-disp} and Lemma \ref{lem:vp-disp} below based on our choice of norms and coordinate systems. 

Finally, in Lemma \ref{lem:largel}, we prove the optimal decay rate in $L^2$ and $L^\infty$ norms at far field, which can be viewed as a property of finite propagation speed.

\subsection{Dispersive estimates for $\bighp$-pieces} \label{sec:41}

\begin{lemma} \label{lem:slab}
	Let $f = P^\hp_k F$ with $F \in X^{0,3}_\beta$ and $k \in\mathbb{Z}$, and let $\mu \in \{+,-\}$, $2^m \le t < 2^{m+1} \, (m \ge 1)$, then the following statements hold.
	\begin{enumerate}[(i)]
		\item If $|x_3| \sim t 2^{-k}$, we have the pointwise estimate
		\begin{equation}
			 \left|e^{\mu i t \Lmd} f (x)\right| \lesssim \frac{2^{\frac32 k} }{t \langle 2^k |x_h|\rangle^\frac12} \|F\|_{X^{0,3}_{\beta}}.
		\end{equation}
		\item Pointwise decay outside a slab. Let $\mathcal{R} = Z_{[m-k-4,m-k+1]^c}^{-\mu} e^{\mu i t \Lmd} f$, then 
		\begin{equation} \label{eq:803-001}
			\left|\mathcal{R}(x) \right| \lesssim \frac{2^{\frac32 k} }{(2^k |x_3| + t)^{1+\beta}} \|F\|_{X^{0}_{\beta}}.
		\end{equation}
		\item $L^2$ decay outside a slab.
		\begin{equation} \label{eq:803-002}
			\sup_{l \in \Z} 2^{(1+\beta)(l+k)} \left\| Z_l \mathcal{T}_{\mu 2^{m-k}}  \mathcal{R} \right\|_{L^2} \lesssim \|F\|_{X^{0}_{\beta}},
		\end{equation}
		\begin{equation} \label{eq:803-003}
			\sup_{l \in \Z} 2^{(1+\beta)(l+k)} \left\| Z_l \mathcal{T}_{\mu 2^{m-k}}  \mathcal{R} \right\|_{L^2_{x_3}  L^\infty_{x_h}  } \lesssim 2^{ k}\|F\|_{X^{0}_{\beta}}.
		\end{equation}
		Here, we used the notations 
		$$(\mathcal{T}_s g)(x) := g(x_h, x_3-s)$$
		and
		$$\|g\|_{ L^2_{x_3}  L^\infty_{x_h} } :=  \left(\int_\R  \mathrm{ess\,sup}_{x_h \in \R^2} |g(x_h,x_3)|^2 dx_3 \right)^\frac12.$$
	\end{enumerate}
\end{lemma}

\begin{proof}
In the proof we shall use the local coordinate system $(\r,\z,\th)$, in particular the partial derivatives $\partial_{\r,\z,\th}$ are defined as in \eqref{eq:413-1}--\eqref{eq:413-3}.	By scaling and reflection with respect to the $x_h$-plane, it suffices to consider the case $k=0$ and $\mu = +$. By definition, we can write
	\begin{equation} \label{eq:F-rep-1}
		e^{i t \Lmd} f = \frac{1}{(2 \pi)^3}\int_{\R^+ \times \R \times \mathbb{S}^1} e^{i \Psi} \wh{f}(\r,\z,\th) \r^2 d\r d\z d\th,
	\end{equation}
	\begin{equation} \label{eq:Psi-def}
		\Psi = t \Lmd + x \cdot \xi = \frac{\z t}{\sqrt{1+\z^2}} + x_3 \r \z + |x_h| \r \cos (\th - \th_x).
	\end{equation}
	Here, $\theta_x$ is the angular coordinate of $x$ in the usual cylindrical chart.
	
	First, we consider pointwise estimate for $e^{ i t \Lmd} f$ in the case $|x_3| \sim t$. Let $\chi_1(\theta)$ be a smooth bump function on the unit circle, with support contained in $[-\frac{\pi}{2}, \frac{\pi}{2}]$ and is equal to $1$ on $[-\frac{\pi}{4}, \frac{\pi}{4}]$, and then let $\chi_2(\theta)  =  1-\chi_1(\theta)$. Performing the integration in $\th$ in \eqref{eq:F-rep-1}, we get
	\begin{equation} \label{eq:F-decomp}
		e^{ i t \Lmd} f(x) =  \sum_{j=1,2}  \int_{\R^+ \times \R} e^{i \Psi_{j}} g_{j}(\r, \z) \r^2 d\r d\z =: \sum_{j=1,2} A_j,
	\end{equation}
	where we denote
	\begin{equation} \label{eq:41-1}
		\Psi_1 = \frac{\z t}{\sqrt{1+\z^2}} + x_3 \r \z + |x_h| \r,
	\end{equation}
	\begin{equation} \label{eq:41-2}
		\Psi_2 = \frac{\z t}{\sqrt{1+\z^2}} + x_3 \r \z - |x_h| \r,
	\end{equation}
	and
	\begin{equation}  \label{eq:41-3}
		g_1(\r,\z) = \frac{1}{(2 \pi)^3} \int_{-\pi}^{\pi} e^{ i |x_h| \r (\cos (\th - \th_x) - 1)} \chi_1(\th - \th_x) \wh{f}(\r, \z, \th) d\theta,
	\end{equation}
	\begin{equation} \label{eq:41-4}
		g_2(\r,\z) = \frac{1}{(2 \pi)^3} \int_{-\pi}^{\pi} e^{ i |x_h| \r (\cos (\th - \th_x) + 1)} \chi_2(\th - \th_x) \wh{f}(\r, \z, \th) d\theta.
	\end{equation}
	Let us focus on the estimate for $A_1$ as $A_2$ can be treated in a  similar way. By standard theory for 1D oscillatory integrals, we have 
	\begin{equation} \label{eq:41-5}
		|\partial_\r^n g_1(\rho, \zeta)| \lesssim \langle |x_h| \rangle^{-\frac12} \int_{-\pi}^{\pi} \left|\{\partial_\r,\partial_\th\}^{\le n+1} \wh{f}(\r,\z,\th) \right| d\theta, \quad n=0,1,2,\cdots.
	\end{equation}
	Using
	\begin{equation} \label{eq:41-8}
		\partial_\r \Psi_1 = x_3 \z + |x_h|	
	\end{equation}
	and integrating by parts in $\r$ twice, we obtain 
	\begin{align} \label{eq:427-1}
		\left| \int_{\R^+} e^{i \Psi_{1}} g_{1}(\r, \z) \r^2 d\r \right| &=   \left| \int_{\R^+} \frac{e^{i \Psi_1}}{(\partial_\r \Psi_1)^2} \partial_\r^2 (g_1 \r^2) d\r \right|  \nonumber \\
		&\le (x_3 \z + |x_h|)^{-2}  \int_{\R^+} \left| \partial_\r^2 (g_1 \r^2)  \right| d\r \nonumber \\
		&\lesssim (x_3 \z + |x_h|)^{-2} \langle |x_h| \rangle^{-\frac12}  \int_{\R^+\times \mathbb{S}^1} |\{S, \Omega\}^{\le 3} \wh{f}| d\r d\th \nonumber \\
		&\lesssim (x_3 \z + |x_h|)^{-2} \langle |x_h| \rangle^{-\frac12} \|F\|_{X^{0,3}_{\beta}}.
	\end{align}
	In the last line we used Sobolev embedding in the $\z$ direction. On the other hand, a crude bound gives
	\begin{align} \label{eq:427-2}
			\left| \int_{\R^+} e^{i \Psi_{1}} g_{1}(\r, \z) \r^2 d\r \right| &\lesssim  \int_{\R^+} 	\left|  g_{1}(\r, \z) \right| d\r \nonumber \\
			&\lesssim \langle |x_h| \rangle^{-\frac12} \int_{\R^+ \times \mathbb{S}^1}	| \{S, \Omega\}^{\le 1} \wh{f} | d\r d\th \nonumber \\
			&\lesssim  \langle |x_h| \rangle^{-\frac12} \|F\|_{X^{0,3}_{\beta}}.
	\end{align}
	Integrating \eqref{eq:427-1}--\eqref{eq:427-2} in $\z$ gives
	\begin{align}
		|A_1| &\lesssim \langle |x_h| \rangle^{-\frac12} \int_{\R} \min \left\{ (x_3 \z + |x_h|)^{-2}, 1 \right\} d\z \   \|F\|_{X^{0,3}_{\beta}} \\
		&\lesssim \langle |x_h| \rangle^{-\frac12} |x_3|^{-1} \|F\|_{X^{0,3}_{\beta}} \\
		&\lesssim \langle |x_h| \rangle^{-\frac12} t^{-1} \|F\|_{X^{0,3}_{\beta}}.
	\end{align}
	This finishes the proof of (i) in the case $|x_3| \sim t$.
	
	Next, consider the case $-x_3 \notin [2^{m-4}, 3 \times 2^{m}]$. Let $l = \lfloor \log_2 |x_3 + 2^m| \rfloor$, then $t \lesssim |x_3 + 2^m| \sim 2^l$. Recall that $\wt{P}^\hp$ is a localization operator which has similar support properties as $P^\hp$ on the Fourier side and satisfies $\wt{P}^\hp P^\hp = P^\hp$.  We make the decomposition
	\begin{equation}
		f = \wt{P}^\hp Z_{\ge l-7} f +  \wt{P}^\hp  Z_{< l-7} f =: h_1 + h_2.
	\end{equation}
	Using the formula
	\begin{equation} \label{eq:427-001}
		(\F_h e^{i t \Lmd} f)(\xi_h, x_3) = \F_3^{-1} (e^{i t \Lmd} \wh{f}) = \frac{1}{2 \pi} \int_{\R} e^{i (t \Lmd + \xi_3 x_3)} \wh{f}(\xi) d\xi_3,
	\end{equation}
	we make a corresponding decomposition
	\begin{equation}
		\F_h e^{i t \Lmd} f = I_1 + I_2:= \frac{1}{2 \pi} \int_{\R} e^{i (t \Lmd + \xi_3 x_3)} \wh{h_1}(\xi) d\xi_3 + \frac{1}{2 \pi} \int_{\R} e^{i (t \Lmd + \xi_3 x_3)} \wh{h_2}(\xi) d\xi_3.
	\end{equation}
	Note that $I_1$ and $I_2$ are supported on $|\xi_h|\sim 1$. Since $[\F_h, Z_l] = 0$, the second part $I_2$ can be written as
	\begin{align} \label{eq:31-1}
		 2\pi I_2		 &= \int_{\R} e^{i (t \Lmd + \xi_3 x_3)} \wt{\chi}^\hp(\z) (\F_3 Z_{< l-7} \F_h f)(\xi)  \r d\z \nonumber \\
		&=\int_{\R}  (Z_{< l-7} \F_h f)(\xi_h, y_3) dy_3 \int_{\R} \r \wt{\chi}^\hp(\z)  e^{i (t \Lmd +  \r \z  (x_3 - y_3))}  d\z.
	\end{align}
	Note that for $y_3$ in the support of $Z_{<l-7}$ and for $\xi$ in the support of $\wt{\chi}^{\hp}(\xi) \varphi(|\xi_h|)$, we have
	\begin{equation}
		|\partial_\z (t \Lmd +  \r \z  (x_3 - y_3))| = \left|\frac{t}{(1+\z^2)^{\frac32}} + \r (x_3-y_3) \right| \gtrsim t + |x_3|,
	\end{equation}
	and 
	\begin{equation}
		|\partial_\z^n (t \Lmd +  \r \z  (x_3 - y_3))|  \lesssim_n t, \quad n=2,3,\cdots.
	\end{equation}
	Hence, repeated integration by parts in $\z$ gives, for any $K \ge 1$,
	\begin{align}
		|I_2(\xi_h, x_3)| &\lesssim_K (t+|x_3|)^{-K} \int_{\R} \left|( \F_h f)(\xi_h, y_3)\right|  dy_3 \\
		&\lesssim (t+|x_3|)^{-K} \left(\int_{\R} \left|( \F_h f)(\xi_h, y_3)\right|^2 \langle y_3\rangle^2   dy_3\right)^\frac12,
	\end{align}
	and consequently, using Young's inequality for $\F_h^{-1}$, we obtain for any $x_3$,
	\begin{align} \label{eq:427-3}
		\|e^{i t \Lmd} {h_2}\|_{L^2_{x_h}\cap L^\infty_{x_h}}(x_3) &\lesssim_K (t+|x_3|)^{-K} \left(\int_{\R^3} \left|( \F_h f)(\xi_h, y_3)\right|^2 \langle y_3\rangle^2   d\xi_h dy_3\right)^\frac12 \nonumber\\
		&= (t+|x_3|)^{-K} \left(\int_{\R^3} \left|f(y)\right|^2 \langle y_3\rangle^2   dy\right)^\frac12 \nonumber\\
		&\lesssim  (t+|x_3|)^{-K} \|F\|_{X^{0}_\beta}.
	\end{align}
	In particular, by integrating \eqref{eq:427-3} in $x_3$ we deduce that
	\begin{align}
		\| e^{i t \Lmd} {h_2}\|_{L^2\left(\{-x_3 \notin [2^{m-4}, 3 \times 2^{m}], |x_3+t|\in [2^l, 2^{l+1})\}\right)} \lesssim_K  (t+2^l)^{-K} \|F\|_{X^{0}_\beta}.
	\end{align}
	Hence, the contribution of $h_2$ to $e^{it\Lmd} f$ is acceptable for (in fact, much better than) the estimates in (ii) and (iii). It remains to consider the contribution of $h_1$, \emph{i.e.}, $e^{i t \Lmd} h_1$. For the $L^2$ norm we have the direct estimate
	\begin{equation}
		 \|e^{it\Lmd} h_1\|_{L^2} \le \|h_1\|_{L^2} \lesssim 2^{-(1+\beta)l} \|F\|_{X^{0}_\beta}.
	\end{equation}
	Recall that $\wh{h_1}$ is supported on $\{\r \sim 1, |\zeta| \lesssim 1\}$. Using Young's inequality for the Fourier transform, we deduce an acceptable bound
	\begin{equation}
		\|e^{it\Lmd} h_1\|_{L^\infty} + \|e^{it\Lmd} h_1\|_{L^2_{x_3} L^\infty_{x_h}  } \lesssim \|h_1\|_{L^2} \lesssim 2^{-(1+\beta)l} \|F\|_{X^{0}_\beta}.
	\end{equation}
\end{proof}

\medskip

\begin{lemma} \label{lem:hp-disp}
	Let $f = P_k^\hp F$ with $F \in X^{0,3}_{\beta}$, $k \in \Z$ and let $\mu \in \{\pm\}$, $ 2^m \le t < 2^{m+1} (m \ge 1)$, $q\in \Z_-$, then we have
	\begin{enumerate}[(i)]
		\item If $|x_3| \sim t 2^{-k}$,
		\begin{equation} \label{eq:427-02}
			|e^{\mu i t \Lmd} P^{\hp, q} f| \lesssim \min\{ t^{-\frac32} 2^{-\frac{q}{2}}, t^{-1} \}\, 2^{\frac{3k}{2}} \|F\|_{X^{0,3}_{\beta}}.
		\end{equation}
		\item Decompose $f = f^I + f^{II}$ with
		\begin{equation} \label{eq:427-03}
			f^I := \wt{P}^\hp Z_{< m-k-7} f, \quad f^{II} := \wt{P}^\hp Z_{\ge m-k-7} f.
		\end{equation}
		Then, there holds that
		\begin{equation} \label{eq:427-002}
			\|e^{\mu i t \Lmd} P^{\hp,q} f^I\|_{L^\infty} \lesssim \min \{ t^{-\frac32} 2^{-\frac{q}{2}}, t^{-1} \} 2^{\frac32 k } \|F\|_{X^{0,3}_{\beta}}
		\end{equation}
		and
		\begin{equation} \label{eq:427-01}
			\|e^{\mu i t \Lmd} P^{\hp,q} f^{II}\|_{L^2} \lesssim t^{-(1+\beta)} \|F\|_{X^0_\beta},
		\end{equation}
		\begin{equation} \label{eq:730-101}
			\|e^{\mu i t \Lmd} P^{\hp,q} f^{II}\|_{L^\infty} \lesssim t^{-(1+\beta)} 2^{\frac{q}{2}+ \frac{3k}{2}} \|F\|_{X^0_\beta}.
		\end{equation}
	\item 
	\begin{equation} \label{eq:730-102}
			\sup_{q \in \Z_-} \|e^{\mu i t \Lmd} P^{\hp,q} f\|_{L^\infty} + \|e^{\mu i t \Lmd} f\|_{L^\infty} \lesssim t^{-1} 2^{\frac{3k}{2}} \|F\|_{X_\beta^{0,3}}.
	\end{equation}
	\end{enumerate}
\end{lemma}

\begin{proof}
	In the proof we use the local coordinate system $(\rho, \zeta, \th)$. By scaling and reflection, we may assume that $k = 0$ and $\mu = +$.  The estimate \eqref{eq:427-01} and \eqref{eq:730-101} follow from a direct X norm bound
	\begin{equation}
		\|e^{ i t \Lmd} P^{\hp,q} f^{II}\|_{L^2} \le \|f^{II}\|_{L^2} \lesssim t^{-(1+\beta)} \|F\|_{X^0_\beta}.
	\end{equation}
	and a crude estimate
	\begin{equation}
		\|e^{ i t \Lmd} P^{\hp,q} f^{II}\|_{L^\infty} \lesssim 2^{\frac{q}{2}}\|f^{II}\|_{L^2}.
	\end{equation}
	The estimate \eqref{eq:730-102} is a corollary of \eqref{eq:427-002} and \eqref{eq:730-101}. Hence, it remains to prove \eqref{eq:427-02} and \eqref{eq:427-002}. Due to the crude bound 
	\begin{align}
		\|e^{ i t \Lmd} P^{\hp, q} \{f, f^I\}\|_{L^\infty} \lesssim \|\chi^{\hp, q} \{\wh{f}, \wh{f^I}\}\|_{L^1} \lesssim 2^q \| \wh{f}\|_{L^\infty} \lesssim 2^q \|F\|_{X^{0,3}_{\beta'}},
	\end{align}
	it is sufficient to consider the case  $q \ge -m$.
	
	First, we consider the case $|x_3|\sim t$. By definition, one can write
	\begin{equation}
		e^{ i t \Lmd} P^{\hp,q} f(x) = \frac{1}{(2\pi)^3} \int_{\R^+ \times \R \times \mathbb{S}^1} e^{i\Psi} \chi^{\hp, q}(\xi) \wh{f}(\xi) \r d\r d\z d\th
	\end{equation}
	with $\Psi$ defined in \eqref{eq:Psi-def}. Similar to \eqref{eq:F-decomp}, here we have
	\begin{equation}
		e^{ i t \Lmd} P^{\hp,q} f = \sum_{j = 1,2} A_j:= \sum_{j=1,2} \int_{\R^+ \times \R} e^{i\Psi_j} \chi^{\hp, q}(\xi) g_j(\r,\z) \r^2 d\r d\z
	\end{equation}
	where $\Psi_j$  and $g_j$ are defined as in \eqref{eq:41-1}--\eqref{eq:41-4}. Multiplying \eqref{eq:427-1}--\eqref{eq:427-2} by $\chi^{\hp,q}$ and then integrating in $\z$, we get  
	\begin{align}
		|A_1| &\lesssim \langle |x_h| \rangle^{-\frac12} \int_{\R} \chi^{\hp,q}(\xi) \min\{(x_3 \z + |x_h|)^{-2}, 1\} d\z \, \|F\|_{X^{0,3}_{\beta}} \nonumber \\
		&\lesssim t^{-1} \langle |x_h| \rangle^{-\frac12} \int_{\R} \varphi(2^{-q} x_3^{-1} \lmd) \min\{(\lmd + |x_h|)^{-2}, 1\} d\lmd \, \|F\|_{X^{0,3}_{\beta}}.
	\end{align}
	If $|x_h| \gtrsim t 2^q$,  there holds that
	\begin{equation}
		|A_1| \lesssim t^{-1} \langle |x_h| \rangle^{-\frac12} \|F\|_{X^{0,3}_{\beta'}} \lesssim t^{-\frac32} 2^{-\frac{q}{2}}\|F\|_{X^{0,3}_{\beta}}.
	\end{equation}
	If  $|x_h| \ll t 2^q$, there holds that
	\begin{align}
		|A_1| &\lesssim t^{-1}  \int_{\R} \varphi(2^{-q} x_3^{-1} \lmd) \lmd^{-2} d\lmd \, \|F\|_{X^{0,3}_{\beta}} \nonumber \\
		&\lesssim t^{-2} 2^{-q} \, \|F\|_{X^{0,3}_{\beta}} \lesssim t^{-\frac32} 2^{-\frac{q}{2}} \, \|F\|_{X^{0,3}_{\beta}}.
	\end{align}
	The estimate for $A_2$ is similar. The same argument also applies to $e^{ i t \Lmd} P^{\hp,q} f^I$.
	
	Next, consider the case $|x_3| \notin [2^{m-4}, 3\times 2^m]$. As in \eqref{eq:427-001} and \eqref{eq:31-1}, we can write
	\begin{align}
		&\quad \ 2\pi (\F_h e^{i t \Lmd} P^{\hp,q} f^{I})(\xi_h, x_3) \nonumber \\
		&=  \int_{\R} e^{i(t\Lmd +\xi_3 x_3)} \chi^{\hp, q}\, \wh{f^{I}}(\xi) d\xi_3 \nonumber \\
		&=  \int_{\R} (Z_{<m-7} \F_h f)(\xi_h, y_3)dy_3  \int_{\R} \r \chi^{\hp,q}(\z)  e^{i(t\Lmd + \r\z(x_3-y_3))} d\z.
	\end{align}
	Note that for $y_3$ in the support of $Z_{<m-7}$ and for $\xi$ in the support of $\wt{\chi}^{\hp}(\xi) \varphi(|\xi_h|)$, we have
	\begin{equation}
		|\partial_{\z} (t \Lmd +  \r \z  (x_3 - y_3))| = \left|\frac{t}{(1+\z^2)^{\frac32}} + (x_3-y_3) \r\right| \gtrsim t.
	\end{equation}
	Hence, repeated integration by parts in $\z$ gives
	\begin{align}
		|(\F_h e^{i t \Lmd} P^{\hp,q} f^{I})(\xi_h, x_3)| &\lesssim_K 2^q (t 2^{q})^{-K} \int_\R |\F_h f(\xi_h, y_3)| dy_3 \nonumber \\
		&\lesssim \min\{t^{-1}, t^{-2} 2^{-q}\} \left(\int_\R |\F_h f(\xi_h, y_3)|^2 \langle y_3 \rangle^2 dy_3\right)^\frac12.
	\end{align}
	Applying $\F_h^{-1}$ in the same way as in \eqref{eq:427-3}, we obtain the desired estimate \eqref{eq:427-002}.
\end{proof}

\subsection{Wave packets with horizontal frequencies} \label{sec:52}


Throughout the rest of the paper, we fix two positive real numbers $\alpha$, $\gamma$ satisfying the constraints
\begin{equation} \label{eq:614-aa1}
	\begin{cases}
		\frac{\gamma}{2}< \alpha< 1-\gamma, 		\vspace{5pt} \\
		\alpha < \frac12, \  \gamma > \frac{1}{2}.
	\end{cases}
\end{equation}
The first line in \eqref{eq:614-aa1} is required by the  linear analysis of this section, while the second line comes from the nonlinear analysis in Section \ref{sec:8}. For definiteness, we take 
\begin{equation} \label{eq:802-001}
	\alpha = 0.35, \quad \gamma = 0.6.
\end{equation}

\medskip

Consider a smooth partition of unity for $\R$ given by
\begin{equation}
	1 \equiv \sum_{j \in \Z} \phi(\cdot - j)
\end{equation}
which satisfies
\begin{equation}
	\supp \, \phi \subset (-0.6, 0.6), \quad \supp \, \nabla \phi \subset (-0.6, -0.4) \cup (0.4, 0.6).
\end{equation}
Then, we denote the rescaled and translated versions by
\begin{equation}
	\phi^J_j(\cdot) = \phi(4^{J}\cdot -j),
\end{equation}
and define a family of Fourier  localization operators $Q^J_j$ with $5 \le J \le m \gamma /2, |j| \le 10 \cdot 4^J$ as
\begin{equation}
	(\F Q^J_j f)(\xi) := \phi^J_j(\ln \r) \wh{f}(\xi), \quad \r = |\xi_h|. 
\end{equation}
To state the next lemma, we need to define a family of cylinders in $\R^3$. Let
\begin{equation}
	\mathcal{C} :=  \left\{x \in \R^3 :  |x_h| < C_1 \,  t^{1-\alpha},  C_2^{-1}  t< -x_3 < C_2  t \right\},
\end{equation}
($C_1, C_2 $ are large absolute constants whose values are specified in Remark \ref{rem:413} below) and
\begin{equation}
	B^J_j(0) := \left\{x \in \mathcal{C} : \frac{t}{\r^J_{j,2}}<-x_3<\frac{t}{\r^J_{j,1}} \right\},
\end{equation}
where the numbers $\r^J_{j,1}, \r^J_{j,2}$ are defined as
\begin{equation}
	\r^J_{j,1} := \exp (4^{-J} (j-0.7)), \quad \r^J_{j,2} := \exp (4^{-J} (j+0.7)).
\end{equation}
For $x \notin B^J_j(0)$, we define a distance
\begin{equation}
	d^J_j(x) = 2^{-m+2J} \min \left\{ \left|x_3 + \frac{t}{\r_{j,1}^J}\right|, \left|x_3 + \frac{t}{\r_{j,2}^J}\right| \right\},
\end{equation}
and then define the sets
\begin{equation}
	B^J_j(1) := \left\{x \in \mathcal{C}\, \setminus  B^J_j(0) : d^J_j(x) < 4   \right\},
\end{equation}
\begin{equation}
	B^J_j(l) := \left\{x \in \mathcal{C} \,\setminus  B^J_j(0) : 2^l \le d^J_j(x)  < 2^{l+1}  \right\}, \  l = 2,3,4, \cdots.
\end{equation}
For $l \gg m-2J$, we have $B^J_j(l) = \emptyset$. Note that the family of sets $\{B^J_j(l)\}_{l\in \mathbb{N}}$ are mutually disjoint, and 
\begin{equation}
	\mathcal{C} = \bigcup_{0 \le l \lesssim m-2J} B^J_j(l).
\end{equation}
The key property of the above construction is that, for   $x \in B^J_j(l)$ with $l \ge 1$, there holds that
\begin{equation} \label{eq:414-3}
	\mathrm{dist} \left(x_3, \left\{t  \r^{-1}: \ln \r \in \mathrm{supp} \, \phi_j^J \right\} \right)  \sim 2^{l+m-2J}.
\end{equation}

\smallskip

\begin{lemma} \label{lem:51-3}
	Suppose that $2^m \le t < 2^{m+1} (m \ge 1) $ , $5 \le J \le \frac{m \gamma}{2}$, $|j| \le 10 \times 4^J$ and  $f = P^\hp_{[-10, 10]} F$ with $F \in X^{0,3}_\beta$. Denote $f^J_j = Q^J_j P^{\hp, \le -m\alpha} f$.	Then the following statements hold.
	\begin{enumerate}[(i)]
		\item 
		For $|x_h| \lesssim t^{1-\alpha}$ and $|x_3| \sim t$,
		\begin{equation}
			|e^{i t \Lmd} f^J_j(x)| \lesssim t^{-1} \langle |x_h| \rangle^{-1/2} \|F\|_{X^{0,3}_{\beta}}.
		\end{equation}
		
		\item For $|x_h| \gg t^{1-\alpha}$ and $|x_3| \lesssim t$,
		\begin{equation}
			|e^{i t \Lmd} f^J_j(x)| \lesssim t^{-3} \|F\|_{X_0^{0,100}}, \quad 			|e^{i t \Lmd} P^{\hp, \le -m\alpha} f (x)| \lesssim t^{-3} \|F\|_{X_0^{0,10}}.
		\end{equation}

%

		\item
		For each $l \ge 1$ and $K \ge 1$,
		\begin{equation} \label{eq:525-a1}
			\|e^{i t \Lmd} f^J_j\|_{L^2(B^J_{j}(l))} + 2^J \|e^{i t \Lmd} f^J_j\|_{L^2_{x_3} L^{\infty}_{x_h}(B^J_{j}(l))} \lesssim  \|Z_{\sim l +m- 2J }Q^J_j f\|_{L^2}  + C_K t^{-K} \|F\|_{X^0_0}.
		\end{equation}
	\end{enumerate}

\end{lemma}

\begin{remark} \label{rem:413}
	We will choose $C_1, C_2$ in the definition of $\mathcal{C}$ in the following way. Let $f$ be defined as in Lemma \ref{lem:51-3}. Take $C_2$ sufficiently large such that outside the slab $\{C_2^{-1} t \le -x_3 \le C_2 t\}$, the $L^2$ and $L^\infty$ norms of $e^{i t \Lmd} f$  and $e^{i t \Lmd} P^{\hp, \le - m \alpha} f$ decay with rate $t^{-1-\beta}$ as proved in Lemma \ref{lem:slab} (ii)--(iii) (the proof there also applies to $e^{i t \Lmd} P^{\hp, \le - m \alpha} f$ with direct adaptations). Then, we take $C_1$ sufficiently large such that (ii) applies when $|x_h| \ge C_1 t^{1-\alpha}$ and $C_2^{-1} t \le x_3 \le C_2 t$ holds.
\end{remark}

\begin{proof}
	By definition, we can write, with $\Psi$ given by \eqref{eq:Psi-def},
	\begin{equation} \label{eq:F-rep}
		e^{i t \Lmd} f^J_j(x) = \frac{1}{(2 \pi)^3}\int_{\R^+\times \R \times \mathbb{S}^1} e^{i \Psi} \phi^J_j(\ln \r) \chi^{\hp, \le -m\alpha}(\xi) \wh{f}(\xi) \r^2 d\r d\z d\th.
	\end{equation}
	First, we consider the case $|x_h| \lesssim t^{1-\alpha}, |x_3| \lesssim t$.  As in \eqref{eq:F-decomp}, here we can write
	\begin{equation} \label{eq:414-2}
		e^{i t \Lmd} f^J_j(x) = \sum_{\sigma = 1,2} \int_{\R^+\times \R} e^{i \Psi_\sigma} \phi^J_j(\ln \r) \chi^{\hp, \le -m\alpha}(\xi) g_\sigma(\r,\z) \r^2 d\r d\z =: \sum_{\sigma = 1,2} A^J_{j, \sigma},
	\end{equation}
	in which $\Psi_{1,2}$ and $g_{1,2}$ are given by \eqref{eq:41-1}--\eqref{eq:41-4}. We shall focus on $A^J_{j,1}$ as the estimate for $A^J_{j,2}$ is similar. Using \eqref{eq:41-5}--\eqref{eq:41-8}, we integrate by parts in $\r$ twice to get, for any $\z \lesssim 1$,
	 \begin{align} \label{eq:41-11}
	 	& \quad \left| \int_{\R^+} e^{i \Psi_{1}} \phi^J_j(\ln \r) g_{1}(\r, \z) \r^2 d\r \right| \nonumber \\
	 	&=   \left| \int_{\R^+} \frac{e^{i \Psi_1}}{(\partial_\r \Psi_1)^2} \partial_\r^2 (g_1 \phi^J_j(\ln \r) \r^2) d\r \right|  \le (x_3 \z + |x_h|)^{-2}  \int_{\R^+} \left| \partial_\r^2 (g_1 \phi^J_j(\ln \r)  \r^2) \right| d\r \nonumber \\
	 	&\lesssim (x_3 \z + |x_h|)^{-2} \langle |x_h| \rangle^{-\frac12} \left( 4^{2J} \int_{\mathrm{supp} (\phi_j^J \circ \ln)} |g_1| d\r +  4^{J} \int_{\R^+} |\partial_\r g_1| d\r +  \int_{\R^+} |\partial_\r^2 g_1| d\r   \right) \nonumber \\
	 	&\lesssim (x_3 \z + |x_h|)^{-2} \langle |x_h| \rangle^{-\frac12} 4^{J} \|F\|_{X^{0,3}_{\beta}}.
	 \end{align}
 	 In the penultimate line, we used \eqref{eq:41-5} and $|\partial_\r^2 \phi^J_j(\ln \r)| \lesssim 4^{2J}$, while in the last line, we used Sobolev embedding in  the $\r$ and $\z$ directions. On the other hand, a crude bound gives
	 \begin{align} \label{eq:414-1}
	 	\left| \int_{\R^+} e^{i \Psi_{1}} \phi^J_j(\ln \r) g_{1}(\r, \z) \r^2 d\r \right| &\lesssim  \int_{\R^+} \phi^J_j(\ln \r)	\left|  g_{1}(\r, \z) \right| d\r \nonumber \\
	 	&\lesssim  \langle |x_h| \rangle^{-\frac12} 4^{-J} \|F\|_{X^{0,3}_{\beta}}.
	 \end{align}
	 Integration of \eqref{eq:41-11} and \eqref{eq:414-1} in $\z$ gives
	 \begin{align}
	 	|A^J_{j,1}| &\lesssim \langle |x_h| \rangle^{-\frac12} \int_{\R} \min \left\{ (x_3 \z + |x_h|)^{-2} 4^J, 4^{-J} \right\} d\z \   \|F\|_{X^{0,3}_{\beta}} \\
	 	&\lesssim \langle |x_h| \rangle^{-\frac12} |x_3|^{-1} \|F\|_{X^{0,3}_{\beta}} \\
	 	&\lesssim \langle |x_h| \rangle^{-\frac12} t^{-1} \|F\|_{X^{0,3}_{\beta}}.
	 \end{align}
	 This finishes the proof of (i).
	 
	 Next,  consider the case $|x_h| \gg t^{1-\alpha}$, $|x_3| \lesssim t$. Note that for $\zeta \lesssim  t^{-\alpha}$, there holds that
	 \begin{equation}
	 	|\partial_\r \Psi_1| = |x_3 \zeta + |x_h|| \gtrsim t^{1-\alpha},
	 \end{equation}
	 which enables us to apply repeated integration by parts in $\r$ to \eqref{eq:414-2}. Note that $\partial_\rho$ landing on $\phi^J_j(\ln \r)$ would give a factor of $4^J$ which is bounded by $t^\gamma$ ($\ll t^{1-\alpha}$ if $t \gg 1$). As a result, we obtain
	 \begin{equation}
	 	|A^J_{j,1}| \lesssim t^{-3} \|F\|_{X^{0,100}_0}.
	 \end{equation}
	 The estimate for $A^J_{j,2}$ and for $e^{i t \Lmd} P^{\hp, \le -m \alpha} f$ is similar. This proves (ii).
	 
	 Finally, we consider the case $x \in B_j^J(l)$ with $l \ge 1$.  We make the decomposition
	 \begin{align}
	 	f^J_j &= {P}^{\hp, \le -m\alpha } Z_{[l +m- 2J-C,l +m- 2J+C]} Q^J_{j} f \nonumber \\
	 	&\quad + {P}^{\hp, \le -m\alpha } Z_{[l +m- 2J-C,l +m- 2J+C]^c} Q^J_{j} f \nonumber \\
	 	&=: h_1 + h_2.
	 \end{align}
	 Correspondingly, we decompose
	 \begin{equation}
	 	(\F_h e^{i t \Lmd} f^J_j)(\xi_h, x_3) = I_1 + I_2 := \frac{1}{2 \pi} \int_\R e^{i (t \Lmd + \xi_3 x_3)} \left\{\wh{h_1}(\xi) + \wh{h_2}(\xi)\right\} d\xi_3. 
	 \end{equation}
	 Note that $I_1(\xi_h, x_3)$ and $I_2(\xi_h, x_3)$ are both supported on 
	 \begin{equation} \label{eq:414-4}
	 	\r = |\xi_h| \in \left[ \exp(4^{-J}(i-0.6)), \exp(4^{-J}(i+0.6))\right].
	 \end{equation}
	  For $I_2$ we can write
	  \begin{align} \label{eq:44-1}
	 	2 \pi I_2  &= \int_\R e^{i (t \Lmd + \xi_3 x_3)} {\chi}^{\hp, \le -m\alpha}(\xi) (\F_3 Z_{[l +m- 2J-C,l +m- 2J+C]^c} \F_h Q^J_j f)(\xi)  \r d\z \nonumber \\
	 	&=\int_\R \r {\chi}^{\hp, \le -m\alpha}(\xi) \int_\R  (Z_{[l +m- 2J-C,l +m- 2J+C]^c} \F_h Q^J_j f)(\xi_h, y_3) e^{i (t \Lmd +  \r \z  (x_3 - y_3))}dy_3  d\z.
	 \end{align}
 	For $|\Lambda| \le 2^{-m \alpha+1}$, we have
 	\begin{equation}
 		\partial_\z \Lambda  = \frac{1}{(1+\z^2)^\frac32} = 1 + O(t^{-2\alpha}) = 1+o(t^{-\gamma}),
 	\end{equation}
	 hence if $y_3$ is in the support of $Z_{[l +m- 2J -C, l +m- 2J +C]^c}$ (with the absolute constant $C$ sufficiently large) and $\r$ satisfies \eqref{eq:414-4}, there holds that
	 \begin{equation}
	 	|\partial_\z (t \Lambda + \r \z (x_3-y_3))| \gtrsim t 4^{-J} \gtrsim t^{1-\gamma},
	 \end{equation}
 	 which enables us to apply repeated integration by parts in $\z$ to \eqref{eq:44-1}. Note that $\partial_\z$ landing on ${\chi}^{\hp, \le -m\alpha+1}$ gives a factor bounded by $t^\alpha \ll t^{1-\gamma}$. As a result, for any $K \ge 1$, there holds that
	 \begin{align}
	 	|I_2(\xi_h, x_3)| &\lesssim_K t^{-K} \int_\R |\F_h Q^J_j f(\xi_h, y_3) | dy_3 \\
	 	&\lesssim t^{-K} \left(\int_\R |\F_h Q^J_j f(\xi_h, y_3) |^2 \langle y_3 \rangle^2 dy_3 \right)^\frac12.
	 \end{align}
	Hence, upon taking $\F_h^{-1}$ and integrating in the $x_3$ direction, we get
	 \begin{equation}
	 	\|e^{i t \Lmd} h_2\|_{L^2\cap L^2_{x_3}L^\infty_{x_h}(B^J_j(l))} \lesssim_K t^{-K} \|Q^J_j f\|_{X^{0}_0},
	 \end{equation}
	 which is acceptable for the estimates in (ii). It remains to consider $e^{i t \Lmd} h_1$. Using a crude bound, we have
	 \begin{align}
	 	\|e^{i t \Lmd} h_1\|_{L^2(B^J_j(l))} \lesssim \| h_1\|_{L^2}
	 	\lesssim \|Z_{[l +m- 2J-C,l +m- 2J+C]} Q^J_j f\|_{L^2}.
	 \end{align}
	 Using Young's inquality for $\F_h^{-1}$ and \eqref{eq:414-4}, we also deduce that
	 \begin{align}
	  	\|e^{i t \Lmd} h_1\|_{L^2_{x_3} L^\infty_{x_h}  (B^J_j(l))} &\lesssim  \| \F_h h_1 \|_{L^2_{x_3} L^1_{\xi_h}}
	  	\lesssim 2^{-J} \| h_1\|_{L^2} \nonumber \\
	  	&\lesssim 2^{-J} \|Z_{[l +m- 2J-C,l +m- 2J+C]} Q^J_j f\|_{L^2}.
	 \end{align}


\end{proof}


\subsection{Dispersive estimates for $\bigvp$-pieces}

\begin{lemma} \label{lem:vp-disp}
	Let $f = P_k^{\vp} F$ with $F \in {Y}^{0,3}_\beta$, $\mu \in \{+,-\}$, $p \in \Z_-$ and $2^m \le t < 2^{m+1} (m \ge 1)$, then we have
	\begin{enumerate}[(i)]
		\item If $|x_h| \sim t 2^{p-k}$, 
		\begin{equation} \label{eq:426-03}
			|e^{\mu i  t \Lmd}P^{\vp, p} f| \lesssim \min\{t^{-\frac32} 2^{-p}, 2^{2p}\} 2^{\frac32 k} \|F\|_{Y^{0,3}_{\beta}}.
		\end{equation}
		\item Decompose $f = f^I + f^{II}$ with
		\begin{equation}
			f^I := \begin{cases}
				H_{\le m-k+p-C} f, \ \ \mbox{if} \ \ 2p \ge -m, \\
				f, \ \ \mbox{if} \ \ 2p < -m,
			\end{cases}
		\end{equation}
	and
	\begin{equation}
		f^{II} := \begin{cases}
			H_{> m-k+p-C} f, \ \ \mbox{if} \ \ 2p \ge -m, \\
			0, \ \ \mbox{if} \ \ 2p < -m,
		\end{cases}
	\end{equation}
		where $C$ is a sufficiently large absolute constant, then there holds  that
		\begin{equation} \label{eq:426-01}
			\|e^{\mu i  t \Lmd}P^{\vp, p} f^I\|_{L^\infty} \lesssim  \min\{t^{-\frac32} 2^{-p}, 2^{2p}\} 2^{\frac32 k} \|F\|_{{Y}^{0,3}_{\beta}},
		\end{equation}
		and
		\begin{equation} \label{eq:426-02}
			\|e^{\mu i  t \Lmd} P^{\vp, p} f^{II}\|_{L^2} \lesssim (t 2^p)^{-(1+\beta)} \|F\|_{{Y}^0_\beta},
		\end{equation}
		\begin{equation} \label{eq:0525-1}
			\|e^{\mu i  t \Lmd} P^{\vp, p} f^{II}\|_{L^\infty} \lesssim t^{-(1+\beta)} 2^{-\beta p + \frac32 k} \|F\|_{{Y}^0_\beta}.
		\end{equation}
		\item 
		\begin{equation} 
			\sup_{p \in \Z_-} \|e^{\mu i t \Lmd} P^{\vp, p} f\|_{L^\infty} + \|e^{\mu i t \Lmd} f\|_{L^\infty} \lesssim t^{-1} 2^{\frac32 k} \|F\|_{{Y}^{0,3}_{\beta}}. \label{eq:59-a1}
		\end{equation}
	\end{enumerate}
\end{lemma}

\begin{proof}
	In this proof we use the local coordinate system $(\vr, z, \theta)$ defined in Section \ref{sec:22}, in particular, the partial derivatives with respect to $\vr, z, \theta$ are given by \eqref{eq:425-1}--\eqref{eq:425-3}. By reflection and scaling, without loss of generality, we may assume that  $\mu = +$, $k = 0$.  Moreover, we assume that $\wh{f}$ is supported on $z > 0$ as the opposite case can be similarly treated (for instance, one may consider negative times).

	The estimates \eqref{eq:426-02}--\eqref{eq:0525-1} follow from the crude bounds
	\begin{equation}
		\|e^{i t \Lmd} P^{\vp, p} f^{II}\|_{L^2} \lesssim \|f^{II}\|_{L^2} \lesssim (t 2^p)^{-(1+\beta)} \|F\|_{Y^{0}_\beta}
	\end{equation}
	and
	\begin{equation}
		\|e^{i t \Lmd} P^{\vp, p} f^{II}\|_{L^\infty} \lesssim 2^{p} \|f^{II}\|_{L^2}.
	\end{equation}
	The estimate \eqref{eq:59-a1} is a corollary of \eqref{eq:426-01} and \eqref{eq:0525-1} via summation in $p$. Hence, it suffices to consider \eqref{eq:426-03} and \eqref{eq:426-01}. If $2p < -m$, then we have a crude bound
	\begin{equation} \label{eq:0525-11}
		\|e^{it\Lmd} P^{\vp,p} f\|_{L^\infty} \lesssim 2^{2p} \| \wh{f}\|_{L^\infty} \lesssim 2^{2p} \|F\|_{Y^{0,3}_{\beta}},
	\end{equation}
	which is sufficient for \eqref{eq:426-03} and \eqref{eq:426-01}. Hence, in the sequel we shall assume that $2p \ge -m$.
	
	By definition, there holds that
	\begin{equation} \label{eq:10-1}
		e^{it\Lmd} f(x) = \frac{1}{(2\pi)^3} \int_{\R^+\times\R^+\times \mathbb{S}^1} e^{i\Psi} \chi^{\vp, p}(\vr) \wh{f}(\xi) \vr z^2 d\vr dz d\theta,
	\end{equation}
	where we denote
	\begin{equation}
		\Psi = \frac{t}{\sqrt{1+\vr^2}} + x_3 z + x_h \cdot \xi_h.
	\end{equation}
    To obtain decay estimates,  we use integration by parts with respect to $z, \th$ or $\vr$ depending on the location of $x$. 
    
	\medskip
	\emph{Case 1.} $|x_h| \sim t \, 2^p$.
	\medskip
	
	Let $\chi_1(\theta)$ be a smooth bump function on the unit circle, with support contained in $[-\frac{\pi}{2}, \frac{\pi}{2}]$ and is equal to $1$ on $[-\frac{\pi}{4}, \frac{\pi}{4}]$, and then let $\chi_2(\theta)  =  1-\chi_1(\theta)$. Integrating \eqref{eq:10-1} in $\theta$, we get
	\begin{equation}
		e^{it\Lmd} f(x) = A_1+A_2 := \sum_{j = 1,2} \int_{\R^+ \times \R^+} e^{i\Psi_{j}}\chi^{\vp, p}(\vr) g_j(\vr, z) \vr z^2 d\vr dz,
	\end{equation}
	where we denote that
	\begin{equation}
		\Psi_1 = \frac{t}{\sqrt{1+\vr^2}} + x_3 z + |x_h| \vr z,
	\end{equation}
	\begin{equation}
		\Psi_2 = \frac{t}{\sqrt{1+\vr^2}} + x_3 z - |x_h| \vr z,
	\end{equation}
and
	\begin{equation}
		g_1 = \frac{1}{(2\pi)^3} \int_0^{2\pi} e^{i|x_h|\vr z (\cos(\theta-\theta_x) - 1) }\chi_1(\theta - \theta_x) \wh{f}(z, \vr, \th)  d\th,
	\end{equation}
	\begin{equation}
		g_2 = \frac{1}{(2\pi)^3} \int_0^{2\pi} e^{i|x_h|\vr z (\cos(\theta-\theta_x) + 1) }\chi_2(\theta - \theta_x) \wh{f}(z, \vr, \th)  d\th.
	\end{equation}
	By the standard theory for 1D oscillatory integrals, we have for $j = 1,2$,
	\begin{equation}
		|\partial_{z}^n g_j(\vr,z)| \lesssim  \langle |x_h| \vr \rangle^{-\frac12} \int_0^{2\pi} |\{S, \Omega\}^{\le n+1} \wh{f}| d\th,\quad  n=1,2,3, \cdots.
	\end{equation}
	We shall focus on $A_1$, as the estimate for $A_2$ is similar. Using that
	\begin{equation}
		\partial_{z} \Psi_1 = x_3 + |x_h| \vr, \quad \partial_{z}^2 \Psi_1 = 0,
	\end{equation}
	and repeated integration by parts in $z$, we get
	\begin{align} \label{eq:425-5}
		\left| \int_{\R^+} e^{i\Psi_1}   g_1(\vr, z) z^2 dz \right| &\lesssim \left|\int e^{i\Psi_1} (x_3 + |x_h| \vr)^{-2} \partial_{z}^2 (g_1 z^2)  dz \right| \nonumber \\
		&\lesssim (x_3 + |x_h| \vr)^{-2} \langle |x_h| \vr \rangle^{-\frac12} \int |\{S, \Omega\}^{\le 3} \wh{f}| dz d\theta.
	\end{align}
    On the other hand, a crude bound gives
	\begin{align} \label{eq:425-4}
		\left| \int_{\R^+} e^{i\Psi_1}  g_1(\vr, z) z^2 dz \right| \lesssim \int |g_1| dz \lesssim  \langle |x_h| \vr \rangle^{-\frac12} \int | \Omega \wh{f}| dz d\theta.
	\end{align}
	Integrating \eqref{eq:425-5}--\eqref{eq:425-4} in $\vr$ leads to
	\begin{align}
		|A_1| &\lesssim \int_{\vr \sim 2^p} \min \{(x_3 + |x_h| \vr)^{-2}, 1\} \, \langle |x_h| \vr \rangle^{-\frac12} \vr d\vr \, \|F\|_{Y^{0, 3}_{\beta}} \nonumber \\ 
		&\lesssim t^{-2} 2^{-2p} \int_{\lmd \sim t 2^{2p}} \langle x_3 + \lmd\rangle^{-2} \,  \lmd^\frac12 d\lmd \, \|F\|_{Y^{0, 3}_{\beta}} \nonumber \\
		&\lesssim t^{-\frac32} 2^{-p}\, \|F\|_{Y^{0, 3}_{\beta}}.
	\end{align}
	Note that the argument in this case also works with $f$ replaced by $f^I$.
	
	\medskip
	\emph{Case 2.} $|x_h| \nsim t2^p$.
	\medskip
	
	 Using $\F = \F_h \F_3$ and $[\F_3, H_l] = 0$, we deduce that
	\begin{equation} \label{eq:425-6}
		\F_3 e^{it\Lmd} f^I(x_h, z) = \frac{1}{(2\pi)^2} \int \psi(2^{-m-p+C} |y_h|) {\chi}^{\vp,p}(\vr)   e^{i \wt{\Psi}} \F_3f(y_h,z) \vr z^2 d\vr d\th dy_h,
	\end{equation}
    in which we denote that
	\begin{equation}
		\wt{\Psi}  = \frac{t}{\sqrt{1+\vr^2}} + (x_h - y_h) \cdot \xi_h.
	\end{equation}
	We perform the $\th$ integration in \eqref{eq:425-6} and obtain that
	\begin{align}
		\F_3 e^{it\Lmd} f^I &= I_1+ I_2 \nonumber \\
		&:= \sum_{j= 1,2} \int \psi(2^{-m-p+C} |y_h|) \F_3f(y_h,z) dy_h \int e^{i \wt{\Psi}_j} {\chi}^{\vp,p}(\vr) G_j(|x_h - y_h| \vr z)  \vr z^2 d\vr  
	\end{align}
	where we denote
	\begin{equation}
		\wt{\Psi}_1 = \frac{t}{\sqrt{1+\vr^2}} + |x_h - y_h| \vr z, \quad 	\wt{\Psi}_2 = \frac{t}{\sqrt{1+\vr^2}} - |x_h - y_h| \vr z,
	\end{equation}
	and
	\begin{equation}
		G_1(r) = \int   e^{i r (\cos \theta -1)} \chi_1(\theta)  d\theta, \quad G_2(r) = \int   e^{i r (\cos \th  +1)} \chi_2(\th)  d\theta.  
	\end{equation}
	Standard theory for 1D oscillatory integrals gives that, for $n =0,1,2,\cdots$,
	\begin{equation}
		|\partial_r^n G_{1,2}(r)| \lesssim_n \langle r \rangle^{-\frac12 - n}.
	\end{equation}
	We shall focus on $I_1$ as $I_2$ can be treated in a similar way. Due to
	\begin{equation} \label{eq:426-1}
		\vr \sim 2^p, \quad z \sim 1, \quad |x_h| \nsim t 2^p, \quad |y_h| \ll t 2^p,
	\end{equation}
    there holds that
	\begin{equation} \label{eq:426-2}
		|\partial_\vr \wt{\Psi}_1| = \left|-\frac{t \vr}{(1+\vr^2)^\frac32} + |x_h-y_h| z\right| \gtrsim t 2^p + |x_h|.
	\end{equation}
	Moreover, notice that under the constraints \eqref{eq:426-1}, there holds that
	\begin{equation} \label{eq:426-3}
		|\partial_{\vr}^n \wt{\Psi}_1| \lesssim_n t, \quad  n = 2,3,\cdots,
	\end{equation}
	and
	\begin{equation} \label{eq:426-4}
		|\partial_\vr^n {\chi}^{\vp, p}| + |\partial_\vr^n G_j(|x_h - y_h|\vr z)|\lesssim (2^{-p})^n, \quad n=1,2,3,\cdots.
	\end{equation}
	Using the integration by parts formula (for $h$ compactly supported in $\R_+$)
	\begin{equation}
		-i \int e^{i\wt{\Psi}_1} h(\vr) d\vr = \int e^{i\wt{\Psi}_1} \partial_\vr \bigg(\frac{h(\vr)}{\partial_\vr \wt{\Psi}_1}\bigg) d\vr
	\end{equation}
	and \eqref{eq:426-2}--\eqref{eq:426-4}, we see that repeated integration by parts in $\vr$ gives 
	\begin{align}
		|I_{1}(x_h, z)| &\lesssim_K 2^{2p} (t 2^{2p})^{-K}  \int |\F_3 f(y_h, z)| dy_h, \quad K=1,2,3,\cdots,
	\end{align}
	and consequently,
	\begin{equation}
		\|\F_3^{-1} I_{1}\|_{L^\infty} \lesssim \|I_{1}\|_{L^\infty_{x_h} L^2_z} \lesssim_K 2^{2p} (t 2^{2p})^{-K} \|F\|_{Y^0_{\beta}},
	\end{equation}
	which is an acceptable contribution to \eqref{eq:426-01}.
\end{proof}

Combining \eqref{eq:59-a1} and Lemma \ref{lem:slab}(i)-(ii), we  deduce the following pointwise decay estimate.

\smallskip

\begin{corollary} \label{cor:59-1} There holds that
	\begin{equation}
		\|P_k e^{\pm i t \Lmd} F\|_{L^\infty} \lesssim  \langle t \rangle^{-1} 2^{\frac32 k} (\|F\|_{X^{0,3}_{\beta}} + \|F\|_{Y^{0,3}_{\beta}}).
	\end{equation}
	As a consequence, we also have
	\begin{equation}
		\| e^{\pm i t \Lmd} F\|_{L^\infty} \lesssim  \langle t \rangle^{-1}  (\|F\|_{X^{3,3}_{\beta}} + \|F\|_{Y^{3,3}_{\beta}}).
	\end{equation}
\end{corollary}

\subsection{Decay at far field}

\begin{lemma} \label{lem:largel}
	Suppose that  $f = P_k F$ with $k \in \Z, F \in X^{0,2}_{\beta}\cap {Y}^{0,2}_{\beta}$, and let $\mu \in \{+,-\}$, $l\ge m-k+2$, then there holds that
	\begin{equation} \label{eq:55-1}
		2^{-\frac32 k} \|\{Z_l, H_l\} e^{\mu i t \Lmd} f\|_{L^\infty}  + \|\{Z_l, H_l\} e^{\mu i t \Lmd} f\|_{L^2} \lesssim  2^{-(1+\beta)(l+k)} \|F\|_{X^{0,2}_{\beta}\cap {Y}^{0,2}_{\beta}}.
	\end{equation}
\end{lemma}

\begin{proof}
	By scaling and reflection with respect to the $x_h$ plane, we may take $k = 0$ and $\mu = +$. Note that $l \ge 0$, hence $Z_l^{(0)} = Z_l$ and $H_l^{(0)} = H_l$. In view of \eqref{eq:729-001}--\eqref{eq:729-002} and the standard theory of Besov spaces, we have, for any $l \ge 0$,
	\begin{equation} \label{eq:426-001}
		\|\{Z_l, H_l\} f\|_{L^2} \lesssim 2^{-(1+\beta)l} \|F\|_{X^{0,2}_{\beta}\cap Y^{0,2}_{\beta}}.
	\end{equation}
	By definition, we can write
	\begin{align} 
		\F_h (e^{ i t \Lmd} f) (\xi_h, x_3) &= \frac{1}{2\pi} \int_{\R^2}  e^{i t\Lmd + i \xi_3 (x_3-y_3)} \chi_0(\xi) \F_hf(\xi_h, y_3)  d\xi_3 dy_3 \nonumber \\
		&= \frac{1}{2\pi} \int_{\R^2} \mathbf{1}_{|y_3| < 2^{l-5}}  e^{i t\Lmd + i \xi_3 (x_3-y_3)} \chi_0(\xi) \F_hf(\xi_h, y_3)  d\xi_3 dy_3 \nonumber \\
		&\quad + \frac{1}{2\pi} \int_{\R^2} \mathbf{1}_{|y_3| > 2^{l-5}}  e^{i t\Lmd + i \xi_3 (x_3-y_3)} \chi_0(\xi) \F_hf(\xi_h, y_3)  d\xi_3 dy_3 \nonumber \\
		&=: I_1+I_2.
	\end{align}
	Note that for $l \ge m+2$, $|y_3| < 2^{l-5}$, $|x_3| \ge \frac45 \times 2^l$ and $\xi$ in the support of $\chi_0$, there holds that
	\begin{equation}
		|\partial_{\xi_3} (t \Lmd + \xi_3(x_3 - y_3))| = \left|t \frac{|\xi_h|^2}{|\xi|^3} + (x_3 - y_3) \right| \gtrsim |x_3|.
	\end{equation}
	Hence, repeated integration by parts in $\xi_3$ gives, for $|x_3| \ge \frac45 \times 2^l$,
	\begin{align} \label{eq:426-002}
		|I_1| \lesssim_K  |x_3|^{-K} \int_{\R} |\F_h f| dy_3, \quad K=1,2,3, \cdots.
	\end{align}
	For $I_2$, we can use \eqref{eq:426-001} to deduce that
	\begin{equation} \label{eq:426-003}
		\|I_2\|_{L^2} \lesssim 2^{-(1+\beta)l} \|F\|_{X^{0,2}_\beta\cap Y^{0,2}_\beta}.
	\end{equation}
	Note that both $I_1$ and $I_2$ are supported on $|\xi_h| \lesssim 1$. Hence, Applying $\F_h^{-1}$ to $I_{1,2}$ and using \eqref{eq:426-002}--\eqref{eq:426-003}, we get the desired estimates for $Z_l e^{ i t \Lmd} f$. 
	
	The estimates for $H_l e^{\mu i t \Lmd} f$ is similar as the above analysis can be adapted to $x_1$ and $x_2$ directions directly, hence we omit the proof.
	
	
\end{proof}

\section{Time decay of $\partial_t S^a \bar{\Omega}^b \f_\pm$}	 \label{sec:dt}

Throughout the rest of the paper, we take $\beta = \frac{2}{3} \beta_0$ and $\beta' = \frac23 \beta$, with $0<\beta_0 \ll 1$. The dispersive estimates of Section \ref{sec:4} will be applied with parameter either $\beta$ or $\beta'$ depending on our need.

First, we prove the general decay rate $t^{-\frac32 + }$ for $\partial_t S^a \bar{\Omega}^b \f_\pm$ in $L^2$, see Lemma \ref{lem:528-01}. Then, in Lemma \ref{lem:428-1} and Lemma \ref{lem:51-1}, we give two results on the improved decay rates at far field based on Lemma \ref{lem:largel}. Note that in this section we  rely only on the physical product structure of the nonlinearity, without invoking any null type condition.

\begin{lemma} \label{lem:528-01}
	Assuming the bootstrap assumption \eqref{eq:BA} and let  $a+b \le 20$, $\mu \in \{+,-\}$, $k \in \Z$, then there holds that
	\begin{equation}
		\|\partial_t P_k S^a \bar{\Omega}^b  \f_\mu\|_{L^2} \lesssim 2^{k-30 k^+} t^{-\frac32 + C \varepsilon_1} \ln t \,\, \varepsilon_1^2.
	\end{equation}
\end{lemma}
\begin{proof}
	 Let $f$ be any component of $P_k S^a \bar{\Omega}^b  \f_\mu$. By \eqref{eq:510-1} and \eqref{eq:510-51}, we can write $\partial_t f$ as a linear combination of bilinear terms of the form
	 \begin{equation}
	 	|\nabla| P_k e^{-\mu i t \Lmd} D \{(e^{\mu_1 i t \Lmd} f_1) \cdot (e^{\mu_2 i t \Lmd} f_2)\}, 
	 \end{equation}
 	where for $i = 1,2$, $f_i = P^{\iota_i}_{k_i} F_i$ with $\iota_{1,2}\in \{\hp, \vp\}$, and $F_i$ is some component of  $ S^{a_i} \bar{\Omega}^{b_i}  \f_{\mu_i}$ with $a_1+a_2 \le a$, $b_1+b_2 \le b$. We will consider the case $\mu = \mu_1 = \mu_2 = +$, as the signs play no role in the proof here. Suppose that $t \in [2^m, 2^{m+1})$ with $m \ge 1$.

 	Let $I = (e^{ i t \Lmd} f_1) \cdot (e^{ i t \Lmd} f_2)$ and notice that
 	\begin{equation}
 		\||\nabla| P_k e^{- i t \Lmd} D I\|_{L^2} \lesssim 2^k \|I\|_{L^2}.
 	\end{equation}
 	By Corollary \ref{cor:59-1} and a direct $L^\infty$-$L^2$ estimate, we have
 	\begin{align} \label{eq:730-001}
 		\|I\|_{L^2} 		&\lesssim \|e^{i t \Lmd} f_1\|_{L^\infty} \|f_2\|_{L^2} \nonumber \\
 		&\lesssim 2^{\frac32 k_1 - 40 k_1^+ - 40 k_2^+} t^{-1} \|\f_\pm\|_{(X\cap Y)^{40,40}_{\beta'}}^2.
 	\end{align} 
    Similarly,  a direct $L^2$-$L^\infty$ estimate gives 
    \begin{equation} \label{eq:730-002}
    	\|I\|_{L^2} \lesssim 2^{\frac32 k_2 - 40 k_1^+- 40 k_2^+} t^{-1} \|\f_\pm\|_{(X\cap Y)^{40,40}_{\beta'}}^2. 
    \end{equation}
	By interpolation between \eqref{eq:BA} and \eqref{eq:slow-growth}, there holds that
	\begin{equation} \label{eq:730-003}
		\|\f_\pm\|_{(X\cap Y)^{40,40}_{\beta'}} \lesssim \langle t \rangle^{C \varepsilon_1} \varepsilon_1.
	\end{equation}
	Combining \eqref{eq:730-001}--\eqref{eq:730-003}, we have an acceptable contribution in the case when $k_1 < -3m$ or $k_2 < -3m$. Hence, in the sequel we may assume that $k_1, k_2 \ge -3m$.
 	
 	\smallskip
 	\emph{Case 1.} $(\iota_1, \iota_2) = (\bighp, \bighp)$.
 	
 	We further decompose $f_1, f_2$ into  $P^{\hp,q_{1}}$ and $P^{\hp,q_{2}}$ pieces, \emph{i.e.},
 	\begin{equation}\label{eq:428-0001}
 		f_j =  \sum_{ -m < q_j \le -1   } P^{\hp, q_j} f_j + P^{\hp, \le -m} f_j, \quad j = 1,2.
 	\end{equation}
 	For the last piece on the right of \eqref{eq:428-0001}, we think of $q_j = -m$. For simplicity, we shall write
 	 \begin{equation}
 	 	P^{\hp,(\le)q_j} := \begin{cases} P^{\hp, q_j}, \quad q_j > - m, \\
 	 		P^{\hp, \le -m}, \quad q_j = - m.
 	 \end{cases}
 	 \end{equation}
 	 In the case that $q_1 \le q_2$, an $L^2$-$L^\infty$ bound with the help of Lemma \ref{lem:hp-disp}  gives
 	\begin{align}
 		&\quad \ \|(e^{ i t \Lmd} P^{\hp, (\le)q_1} f_1) \cdot (e^{ i t \Lmd}  P^{\hp, (\le)q_2} f_2)\|_{L^2} \nonumber \\	
 		&\lesssim  \| P^{\hp, (\le)q_1} f_1\|_{L^2} \|e^{ i t \Lmd } P^{\hp, (\le)q_2} f_2^I\|_{L^\infty} + \| e^{  i t \Lmd } P^{\hp, (\le)q_1} f_1\|_{L^\infty} \| P^{\hp,(\le)q_2} f_2^{II}\|_{L^2} \nonumber \\
 		&\lesssim (2^{ \frac{q_1}{2} + \frac32k_2 -\frac{q_2}{2}} t^{-\frac32} + 2^{\frac32 k_1} t^{-2-\beta'}) \|F_1\|_{X^{0,3}_{\beta'}} \|F_2\|_{X^{0,3}_{\beta'}} \nonumber \\
 		&\lesssim 2^{ \frac32 \max\{k_1,k_2\} - 40k_1^+ - 40 k_2^+} t^{-\frac32 + C \varepsilon_1} \varepsilon_1^2. 
 	\end{align} 
 		(For $q_2 = -m$, we simply take $f_2^I = f_2$ and $f_2^{II} = 0$.) The case when $q_1 > q_2$ can be similarly treated via exchanging the roles of $f_1$ and $f_2$. Summing over the possible $q_{1}, q_2$ and $k_{1}, k_2$ gives an acceptable contribution.
 	
 	\medskip
 	\emph{Case 2.} $(\iota_1, \iota_2) = (\bighp, \bigvp)$. (The case $(\iota_1, \iota_2) = (\bigvp, \bighp)$ is similar.)
 	
 	We further decompose $f_1, f_2$ into 
 	\begin{equation}\label{eq:428-1}
 		f_1 =  \sum_{   -m < q_1 \le -1 } P^{\hp, q_1} f_1 + P^{\hp, \le -m} f_1, \quad f_2 =  \sum_{    -\frac{m}{2} < p_2 \le -1 } P^{\vp, p_2} f_2 + P^{\vp, \le -\frac{m}{2}} f_2.
 	\end{equation}
 	For the last piece on the right of \eqref{eq:428-1}, we think of $p_2 = -\frac{m}{2}$ and write
 		 \begin{equation}
 		P^{\vp,(\le)p_2} := \begin{cases} P^{\vp, p_2}, p_2 >  - \frac{m}{2}, \\
 			P^{\vp, \le - \frac{m}{2}}, p_2 = - \frac{m}{2}.
 		\end{cases}
 	\end{equation}
 	 For $q_1 \le 2p_2$, an $L^2$-$L^\infty$ bound with the help of Lemma \ref{lem:vp-disp}  gives
 	\begin{align}
 		&\quad \ \|(e^{ i t \Lmd} P^{\hp, (\le)q_1} f_1) \cdot (e^{ i t \Lmd}  P^{\vp, (\le)p_2} f_2)\|_{L^2} \nonumber\\	
 		&\lesssim  \| P^{\hp, (\le)q_1} f_1\|_{L^2} \|e^{ i t \Lmd } P^{\vp,(\le) p_2} f_2^{I}\|_{L^\infty} + \| e^{ i t \Lmd} P^{\hp, (\le) q_1} f_1\|_{L^\infty} \| P^{\vp,(\le) p_2} f_2^{II}\|_{L^2} \nonumber \\
 		&\lesssim ( 2^{ \frac{q_1}{2} + \frac32k_2 - p_2} t^{-\frac32}  +  2^{ \frac32 k_1 - (1+\beta')p_2} t^{-2-\beta'})  \|F_1\|_{X^{0,3}_{\beta'}} \|F_2\|_{Y^{0,3}_{\beta'}} \nonumber \\
 		&\lesssim 2^{ \frac32 \max\{k_1, k_2\} - 40 k_1^+ - 40 k_2^+} t^{-\frac32 + C\varepsilon_1} \varepsilon_1^2. 
 	\end{align} 
 	(For $p_2 = -\frac{m}{2}$, we simply take $f_2^I = f_2$ and $f_2^{II} = 0$.)  The case $q_1 > 2p_2$ can be similarly treated by exchanging the role of $f_1$ and $f_2$. Summing over possible $q_{1}, p_2$ and $k_{1,2}$ again gives an acceptable contribution.
 	 	
 		\medskip
 	\emph{Case 3.} $(\iota_1, \iota_2) = (\bigvp, \bigvp)$.
 	Here,	we decompose $f_1, f_2$ into 
 	\begin{equation}
 		f_j =  \sum_{  -\frac{m}{2}<p_j \le -1  } P^{\vp, p_j} f_j + P^{\vp, \le -\frac{m}{2}} f_j,\quad j = 1,2.
 	\end{equation}
 	Then, the proof is similar to Case 1, by comparing $p_1$ and $p_2$ and using Lemma \ref{lem:vp-disp} instead of Lemma \ref{lem:hp-disp}. 
\end{proof}

\medskip 

The following lemma implies the propagation of $X$ and $Y$ norms if the localization parameter $l$ is suitably large with respect to time.

\begin{lemma} \label{lem:428-1}
	Assuming the bootstrap assumption \eqref{eq:BA}, for $l+k \ge (1+\delta) m$, $t \in [2^m, 2^{m+1}) \, (m \ge 1)$ and $a+b \le 20$, there holds that
	\begin{equation} \label{eq:428-ddddd}
		\left\| \partial_t \left(Z_l P_k^\hp S^a \Omega^b \Fp_\pm\right) \right\|_{L^2} + \left\| \partial_t \left(H_l P_k^\vp S^a \bar{\Omega}^b \f_\pm \right) \right\|_{L^2} \lesssim 2^{- 30 k^+ - (1+\beta)(l+k)}  t^{-1-\frac{\delta}{4}}\varepsilon_1^2.
	\end{equation}
\end{lemma}

\begin{proof}
	Consider the first norm on the left of \eqref{eq:428-ddddd}. By \eqref{eq:510-3-001}, \eqref{eq:510-51} and \eqref{eq:808-i11}, $\partial_t \left(Z_l P_k^\hp S^a \Omega^b \Fp_\mu \right)$ can be written as a linear combination of
		 \begin{equation}
		Z_l |\nabla| P_k^\hp e^{-\mu i t \Lmd} D \{(e^{\mu_1 i t \Lmd} f_1) \cdot (e^{\mu_2 i t \Lmd} f_2)\},
	\end{equation}
	where $f_i = P_{k_i} F_i$ and $F_i$ stands for some component of $ S^{a_i} \bar{\Omega}^{b_i}\f_{\mu_i}$, $i=1,2$	with $a_1+a_2 \le a$, $b_1+b_2 \le b$ and $\mu_1, \mu_2 \in \{\pm\}$. Moreover, $D$ is generated by the operators $\frac{\partial_j}{|\nabla|}$ and $\frac{\partial_j}{|\nabla_h|}$, $i=1,2,3$.	We will consider the case $\mu = \mu_1 = \mu_2 = +$, as the signs play no role in the proof here. Without loss of generality, we assume that $k_1 \le k_2$.
	
	By interpolation between \eqref{eq:BA} and \eqref{eq:slow-growth}, there holds that
	\begin{equation} \label{eq:428-dd}
		 \|\f_\pm\|_{(X\cap Y)^{40,40}_{\beta-\delta_0}}\le  \langle t \rangle^{C \varepsilon_1} \varepsilon_1.
	\end{equation}
	  By Lemma \ref{lem:a1-1} and the assumption that $l+k \ge (1+\delta)m$, we have, for any $K\ge 1$,
	  \begin{align}
	  	&\ \quad \|Z_l |\nabla| P_k^\hp e^{- i t \Lmd} D Z_{[l-1, l+1]^c} \{(e^{i t \Lmd} f_1) \cdot (e^{i t \Lmd} f_2)\}\|_{L^2} \nonumber \\
	  	&\lesssim_K 2^{-K(l+k) +k} \|(e^{i t \Lmd} f_1) \cdot (e^{i t \Lmd} f_2)\|_{L^2} \nonumber \\
	  	&\lesssim 2^{-K (l+k) +k } \|e^{i t \Lmd} f_1\|_{L^\infty} \|e^{i t \Lmd} f_2\|_{L^2} \nonumber \\
	  	&\lesssim 2^{-2 (l+k)-(K-2)m + k + \frac{3k_1}{2} - 40 k_1^+ - 40 k_2^+} \langle t \rangle^{C \varepsilon_1} \varepsilon_1^2,
	  \end{align}
	which is an acceptable contribution. Hence, it suffices to effectively bound
	\begin{equation} \label{eq:428-01}
		2^k \|Z_{[l-1,l+1]}  \{(e^{i t\Lmd} f_1) \cdot   ( e^{ i t \Lmd} f_2) \}\|_{L^2}.
	\end{equation}

	\medskip
	\emph{Case 1.} $k-(l+k-m)/2 < k_1 \le k_2 \sim k$.
	\medskip
	
	In this case, $l+k_2 \ge l +k_1 \gg m$. Using Lemma \ref{lem:largel} and \eqref{eq:428-dd}, there holds that
	\begin{align} \label{eq:428-cc}
	\eqref{eq:428-01} &\lesssim 2^k \|Z_{\sim l}  e^{ i t\Lmd} f_1\|_{L^\infty} \|Z_{\sim l}   e^{ i t \Lmd} f_2\|_{L^2} \nonumber \\
	&\lesssim 2^{k+\frac32 k_1} 2^{-(1+\beta-\delta_0) (2l+k_1+k_2)} \|F_1\|_{(X\cap Y)^{0,2}_{\beta-\delta_0}} \|F_2\|_{(X\cap Y)^{0,2}_{\beta-\delta_0}} \nonumber  \\
	&\lesssim 2^{k+\frac32 k_1} 2^{-(1+\beta-\delta_0) (2(l+k)-(l+k-m)/2)} \|F_1\|_{(X\cap Y)^{0,2}_{\beta-\delta_0}} \|F_2\|_{(X\cap Y)^{0,2}_{\beta-\delta_0}} \nonumber  \\
	&\lesssim 2^{k+\frac32 k_1 - 40 k_1^+ - 40 k_2^+} 2^{-\frac32(1+\beta-\delta_0)(l+k) - \frac12 (1+\beta-\delta_0) m} \langle t \rangle^{C \varepsilon_1}\varepsilon_1^2 \nonumber \\
	 &\lesssim 2^{k+\frac32 k_1 - 40 k_1^+ - 40 k_2^+} 2^{-(1+\beta)(l+k) -  (1+\frac{\beta}{2}) m} \varepsilon_1^2.
	\end{align}
 Then, summing over such $k_1, k_2$ we obtain an acceptable contribution.
 
 	\medskip
 \emph{Case 2.} $k_1 \le k-(l+k-m)/2   < k_2 \sim k$.
 \medskip
 
 Here, Lemma \ref{lem:largel} can still be applied to $Z_{\sim l}  e^{ i t \Lmd} f_2$. On the other hand, we use the pointwise $t^{-1}$ decay guaranteed by Corollary \ref{cor:59-1} for  $  e^{ i t \Lmd} f_1$. 
\begin{align} \label{eq:428-bb}
			 \eqref{eq:428-01} &\lesssim 2^k \|e^{ i t\Lmd} f_1\|_{L^\infty} \|Z_{\sim l}  e^{ i t \Lmd} f_2\|_{L^2} \nonumber \\
			 &\lesssim 2^{k+\frac32 k_1 } t^{-1}  2^{-(1+\beta-\delta_0)(l+k)}  \|F_1\|_{(X\cap Y)^{0,3}_{\beta'}}  \|F_2\|_{(X\cap Y)^{0,2}_{\beta-\delta_0}} \nonumber \\
			 &\lesssim 2^{2k + \frac12 k_1 - 40 k_1^+  - 40 k_2^+ }   2^{-(1+\beta-\delta_0) (l+k) - \frac12 (l+k-m)} t^{-1+C\varepsilon_1} \varepsilon_1^2 \nonumber \\
			 &\lesssim 2^{2k + \frac12 k_1  - 40 k_1^+ - 40 k_2^+ }   2^{-(1+\beta) (l+k)} t^{-1-\frac{\delta}{4}} \varepsilon_1^2.
\end{align}
In the last line, we have used the fact that 
$(\frac12-\delta_0) (l+k) > (\frac12 + \frac{\delta}{3}) m.$
Then, summing \eqref{eq:428-bb} over such $k_1, k_2$ gives an acceptable contribution.

 	\medskip
\emph{Case 3.} $k_1 \sim k_2 \gtrsim k$.
\medskip

In this case, $l+k_2 \sim l +k_1 \gg m$. Applying Lemma \ref{lem:largel} and \eqref{eq:428-dd}, we now have
\begin{align} 
	\eqref{eq:428-01} &\lesssim 2^k \|Z_{\sim l}  e^{ i t\Lmd} f_1\|_{L^\infty} \|Z_{\sim l}   e^{ i t \Lmd} f_2\|_{L^2} \nonumber \\
	&\lesssim 2^{k+\frac32 k_1} 2^{-(1+\beta-\delta_0) (2l+k_1+k_2)} \|F_1\|_{(X\cap Y)^{0,2}_{\beta-\delta_0}} \|F_2\|_{(X\cap Y)^{0,2}_{\beta-\delta_0}} \nonumber  \\
	&\lesssim 2^{k+\frac32 k_1 - 40 k_1^+ - 40 k_2^+ } 2^{-2(1+\beta-\delta_0) (l+k)} \langle t \rangle^{C \varepsilon_1}\varepsilon_1^2 \nonumber  \\
	&\lesssim 2^{k+\frac32 k_1 - 40 k_1^+ - 40 k_2^+ } 2^{-(1+\beta) (l+k)}  t^{-1-\frac{\beta}{2}}\varepsilon_1^2.
\end{align}
Again, summing over such $k_1, k_2$ we obtain an acceptable contribution.

So far, we have proved the desired estimate involving $Z_l P_k^\hp$ in \eqref{eq:428-ddddd}. The estimate involving $H_l P_k^\vp$ can be proved in an entirely similar way.
\end{proof}

The next lemma is important for applying normal form in trilinear energy type estimates.

\begin{lemma} \label{lem:51-1}
	Assuming the bootstrap assumption \eqref{eq:BA}, for $l+k \ge m+5$,  $t \in [2^m, 2^{m+1}) (m \ge 1)$ and $a+b \le 20$, there holds that
	\begin{equation}
		\left\| \partial_t \left(Z_l P_k^\hp S^a \Omega^b \Fp_\pm\right) \right\|_{L^2} \lesssim 2^{k - 30 k^+}  t^{-2}\varepsilon_1^2.
	\end{equation}
\end{lemma}

\begin{proof}
	By the same notation and reasoning as in the proof of Lemma \ref{lem:428-1}, it is sufficient to effectively bound
	\begin{equation} \label{eq:428-ww}
		2^k \| e^{\mu_1 i t\Lmd} f_1 \cdot  Z_{[l-1,l+1]}   e^{\mu_2 i t \Lmd} f_2\|_{L^2}.
	\end{equation}
	We shall only consider the case $k_1 \le k_2$, as the opposite case can be similarly treated by exchanging the roles of $f_1, f_2$. Note that $k_2 \ge k-2$. 	Similar to Case 2 in the proof of  Lemma \ref{lem:428-1}, here we have
	\begin{align} \label{eq:428-www}
		\eqref{eq:428-ww} &\lesssim 2^k \|e^{\mu_1 i t\Lmd} f_1\|_{L^\infty} \|Z_{[l-1,l+1]}  e^{\mu_2 i t \Lmd} f_2\|_{L^2} \nonumber \\
		&\lesssim 2^{k+\frac32 k_1 } t^{-1}  2^{-(1+\beta-\delta_0)(l+k)}  \|F_1\|_{(X \cap Y)^{0,3}_{\beta'}}  \|F_2\|_{(X\cap Y)^{0,2}_{\beta-\delta_0}} \nonumber \\
		&\lesssim 2^{k+\frac32 k_1 - 40 k_1^+ - 40 k_2^+} t^{-2}    \varepsilon_1^2,
	\end{align}
which is acceptable after summation in $k_1, k_2$.
\end{proof}

\section{The method of partial symmetries} \label{sec:srm}

In this section, we give a brief overview on the spacetime resonance formalism in \cite{GuoInvent}, with adaptations to the non-axisymmetric setting. 

\subsection{Integration by parts along vector fields $S$ and $\Omega$}
As in \cite{GuoInvent}, we denote
\begin{equation}
	S_\eta = \eta \cdot \nabla_\eta, \quad \Omega_\eta = \eta_h^\perp \cdot \nabla_\eta, \quad S_{\xi-\eta} = (\xi-\eta) \cdot \nabla_{\xi-\eta}, \quad \Omega_{\xi-\eta} = (\xi-\eta)_h^\perp \cdot \nabla_{\xi-\eta},
\end{equation}
and
\begin{equation}
	\bar{\sigma} = \xi_3 \eta_h - \eta_3 \xi_h = -(\xi \times \eta)_h^\perp.
\end{equation}
The following observation from \cite{GuoCpam,GuoInvent} is of fundamental importance.
\begin{lemma}\label{lem:53-1} There holds that
		\begin{equation}
		S_\eta \Lambda(\xi - \eta) = \frac{(\xi - \eta)_h \cdot \bar{\sigma}}{|\xi- \eta|^3}, \quad \Omega_\eta \Lambda(\xi - \eta) = - \frac{(\xi - \eta)_h^\perp \cdot \bar{\sigma}}{|\xi- \eta|^3}.
	\end{equation}
	As a consequence, for $\Phi = \pm \Lmd(\xi) \pm \Lmd(\xi - \eta) \pm \Lmd(\eta)$, 
	\begin{equation} \label{eq:805-a1}
		|S_\eta \Phi| + |\Omega_\eta \Phi| \sim \frac{|(\xi-\eta)_h| |\bar{\sigma}|}{|\xi -\eta|^3}. 
	\end{equation}
		The symmetric statements hold with the roles of $\eta$ and $\xi - \eta$
	exchanged.
\end{lemma}

\medskip
To apply integration by parts along vector fields, a first step is to find a lower bound on $|S_\eta \Phi|+|\Omega_\eta \Phi|$, which is equivalent to finding a lower bound on $\bar{\sigma}$ in view of \eqref{eq:805-a1}. In many cases, one can make use of the following result. 

\begin{lemma} \label{lem:805-c1}
	Suppose that the triangle formed by $\xi, \xi-\eta, \eta$ is non-degenerate in the sense that one of the three corner angles takes value in $[2^{-5}, \pi-2^{-5}]$, and moreover $|\Lmd(\zeta)| \ge 2^{-5}$ for some $\zeta \in \{\xi, \xi-\eta, \eta\}$. Then, there holds that
	\begin{equation}
		|\bar{\sigma}| \sim \min \{|\xi|, |\xi-\eta|, |\eta|\} \cdot \max \{|\xi|, |\xi-\eta|, |\eta|\}.
	\end{equation}
\end{lemma}  

\begin{proof}
	Note that $\bar{\sigma} = -(\xi \times \eta)_h^\perp = -((\xi-\eta) \times \eta)_h^\perp = -(\xi \times (\xi-\eta))_h^\perp$, and consequently,
	$$|\bar{\sigma}| \lesssim \min \{|\xi|, |\xi-\eta|, |\eta|\} \cdot \max \{|\xi|, |\xi-\eta|, |\eta|\}.$$ 
	Without loss of generality, let us assume that 
	$$|\xi| = \min \{|\xi|, |\xi-\eta|, |\eta|\}, \quad |\eta| = \max \{|\xi|, |\xi-\eta|, |\eta|\}.$$  
	By the assumptions and simple geometry, one can deduce that $|\Lmd (\xi)|+| \Lmd(\eta)| \gtrsim 1$, and the angle between $\xi$ and $\eta$ is separated from $0$ and $\pi$. Hence, $\Lmd(\xi \times \eta)$ is separated from $\pm 1$, and $|\xi \times \eta| \gtrsim |\xi| |\eta|$. It follows that $|\bar{\sigma}| \gtrsim |\xi||\eta|$.
\end{proof}

\medskip

The next lemma follows from straightforward calculations, and is useful for bilinear estimates after integration by parts. The identities \eqref{eq:805-bb1}--\eqref{eq:805-bb3} are crucial for controlling $S_\eta \wh{f}(\xi-\eta)$ and $\Omega_\eta \wh{f}(\xi-\eta)$ via the $X$ and $Y$ norm bounds of $f$, while \eqref{eq:805-bb4}--\eqref{eq:805-bb7} are very helpful for computing the cross terms.  

\begin{lemma} \label{lem:429-ct}
	The following identities are valid.
	\begin{enumerate}[(i)]
		\item There holds that \begin{align} \label{eq:805-bb1}
			S_{\eta} = - \frac{(\xi - \eta)_h \cdot \eta_h}{| (\xi - \eta)_h |^2}
			S_{\xi - \eta} - \frac{(\xi - \eta)^{\perp}_h \cdot \eta_h}{| (\xi -
				\eta)_h |^2} \Omega_{\xi - \eta} + \frac{(\xi - \eta)_h \cdot
				\bar{\sigma}}{| (\xi - \eta)_h |^2} \partial_{(\xi - \eta)_3},
		\end{align}  
		\begin{align} \label{eq:805-bb2} \Omega_{\eta} = \frac{(\xi - \eta)_h^{\perp} \cdot \eta_h}{| (\xi -
				\eta)_h |^2} S_{\xi - \eta} - \frac{(\xi - \eta)_h \cdot \eta_h}{| (\xi -
				\eta)_h |^2} \Omega_{\xi - \eta} - \frac{(\xi - \eta)_h^{\perp} \cdot
				\bar{\sigma}}{| (\xi - \eta)_h |^2} \partial_{(\xi - \eta)_3}.
		\end{align}
	\item There holds that
		\begin{align} \label{eq:805-bb3}
		S_{\eta} = - \frac{\eta_3}{(\xi - \eta)_3} S_{\xi - \eta} - \frac{\bar{\sigma} \cdot \nabla_{\xi - \eta}}{(\xi - \eta)_3}, \quad \Omega_{\eta} = -\eta_h^\perp \cdot \nabla_{\xi - \eta} = \Omega_{\xi-\eta} -\xi_h^\perp \cdot \nabla_{\xi - \eta}.
	\end{align}  
	\item There holds that
	\begin{equation} \label{eq:805-bb4}
		S_\eta \eta = -S_\eta (\xi-\eta) = \eta, \quad \Omega_\eta \eta = \eta_h^\perp,
	\end{equation}
	\begin{equation} \label{eq:805-bb5}
		S_\eta |\xi-\eta| = - \frac{\eta \cdot (\xi - \eta)}{|\xi - \eta|}, \quad \Omega_\eta |\xi - \eta| = \frac{\xi_h^\perp \cdot \eta_h}{|\xi - \eta|},
	\end{equation}
	\begin{equation} \label{eq:805-bb6}
		S_\eta |(\xi-\eta)_h| = - \frac{\eta_h \cdot (\xi - \eta)_h}{|(\xi - \eta)_h|}, \quad \Omega_\eta |(\xi - \eta)_h| = \frac{\xi_h^\perp \cdot \eta_h}{|(\xi - \eta)_h|},
	\end{equation}
	\begin{equation} \label{eq:805-bb7}
		S_\eta \bar{\sigma} = \bar{\sigma}, \quad \Omega_\eta \bar{\sigma} = \xi_3 \eta_h^\perp.
	\end{equation}
	\end{enumerate}
	The symmetric statements hold with the roles of $\eta$ and $\xi - \eta$
	exchanged.
\end{lemma}

\medskip

%
%
%
%

We now demonstrate how to use integration by parts in $S_\eta$ and $\Omega_\eta$ to estimate the bilinear form  
\begin{equation}
	\Q[\m \chi] (f_1, f_2),
\end{equation}
where $\chi$ is a phase space cutoff function of the type
\begin{equation} \label{eq:429-1}
	\chi(\xi, \eta) = \chi^{\vp/\hp, p/q}_k(\xi) \chi^{\vp/\hp, p_1/q_1}_{k_1}(\xi-\eta) \chi^{\vp/\hp, p_2/q_2}_{k_2}(\eta).
\end{equation}
Suppose that $f_j = Z_{\le l_j} P^{\hp} F_j$ or $H_{\le l_j} P^{\vp} F_j$, $j =1, 2$, and $\|\{S, \Omega\}^{{M_2}/{2}} F_j\|_{H^{M_1/2}} < +\infty$. For simplicity of presentation, we shall think of $p = 0$ in the case of a $(\hp, q)$-localization, and think of $q = 0$ in the case of a $(\vp, p)$-localization. Hence, on the support of $\chi$, we always have
\begin{equation}
	|\Lmd(\xi)| \sim 2^q, \quad  \sqrt{1-\Lmd^2}(\xi) \sim 2^p.
\end{equation}
The same convention applies to the parameters $p_{1,2}, q_{1,2}$ as well. Denote
\begin{equation}
	k_{\max} = \max \{k, k_1, k_2\}, \quad k_{\min} = \min \{k, k_1, k_2\},
\end{equation}
and similarly define $p_{\max}, p_{\min}$ and $q_{\max}, q_{\min}$. In all our applications, we will have $p_{\max} = 0$.
 
Assume that on the support of $\chi$, there exists some constant $L$ such that
\begin{equation} \label{eq:429-2}
	|\bar{\sigma}| \ge L \gtrsim 2^{k_{\max} + k_{\min} + q_{\max}},
\end{equation}
so that
\begin{equation}
	|S_\eta \Phi|+|\Omega_\eta \Phi| \gtrsim 2^{p_1-2k_1} L.
\end{equation}
Note that the condition \eqref{eq:429-2} quantifies the non-degeneracy of the triangle formed by $\xi, \xi-\eta$ and $\eta$. Then, introduce the following cut-off functions on the phase space,
\begin{equation}
	\chi_{_{S_\eta}}(\xi, \eta)  :=  (1-\psi)(2^{2k_1 - p_1+C} L^{-1} S_\eta \Phi), \quad \chi_{_{\Omega_\eta}} : = 1- \chi_{_{S_\eta}},
\end{equation}
with constant $C$ sufficiently large so that, for $V \in \{S, \Omega\}$, on the support of $\chi_{_{V_\eta}}$ there holds that
\begin{equation}
	|V_\eta \Phi| \gtrsim 2^{p_1-2k_1} L.
\end{equation}
Integration by parts in $V_\eta$ gives
\begin{align} \label{eq:429-001}
\F \Q[\m \chi \, \chi_{_{V_\eta}}](f_1, f_2) 
	&=  \int_{\R^3} \frac{V_\eta e^{i t \Phi}}{it V_\eta \Phi}    \m\chi \, \chi_{_{V_\eta}} \wh{f_1}(\xi-\eta) \wh{f_2}(\eta) d\eta \nonumber \\
	&= -\frac{1}{it} \int_{\R^3} e^{i t \Phi} V_\eta^* \left(\frac{ \m\chi \, \chi_{_{V_\eta}}}{V_\eta \Phi} \wh{f_1}(\xi-\eta) \wh{f_2}(\eta) \right) d\eta \nonumber \\
	&=  -\frac{1}{it} \int_{\R^3} e^{i t \Phi} V_\eta^* \left(\frac{ \m\chi \, \chi_{_{V_\eta}}}{V_\eta \Phi} \right) \wh{f_1}(\xi-\eta) \wh{f_2}(\eta)  d\eta \nonumber \\
	&\quad -\frac{1}{it} \int_{\R^3} e^{i t \Phi} \frac{ \m\chi \, \chi_{_{V_\eta}}}{V_\eta \Phi}  V_\eta \wh{f_1}(\xi-\eta) \wh{f_2}(\eta)  d\eta \nonumber \\
	&\quad -\frac{1}{it} \int_{\R^3} e^{i t \Phi} \frac{ \m\chi \, \chi_{_{V_\eta}}}{V_\eta \Phi}  \wh{f_1}(\xi-\eta)  V_\eta \wh{f_2}(\eta)  d\eta,
\end{align}
where we denote
\begin{equation*}
	V_\eta^*  = \begin{cases}
		S_\eta + 3, &\quad \mathrm{if} \  V = S, \\
		\Omega_\eta, &\quad \mathrm{if} \ V = \Omega.
	\end{cases}  
\end{equation*}
The first term on the right of \eqref{eq:429-001} (we call it  a ``cross term") can be calculated using the identities in Lemma \ref{lem:429-ct} (iii), while the penultimate term in \eqref{eq:429-001} can be treated using Lemma \ref{lem:429-ct} (i)--(ii).  Repeated integration by parts would generate more ``cross terms"; we refer to  \cite[Section 5]{GuoInvent} for the  complete formalism.  Each time of applying \eqref{eq:429-001}, we gain a factor of $t^{-1}$ multiplied by some function of the parameters $k, p, q, k_{1,2}, p_{1,2}, q_{1,2}, l_{1,2}$. If the overall gain of each step is controlled by $t^{-\delta}$, then one gets very fast decay by applying integration by parts for $O(\delta^{-1})$ times. To be precise, when we say each step of integration by parts gains a factor of size $\vartheta \ll 1$, we really mean an estimate of the type
\begin{equation}
	\|\F \Q[\m\chi](f_1, f_2)\|_{L^\infty} \lesssim \vartheta^N \|\m \chi\|_{L^\infty}  \|P^{\hp/\vp}_{k_1} \{S, \Omega\}^{\le N} F_1\|_{L^2} \|P^{\hp/\vp}_{k_2} \{S, \Omega\}^{\le N} F_2\|_{L^2}.
\end{equation}
In Section \ref{sec:8} and Section \ref{sec:9}, the above methodology will be implemented in a  case-by-case style. 

%

\subsection{Normal form}

In the bilinear or trilinear expressions to be estimated, we often need to decompose the phase space towards a normal form. For  a number $\lmd > 0 $ (to be chosen specifically in each case of application), we can define
\begin{equation}
	\chi^\res = \psi(\lmd^{-1} \Phi), \quad \chi^\nr = (1-\psi)(\lmd^{-1} \Phi).
\end{equation}
Then, on the support of $\chi^\nr$ we can integrate by parts in the dyadic time interval $\mathcal{I}_m = [2^m, 2^{m+1}) \cap [0,t]$ to get
\begin{align}
	\int_{\mathcal{I}_m} \Q[\chi^\nr \m](f_1, f_2)ds &= -i \Q[\Phi^{-1}\chi^\nr \m](f_1, f_2)|_{2^m}^{2^{m+1}\wedge t} \nonumber \\
	&\quad + i \int_{\mathcal{I}_m} \Q[\Phi^{-1} \chi^\nr \m](\partial_t f_1, f_2)ds \nonumber \\
	&\quad + i \int_{\mathcal{I}_m} \Q[\Phi^{-1} \chi^\nr \m]( f_1, \partial_t f_2)ds.
\end{align}
When following the approach of energy method, we also need
\begin{align}
	\int_{\mathcal{I}_m} \langle \Q[\chi^\nr \m](f_1, f_2), f_3 \rangle_{L^2} ds &= -i\langle \Q[\Phi^{-1}\chi^\nr \m](f_1, f_2), f_3\rangle_{L^2}|_{2^m}^{2^{m+1}\wedge t} \nonumber \\
	&\quad +i  \int_{\mathcal{I}_m} \langle \Q[\Phi^{-1} \chi^\nr \m ](\partial_t f_1, f_2), f_3  \rangle_{L^2} ds \nonumber \\
	&\quad +i  \int_{\mathcal{I}_m} \langle \Q[\Phi^{-1} \chi^\nr \m]( f_1, \partial_t f_2), f_3  \rangle_{L^2} ds  \nonumber \\
	&\quad +i  \int_{\mathcal{I}_m} \langle \Q[\Phi^{-1} \chi^\nr \m]( f_1,  f_2), \partial_t f_3  \rangle_{L^2} ds.
\end{align}

\section{Horizontal frequency interactions} \label{sec:8}


 For simplicity of presentation, we shall write
\begin{equation}
	f = P^\hp_k F, \quad f_1 = P^\hp_{k_1} F_1, \quad f_2 = P^\hp_{k_2} F_2,
\end{equation}
with
\begin{equation}
	F = S^a \Omega^b \Fp_\mu, \quad F_1 = S^{a_1} \Omega^{b_1}  \Fp_{\mu_1}, \quad F_2 =  S^{a_2} \Omega^{b_2}  \Fp_{\mu_2},
\end{equation} 
so that \eqref{eq:807-000} can be written as
\begin{equation}
	\sum_{k_1, k_2 \in \Z} {\sum}^* Z_l^{(k)} P_k^\hp \Q_\mua [\m_\mua] (f_1, f_2).
\end{equation}
Let $s \in \mathcal{I}_m = [2^m, 2^{m+1}) \cap [0,t]$ with $m \ge 1, t \in [1, T]$.  Assuming the bootstrap assumption \eqref{eq:BA}, the aim of this section is to prove the following bilinear and trilinear estimates, which imply the propagation of $X^{10}_\beta$ and $X^{20, 20}_{\beta'}$ norms for the scalar profiles $\Fp_\pm$, see Section \ref{sec:25-boot}.

\begin{itemize}
	\item Suppose that $a+b \le 10$ and $l+k \in [0, m+4]$, then there holds that
	\begin{align} \label{eq:511-1}
	  \sum_{k_1, k_2}     \left\| \int_{\mathcal{I}_m} {\sum}^*  Z_l^{(k)} P_k^\hp \Q_\mua[\m_\mua] (f_1, f_2) ds   \right\|_{L^2}   
	  &\lesssim 2^{-(10-a-b)k^+ -(1+\beta)m}  \,  \varepsilon_1^2.
	\end{align} 
	If the first condition is relaxed to $a+b \le 20$, we can replace the right hand side of \eqref{eq:511-1} with
	\begin{equation} \label{eq:511-901}
		2^{-20 k^+ - (1+\beta')(l+k) - \delta m }  \,  \varepsilon_1^2.
	\end{equation}
	
	\item Suppose that $a+b \le 10$ and $l+k \in [m+5, (1+\delta) m]$, then there holds that
	\begin{align} \label{eq:511-2}
		&\qquad \ \left|\sum_{k_1, k_2} \sum_{\mu=\pm} {\sum}^*  \int_{\mathcal{I}_m} \langle  Z_l P_k^\hp \Q_\mua[\m_\mua] (f_1, f_2)    ,Z_l f\rangle_{L^2} ds \right| \nonumber\\
		& \lesssim 2^{-2(10-a-b)k^+ - 2(1+\beta)(l+k)} \left(2^{-\delta m} + 2^{- |m-l| } + \mathbf{1}_{l+k \le m+15}\right) \,  \varepsilon_1^3.
	\end{align}
		If the first condition is relaxed to $a+b \le 20$, we can replace the right hand side of \eqref{eq:511-2} with
	\begin{equation} \label{eq:511-902}
		2^{-40 k^+ - 2(1+\beta')(l+k) - \delta m }  \,  \varepsilon_1^3.
	\end{equation}

\end{itemize}

During the proof, we will drop the multi-index $\mua$ for simplicity when it plays no role, \emph{i.e.}, we sometimes write $\Q[\m]$ instead of $\Q_\mua[\m_\mua]$.  By the basic energy estimate \eqref{eq:slow-growth} and Young's inequality (cf. the set size gain lemma \cite[Lemma A.3]{GuoInvent}), we have the crude bound at time $s \in \mathcal{I}_m$ for $a+b \le 20$,
\begin{align} \label{eq:806-aa1}
	\left\| P^\hp_k \mathcal{Q}[\m](f_1,f_2)\right\|_{L^2} &\lesssim 2^{ k+\frac32 k_{\min} } \|F_1\|_{L^2} \|F_2\|_{L^2} \nonumber \\
	&\lesssim 2^{ k+\frac32 k_{\min} } 2^{-\frac{M_1}{2} k_1^+ - \frac{M_1}{2} k_2^+} 2^{C \varepsilon_1 m} \varepsilon_1^2.
\end{align}
Hence, it suffices to consider the case
\begin{equation} \label{eq:not-too-large-k}
	k_{\max} := \max\{k, k_1, k_2\} < \delta_0 m \quad \mathrm{and} \quad k_{\min} := \min\{k, k_1, k_2\} > - 3m.
\end{equation}
In the rest of this section, we shall always assume that \eqref{eq:BA}, \eqref{eq:not-too-large-k} and  
\begin{equation} \label{eq:415-01}
	l+k \le (1+\delta)m
\end{equation} 
hold. In Sections \ref{sec:high-low-hp}--\ref{sec:high-high} we  treat the case $a+b \le 10$, while in Section \ref{sec:Xbetap} we  relax the constraint to $a+b \le 20$. 

\subsection{high-low interactions} \label{sec:high-low-hp}

Here, we consider the case $k_1 \ge k_2 + 11$ (thus $|k-k_1|\le 1$) and the main result is that, for any $\mua \in \{\pm\}^3$, $\bar{a}, \bar{b}$ with $a+b \le 10$ and at time $s \in \mathcal{I}_m (m \ge 1)$, there holds that
\begin{align} \label{eq:hh-hl-goal-2}
	 \sum_{k_2 \le k_1-11} \left\| Z_l^{(k)} P^\hp_k \mathcal{Q}[\m](f_1,f_2)\right\|_{L^2} &\lesssim 2^{-(10-a-b)k^+ - (1+\beta)(l+k)} A_{k,l,m} \,  \varepsilon_1^2,
\end{align}
in which we set
\begin{equation} \label{eq:A-def}
	A_{k,l,m} = \left(2^{-(1+\delta)m} + 2^{- |m-l| - m}\right) \mathbf{1}_{m \le l+k-5} + 2^{-(2+\beta)m + (1+\beta)(l+k)} \mathbf{1}_{m \ge l+k-4}.
\end{equation}
 The three types of bounds in $A_{k,l,m}$ will essentially correspond to the three cases considered below.

Due to Lemma \ref{lem:L-bilinear}, it suffices to prove \eqref{eq:hh-hl-goal-2} with $\Q[\m](f_1, f_2)$ replaced by
\begin{equation}
	 e^{-\mu i  s\Lmd }D_0[(D_1 e^{\mu_1 i  s \Lmd} f_1)\cdot(|\nabla| D_2 e^{\mu_2 i  s \Lmd} f_2)],
\end{equation}
where $D_{0,1,2}$ are 0th order operators generated by \eqref{eq:415-02}.


\medskip
\emph{Case 1}. $l + k \ge m + 5$ and $l + k_2 \notin [m-6, m+3]$. 
\smallskip

In this case, $Z_l^{(k)}=Z_l$. The first step is to insert a physical space localization $Z_{[l-1,l+1]}$ after the Fourier multipliers $P^\hp_k D_0 e^{- \mu i  s\Lambda}$. By Lemma \ref{lem:a1-1}, there holds that
\begin{equation} \label{eq:easy-comm}
	\|Z_l P^\hp_k D_0  e^{- \mu i s \Lambda} Z_{[l-1,l+1]^c}\|_{L^2 \to L^2} \lesssim_K 2^{-K(l+k)}, \quad  \forall K\ge 1,
\end{equation}
and thus we have
\begin{align} \label{eq:easy-part}
	&\quad\left\|Z_l P^\hp_k D_0  e^{- \mu i  s \Lambda} Z_{[l-1,l+1]^c} \left[(e^{\mu_1 i s \Lambda}D_1 f_1) \cdot (e^{\mu_2 i s \Lambda } |\nabla| D_2 f_2)\right]\right\|_{L^2} \nonumber \\
	&\lesssim_K  2^{k_2+\frac32 k_1} 2^{-K(l+k)}  \|f_1\|_{L^2} \|f_2\|_{L^2} \lesssim 2^{k_2+\frac32 k_1} 2^{-K(l+k)} \varepsilon_1^2.
\end{align}
Since $l+k > m$, $2^{-K(l+k)}$ gives an arbitrarily fast decay in $s$, which is more than needed. Consequently, to show \eqref{eq:hh-hl-goal-2}, it suffices to effectively bound
\begin{equation} \label{eq:26a}
		 \left\|(Z_{l_1} e^{\mu_1 i s \Lambda}D_1 f_1) \cdot (Z_{l_2} e^{\mu_2 i s \Lambda} |\nabla| D_2 f_2)\right\|_{L^2},
\end{equation}
where $l_1,l_2 \in [l-2, l+2]$. By the assumptions on $l+k$ and $l+k_2$, we know that $l_i \notin  [m - k_i -4,  m-k_i+1], \ i=1,2$. Using Lemma \ref{lem:slab} (ii)--(iii), there holds that
\begin{align} \label{eq:G2-Linf}
	\|Z_{l_2} e^{\mu_2 i s \Lambda} |\nabla| D_2 f_2\|_{L^\infty}  &\lesssim 2^{\frac52 k_2} 2^{-(1+\beta')m} \|F_2\|_{X^{0,3}_{\beta'}} \nonumber \\
	&\lesssim 2^{\frac52 k_2 - 20 k_2^+} 2^{-(1+\beta')m} \varepsilon_1.
\end{align} 
and
\begin{align} \label{eq:G1-L2}
	\|{Z}_{l_1} e^{\mu_1 i s \Lambda}  D_1 f_1\|_{L^2}  &\lesssim  2^{-(1+\beta)(l_1+k_1)} \|F_1\|_{X^{0}_{\beta}} \nonumber \\
	&\lesssim 2^{-(10-a-b)k^+} 2^{-(1+\beta)(l+k)} \varepsilon_1.
\end{align}
Hence,  a direct $L^\infty$--$L^2$ bound gives
\begin{align}
			&\quad  \left\| (Z_{l_1}e^{\mu_1 i s \Lambda}  D_1 f_1) \cdot ({Z}_{l_2} e^{\mu_2 i s \Lambda} |\nabla| D_2 f_2)  \right\|_{L^2} \nonumber \\
			&\lesssim \|Z_{l_1}e^{\mu_1 i s \Lambda}  D_1 f_1\|_{L^2} \|{Z}_{l_2} e^{\mu_2 i s \Lambda} |\nabla| D_2 f_2\|_{L^\infty} \nonumber\\
				 &\lesssim 2^{- (10-a-b)k^+ +\frac52 k_2 - 20 k_2^+ }   2^{-(1+\beta)(l+k)-(1+\beta')m} \varepsilon_1^2, 
\end{align}
which is acceptable for \eqref{eq:hh-hl-goal-2}.

\medskip
\emph{Case 2}.  $l + k_2 \in [m-6, m+3]$.
\smallskip

In this case, we still have $l+k_1 \ge m+5$ and $Z_l^{(k)} = Z_l$. As in Case 1, it suffices to effectively bound \eqref{eq:26a} with $l_1,l_2 \in [l-2,l+2]$. Instead of \eqref{eq:G2-Linf}, by Lemma \ref{lem:slab}(i) we now have
\begin{align} \label{eq:525-11}
	\|Z_{l_2} e^{\mu_2 i s \Lambda} |\nabla| D_2 f_2\|_{L^\infty} &\lesssim 2^{\frac52 k_2 - m} \|F_2\|_{X_{\beta'}^{0,3}} \nonumber \\
	&\lesssim 2^{\frac52 k_2 - 20 k_2^+ - m} \varepsilon_1.
\end{align}
Using a direct  $L^2-L^\infty$ bound with  \eqref{eq:G1-L2}, \eqref{eq:525-11},  and the fact that $|k_2| \ge |m-l|- 6$, we obtain
\begin{align*}
	\left\| (Z_{l_1}e^{\mu_1 i s \Lambda} D_1 f_1) \cdot (Z_{l_2} e^{\mu_2 i s \Lambda} |\nabla| D_2 f_2)  \right\|_{L^2} &\lesssim \|Z_{l_1}e^{\mu_1 i s \Lambda}  D_1 f_1\|_{L^2} \|Z_{l_2} e^{\mu_2 i s \Lambda} |\nabla| D_2 f_2\|_{L^\infty} \\
	&\lesssim 2^{\frac52 k_2 - 20 k_2^+ -(10-a-b)k^+}   2^{-(1+\beta)(l+k)-m} \varepsilon_1^2 \\
	&\lesssim 2^{-\frac32 |m-l|-m} 2^{-|k_2|-(10-a-b)k^+} 2^{-(1+\beta)(l+k)} \varepsilon_1^2,
\end{align*}
which is again acceptable for \eqref{eq:hh-hl-goal-2}.

\medskip
\emph{Case 3}.  $l+k \le m+4$.  
\smallskip

As a first step, we use the $L^2 \to L^2$ boundedness of the localization operators and $e^{- \mu i  t \Lmd}$, as well as Lemma \ref{lem:L-bilinear} to deduce that
\begin{align*}
	\left\| Z_l^{(k)} P^\hp_k \Q[\m](f_1, f_2)\right\|_{L^2}\lesssim \sum \left\|(e^{\mu_1 i s\Lambda} D_1 f_1) \cdot (e^{\mu_2 i  s\Lambda} |\nabla| D_2 f_2)\right\|_{L^2},
\end{align*} 
where the summation runs over a finite number of possible combinations of $D_1, D_2$. Then we make the physical space decomposition
\begin{align} \label{eq:prod-decomp}
	&\quad (e^{\mu_1 i s\Lambda} D_1 f_1) \cdot (e^{\mu_2 i s\Lambda} |\nabla| D_2 f_2) \nonumber \\
	&= \left[Z_{\sim m-k_1} + Z_{\sim m-k_2} + (Id-Z_{\sim m-k_1} - Z_{\sim m-k_2})\right] \left[ ( e^{\mu_1 i s\Lambda} D_1 f_1) \cdot ( e^{\mu_2 i  s\Lambda} |\nabla| D_2 f_2) \right] \nonumber\\
	&=:I_1+I_2+I_3.
\end{align}
Here, $Z_{\sim m-k_i}$ stands for $Z_{[m-k_i-5, m-k_i+2]}$, $i=1,2$. Let us also write $\wt{Z}_{\sim m-k_i}$ for $Z_{[m-k_i-6, m-k_i+3]}$, $i=1,2$, and remark that $\wt{Z}_{\sim m-k_i} Z_{\sim m-k_i} = Z_{\sim m-k_i}$ and $\wt{Z}_{\sim m-k_1} Z_{\sim m-k_2} = \wt{Z}_{\sim m-k_2} Z_{\sim m-k_1} = 0$. Similar to Cases 1 and Case 2, using $L^2$--$L^\infty$ or $L^\infty$--$L^2$ bounds and Lemma \ref{lem:slab}, there holds that
\begin{align} \label{eq:case3-I1-L2}
	\|I_1\|_{L^2} &\lesssim \|Z_{\sim m-k_1} e^{\mu_1  i s\Lambda} D_1 f_1\|_{L^\infty} \|\wt{Z}_{\sim m-k_1} e^{\mu_2 i s\Lambda} |\nabla| D_2 f_2\|_{L^2} \nonumber \\
	&\lesssim 2^{\frac32 k_1-m} \|F_1\|_{X_{\beta'}^{0,3}} 2^{k_2-(1+\beta) m} \|F_2\|_{X_\beta^{0}}  \nonumber \\
	&\lesssim 2^{\frac32 k_1 -20 k_1^++k_2 - (10-a-b)k_2^+} 2^{-(2+\beta)m} \varepsilon_1^2,
\end{align}
and
\begin{align} \label{eq:case3-I2-L2}
	\|I_2\|_{L^2} &\lesssim \|Z_{\sim m-k_2} e^{\mu_1 i s\Lambda} D_1 f_1\|_{L^2} \|\wt{Z}_{\sim m-k_2} e^{\mu_2 i  s\Lambda} |\nabla| D_2 f_2\|_{L^\infty} \nonumber \\
	&\lesssim 2^{-(1+\beta)m} \|F_1\|_{X_\beta^{0}} \, 2^{\frac52 k_2 - m} \|F_2\|_{X^{0,3}_{\beta'}} \nonumber \\
	&\lesssim 2^{\frac52 k_2 -20 k_2^+ -(10-a-b)k_1^+} 2^{-(2+\beta)m} \varepsilon_1^2.
\end{align}
Writing $Z_{\nsim\nsim} = Id-Z_{\sim m-k_1} - Z_{\sim m-k_2}$ and consider an  operator $\wt{Z}_{\nsim\nsim}$ which is similar to $Z_{\nsim\nsim}$ only with slightly enlarged physical support and satisfies $\wt{Z}_{\nsim\nsim} Z_{\nsim\nsim} = Z_{\nsim\nsim}$ and $\wt{Z}_{\nsim\nsim} Z_{[m-k_i-4,m-k_i+1]} = 0, \ i=1,2$. Using an $L^2$--$L^\infty$ bound and Lemma \ref{lem:slab} again, we have
\begin{align} \label{eq:case3-I3-L2}
	\|I_3\|_{L^2} &\lesssim \| Z_{\nsim\nsim} e^{\mu_1 i s\Lambda} D_1 f_1 \|_{L^\infty} \|\wt{Z}_{\nsim\nsim} e^{\mu_2 i s\Lambda} |\nabla| D_2 f_2 \|_{L^2} \nonumber \\
	&\lesssim 2^{\frac32 k_1 -(1+\beta')m} \|F_1\|_{X^{0,3}_{\beta'}} \, 2^{k_2 - (1+\beta)m} \|F_2\|_{X^{0}_\beta} \nonumber\\
	&\lesssim 2^{\frac32 k_1 -20 k_1^+ +k_2 -(10-a-b)k_2^+} 2^{-(2+\beta+\beta')m} \varepsilon_1^2.
\end{align}
All three contributions \eqref{eq:case3-I1-L2}--\eqref{eq:case3-I3-L2} are acceptable for \eqref{eq:hh-hl-goal-2}.


\subsection{low-high interactions}

In this section, we consider the case $k_2 \ge k_1 + 11$, in particular, there holds that $|k_2 - k| \le 1$. Similar to Section \ref{sec:high-low-hp}, we shall consider each time slice with $s \in \mathcal{I}_m$, and show the following timewise bounds.

\begin{enumerate}[(1)]
	\item Suppose that $l + k  \ge m + 5$ and $a+b\le 10$, then for any $\mu_1 \in \{\pm\}$,
	\begin{align} \label{eq:l-h-goal-i'}
		&\sum_{k_1 \le k_2-11} \bigg| \sum_{\substack{\mu = \pm \\ \mu_2 = \pm}} \left< Z_l P^\hp_k \mathcal{Q}_\mua[\m_\mua](f_1,f_2), Z_l f \right>_{L^2} \bigg| \nonumber \\
		&\quad \quad \lesssim A_{k,l,m} 2^{-2(10-a-b)k^+ - 2(1+\beta)(l+k)} \varepsilon_1^3,
	\end{align}
	with $A_{k,l,m}$ given by \eqref{eq:A-def}.
	
	\item Suppose that $l+k \le m+4$ and $a+b \le 10$, then for any $\mua \in \{\pm\}^3$,
	\begin{align}  \label{eq:l-h-goal-ii'}
		\sum_{k_1 \le k_2-11} \left\|Z_l^{(k)} P^\hp_k \mathcal{Q}_\mua[\m_\mua] (f_1,f_2)\right\|_{L^2}  \lesssim A_{k,l,m} 2^{-(10-a-b)k^+ - (1+\beta)(l+k)} \varepsilon_1^2.
	\end{align}
\end{enumerate}

\medskip
\emph{Case 1}.  $l+k \ge m+5$ and $l + k_1 \notin [m-6, m+3]$.
\smallskip

In this case, the energy structure of the nonlinearity is not essential. To show \eqref{eq:l-h-goal-i'}, we start with a direct Cauchy-Schwarz estimate:
\begin{align}
	 \left| \left< Z_l P^\hp_k \mathcal{Q}[\m](f_1,f_2), Z_l f \right> \right|  &\lesssim \|Z_l P^\hp_k \mathcal{Q}[\m](f_1,f_2)\|_{L^2}  \|Z_l f\|_{L^2} \nonumber \\
	 &\lesssim \|Z_l P^\hp_k \mathcal{Q}[\m](f_1,f_2)\|_{L^2}  \left(2^{-(10-a-b)k^+-(1+\beta)(l+k)} \varepsilon_1\right).
\end{align}
The estimate \eqref{eq:easy-part} is still valid in this case, hence it suffices to effectively control
\begin{equation} \label{eq:416-1}
	\left\|(Z_{l_1} e^{\mu_1 i s \Lambda}D_1 f_1) \cdot (Z_{l_2} e^{\mu_2 i s \Lambda} |\nabla| D_2 f_2)\right\|_{L^2},
\end{equation}
 with $l_1, l_2 \in [l-2,l+2]$. Since $l_i+k_i \notin [m-4,m+1], \  i=1,2$, by Lemma \ref{lem:slab} we have
\begin{align}
	\|Z_{l_1}e^{\mu_1 i s \Lambda} D_1 f_1\|_{L^\infty} &\lesssim 2^{\frac32 k_1} 2^{-(1+\beta') m}  \|F_1\|_{X^{0,3}_{\beta'}} \nonumber \\
	&\lesssim 2^{\frac32 k_1 - 20 k_1^+} 2^{-(1+\beta')m} \varepsilon_1,
\end{align}
and
\begin{align}
	\|Z_{l_2}e^{\mu_2 i s \Lambda} |\nabla| D_2 f_2\|_{L^2} &\lesssim 2^{k_2} 2^{-(1+\beta) (l_2+k_2)}  \|F_2\|_{X^{0}_{\beta}} \nonumber \\
	&\lesssim 2^{ k -(10-a-b) k^+} 2^{-(1+\beta)(l+k)} \varepsilon_1.
\end{align}
Hence, we deduce that
\begin{align}
    &\quad \left\|(Z_{l_1} e^{\mu_1 i s \Lambda}D_1 f_1) \cdot (Z_{l_2} e^{\mu_2  i s \Lambda} |\nabla| D_2 f_2)\right\|_{L^2} \nonumber \\
    &\lesssim \|Z_{l_1}e^{\mu_1 i s \Lambda} D_1 f_1\|_{L^\infty} \|Z_{l_2}e^{\mu_2 i s \Lambda} |\nabla| D_2 f_2\|_{L^2} \nonumber \\
    &\lesssim 2^{\frac32 k_1 - 20 k_1^+ +k - (10-a-b) k^+}  2^{-(1+\beta')m}   2^{-(1+\beta)(l+k)} \varepsilon_1^2,
\end{align}
which gives an acceptable contribution to \eqref{eq:l-h-goal-i'}. Note that the factor $2^k$ could be large, but it is absorbed by the extra decay in $2^{m}$ due to $k < \delta_0 m$.

\medskip
\emph{Case 2}.  $l + k_1 \in [m-6, m+3]$.
\smallskip

In this case, to prove the trilinear estimate \eqref{eq:l-h-goal-i'} we have to exploit the energy structure of the nonlinearity captured by the identity \eqref{eq:ener-struc}. 

\medskip
\emph{Subcase 2.1}.  $a_2=a, b_2 = b$.
\smallskip

Since $a_1+a_2\le a, b_1+b_2\le b$, the assumption implies that $a_1 = b_1 = 0$. Let 
\begin{equation}
	\wt{P}_k^\hp f = \F^{-1} \bigg\{\wt{\chi}^\hp(\xi) \sum_{k' \in [k-2, k+2]} \varphi(2^{-k'} |\xi_h|)  \wh{f} \bigg\}, 
\end{equation}
then, since $|k_2 - k|\le 1$, there holds $\widetilde{P}_k^\hp P_k^\hp = P_k^\hp$ and $\widetilde{P}_k^\hp P_{k_2}^\hp = P_{k_2}^\hp$ simultaneously. 
For simplicity of presentation, we denote
\begin{equation}
	\quad g_2 = \widetilde{P}_k^\hp F_2 = \widetilde{P}_k^\hp S^a \Omega^b \Fp_{\mu_2},\quad g = \widetilde{P}_k^\hp F = \widetilde{P}_k^\hp S^a \Omega^b \Fp_\mu.
\end{equation}
 By the same reasoning as in \eqref{eq:easy-comm}--\eqref{eq:easy-part}, up to good remainder terms, in \eqref{eq:l-h-goal-i'} we can replace 
 $f_2$ with
 $$h_2 = P_{k_2}^\hp Z_{[l-2,l+2]} g_2,$$ 
 and also replace $f$ with
 $$h = P_k^\hp Z_{[l-1,l+1]} g.$$
 Namely, it is sufficient to effectively bound the modified trilinear form
\begin{align} \label{eq:11a}
 \sum_{\substack{\mu = \pm \\ \mu_2 = \pm}} \left< Z_l P^\hp_k \mathcal{Q}_\mua[\m_\mua](f_1, h_2), Z_l h \right>_{L^2}.
\end{align}

By Lemma \ref{lem:L-bilinear}, we can decompose $\Q_\mua[\m_\mua](f_1, h_2)$ in \eqref{eq:11a} into terms of the form
\begin{equation}\label{eq:416-2}
		 e^{-\mu i  s\Lmd }D_0[(e^{\mu_1 i  s \Lmd} D_1  f_1)\cdot(e^{\mu_2 i  s \Lmd} |\nabla| D_2  h_2)],
\end{equation}
where $D_{0,1,2}$ are 0th order differential operators generated by \eqref{eq:415-02}. Then, we can commute $Z_l$  with the operators $P_k^\hp e^{-\mu i  s \Lambda} D_0$ and ${P}_{k_2}^\hp e^{\mu_2 i  s \Lambda} D_2 |\nabla|$ within \eqref{eq:11a}, based on the commutator estimate \eqref{eq:comm-2}. Indeed, the first commutator term
\begin{align} \label{eq:11b}
	\left[Z_l, P^\hp_k e^{- \mu i  s \Lambda} D_0\right] \left( (e^{\mu_1 i s \Lambda} D_1 f_1)\cdot (e^{\mu_2 i s \Lambda} |\nabla| D_2 h_2) \right)
\end{align}
gives an acceptable contribution since
\begin{align} \label{eq:525-c1}
	\|\eqref{eq:11b}\|_{L^2} &\lesssim 2^{-l-k+m} \|e^{\mu_1 i s \Lambda} D_1 f_1\|_{L^\infty} \| |\nabla| h_2\|_{L^2} \nonumber \\
	&\lesssim 2^{-l-k+m} 2^{\frac32 k_1  - m} 2^{k_2 - (1+\beta)(l+k_2)} \|F_1\|_{{X}^{0,3}_{\beta'}} \|F_2\|_{\wt{X}^{0}_{\beta}} \nonumber \\
	&\lesssim 2^{\frac52 k_1 - 20 k_1^+ - (10-a-b)k^+} 2^{-m-(1+\beta)(l+k)} \varepsilon_1^2 \nonumber \\
	&\lesssim  2^{- |k_1| - (10-a-b)k^+} 2^{-\frac32|m-l|-m-(1+\beta)(l+k)} \varepsilon_1^2.
\end{align}
Here, the $\wt{X}$ norm is defined by replacing $P^\hp$ with $\wt{P}^\hp$ in the definition of corresponding the $X$ norms, and in view of \eqref{eq:729-001}--\eqref{eq:729-002} we have  that  
$$\|F_2\|_{\wt{X}^{0}_{\beta}} \lesssim \|F_2\|_{{X}^{0}_{\beta}} + \|F_2\|_{{Y}^{0,2}_{\beta}} \lesssim 2^{-(10-a-b)k^+} \varepsilon_1.$$
In the last line of \eqref{eq:525-c1} we also used that $k_1 \sim m-l$. The second commutator term
\begin{align} \label{eq:11c}
	P^\hp_k e^{- \mu i  s \Lambda} D_0 \left( (e^{\mu_1 i s \Lambda} D_1 f_1)\cdot \Big([Z_l, e^{\mu_2 i s \Lambda} |\nabla| D_2 P_{k_2}^\hp] Z_{[l-2,l+2]} g_2 \Big) \right)
\end{align}
can be estimated similarly as 
\begin{align*}
	\|\eqref{eq:11c}\|_{L^2} &\lesssim 2^{-l+m} \|e^{\mu_1 i s \Lambda} D_1 f_1\|_{L^\infty} \| Z_{[l-2,l+2]} g_2\|_{L^2} \\
	&\lesssim 2^{-l+m} 2^{\frac32 k_1 - m} 2^{ - (1+\beta)(l+k_2)} \|F_1\|_{X^{0,3}_{\beta'}} \|F_2\|_{\wt{X}^{0}_{\beta}} \\
	&\lesssim 2^{\frac52 k_1 - 20 k_1^+ - (10-a-b)k^+} 2^{-m-(1+\beta)(l+k)} \varepsilon_1^2 \\
	&\lesssim  2^{- |k_1| - (10-a-b)k^+} 2^{-\frac32|m-l|-m-(1+\beta)(l+k)} \varepsilon_1^2.
\end{align*}
By such two commutations within \eqref{eq:11a}, we obtain the expression
\begin{align} \label{eq:11d}
	P^\hp_k e^{- \mu i  s \Lambda} D_0 \left( (e^{\mu_1 i s \Lambda} D_1 f_1)\cdot \big(e^{\mu_2 i s \Lambda} |\nabla| D_2 P_{k_2}^\hp Z_{l} g_2 \big) \right).
\end{align}
Summing over all versions of $D_{0,1,2}$, we get
\begin{align} \label{eq:12a}
	 Z_lP_k^\hp \Q_\mua[\m_\mua] (f_1, h_2) - \textit{comm.} =  P^\hp_k \Q_\mua[\m_\mua] (f_1, P_{k_2}^\hp Z_{l} g_2),
\end{align}
where we write \emph{comm.} for the \emph{acceptable commutator terms}. For the second factor $Z_l h$ in the inner product of \eqref{eq:11a}, we commute $Z_l$ with $P_k^\hp$ with the identity
\begin{align} \label{eq:12b}
	Z_l P_k^\hp Z_{[l-1,l+1]} g - [Z_l, P_k^\hp] Z_{[l-1,l+1]} g = P^\hp_k Z_l g.
\end{align}
The commutator term in \eqref{eq:12b} is also acceptable. Indeed, using \eqref{eq:comm-2} with $m=0$, we have
\begin{align*}
	&\quad \left|\left< Z_l P^\hp_k \Q[\m](f_1, h_2), [Z_l, P^\hp_k]Z_{[l-1,l+1]} g \right>_{L^2}\right|  \\
	&\lesssim 2^{k_2} \|e^{\mu_1 i  s \Lmd} f_1\|_{L^\infty} \|h_2\|_{L^2} \cdot 2^{-l-k} \|Z_{[l-1,l+1]} g\| \\
	&\lesssim 2^{-l} 2^{\frac32 k_1 -m} 2^{-2(1+\beta)(l+k)} \|F_1\|_{X^{0,3}_{\beta'}} \|F_2\|_{\wt{X}^0_\beta} \|F\|_{\wt{X}^0_\beta} \\
	&\lesssim 2^{\frac52 k_1 - 20k_1^+ -2m} 2^{-2(1+\beta)(l+k)-2(10-a-b)k^+} \varepsilon_1^3.
\end{align*}

Next, expressing the inner product between the right of \eqref{eq:12a} and \eqref{eq:12b} on the Fourier side, and symmetrizing the roles of $\xi, \eta$, we obtain
\begin{align} \label{eq:13a}
	&\quad \sum_{\mu, \mu_2} \left<P^\hp_k \mathcal{Q}_\mua[\m_\mua](f_1, P_{k_2}^\hp Z_{l} g_2), P^\hp_k Z_l g\right> \nonumber\\ 
	&= \sum_{\mu, \mu_2} \iint e^{i s \Phi_\mua} \m_\mua (\xi, \eta) \left[\chi_k^\hp(\xi)\right]^2 \chi_{k_2}^\hp(\eta) \widehat{f_1}(\xi-\eta)  \widehat{Z_l g_2}(\eta) \overline{\widehat{Z_l g}(\xi)} d\xi  d\eta \nonumber \\
	&\stackrel{\eqref{eq:ener-struc}}{=} \frac12 \sum_{\mu, \mu_2}  \iint  e^{i s \Phi_\mua} \mathfrak{m}_\mua(\xi, \eta) \left[ \left(\chi_k^\hp\right)^2(\xi) \chi_{k_2}^\hp(\eta) - \left(\chi_k^\hp\right)^2(\eta) \chi_{k_2}^\hp(\xi)\right] \widehat{f_1}(\xi-\eta)  \widehat{Z_l g_2}(\eta) \overline{\widehat{Z_l g}(\xi)} d\xi  d\eta.
\end{align}
Let $\mathfrak{n}(\xi, \eta) = \chi_{k_1}^\hp(\xi-\eta) \big[\left(\chi_k^\hp\right)^2(\xi) \chi_{k_2}^\hp(\eta) - \left(\chi_k^\hp\right)^2(\eta) \chi_{k_2}^\hp(\xi)\big]$. Using the fact that
\begin{equation} \label{eq:813-pp1}
	\left\|\F_{\xi,\eta}^{-1} \mathfrak{n} \right\|_{L^1} \lesssim 2^{-k+k_1}
\end{equation}
($\F_{\xi, \eta}^{-1}$ stands for the inverse Fourier transform in both $\xi$ and $\eta$) and the physical product structure of $\m$ given by Lemma \ref{lem:L-bilinear}, we obtain that
\begin{align*}
	|\eqref{eq:13a}| &\lesssim \sum_{\mu, \mu_2} 2^{-k+k_1+k_2} \|e^{\mu_1 i s \Lambda} f_1\|_{L^\infty} \|Z_l g_2\|_{L^2} \|Z_l g\|_{L^2}   \\
	&\lesssim  \sum_{\mu, \mu_2} 2^{\frac52 k_1 - m} 2^{- 2(1+\beta)(l+k)} \|F_1\|_{X^{0,3}_{\beta'}} \|F_2\|_{\wt{X}^0_\beta} \|F\|_{\wt{X}^0_\beta} \\
	&\lesssim 2^{\frac52 k_1 - 20k_1^+ - m} 2^{-2(10-a-b)k^+ - 2(1+\beta)(l+k)} \varepsilon_1^2,
\end{align*}
which is again acceptable since $k_1 \sim m-l$.  To see \eqref{eq:813-pp1}, we point out that
\begin{align}
	&\iint \left|\mathfrak{n}\right|  d\xi d\eta + \sum_{N=1}^{7} \iint  \left( 2^{N k} |(\nabla_{\xi} + \nabla_\eta)^N \mathfrak{n}| + 2^{N k_1} |(\nabla_{\xi} - \nabla_\eta)^N \mathfrak{n}| \right) d\xi d\eta \nonumber \\
	& \qquad  \lesssim 2^{-k+k_1+3k+3k_1} = 2^{2k+4k_1}.
\end{align}

\medskip
\emph{Subcase 2.2}.  $a_2<a$ or $b_2 < b$.
\smallskip

By assumption, we have $a_1 + b_1 = a+b-(a_2+b_2)\ge 1$. A direct $L^\infty$--$L^2$ estimate gives
\begin{align*}
	\eqref{eq:416-1} &\lesssim \|Z_{l_1} e^{\mu_1 i s \Lambda}  D_1 f_1\|_{L^\infty} \|Z_{l_2}e^{\mu_2 i s \Lambda} |\nabla| D_2 f_2\|_{L^2} \\
	&\lesssim 2^{\frac32 k_1 - m} 2^{k_2 - (1+\beta)(l_2+k_2)} \|F_1\|_{X^{0,3}_{\beta'}} \|F_2\|_{X^0_\beta} \\
	& \lesssim 2^{\frac32 k_1 - 20k_1^+ - m} 2^{k_2-(10-a_2-b_2)k_2^+ - (1+\beta)(l_2+k_2)} \varepsilon_1^2 \\
	&\lesssim 2^{\frac32 k_1 - 20k_1^+ - m} 2^{-(10-a-b)k^+ - (1+\beta)(l+k)} \varepsilon_1^2,
\end{align*}
which is acceptable since $k_1 \sim m-l$.

\medskip
\emph{Case 3}.  $l+k \le m+4$.
\smallskip

Similar to Case 3 of Section \ref{sec:high-low-hp}, we define $I_1,I_2, I_3$ as in \eqref{eq:prod-decomp}. Using Lemma \ref{lem:slab}, there holds that 
\begin{align} \label{eq:case3-I1-L2-lh}
	\|I_1\|_{L^2} &\lesssim \|Z_{\sim m-k_1} e^{\mu_1 i s\Lambda} D_1 f_1\|_{L^\infty} \|\wt{Z}_{\sim m-k_1} e^{\mu_2 i  s\Lambda} |\nabla| D_2 f_2\|_{L^2} \nonumber \\
	&\lesssim 2^{\frac32 k_1-m} \|F_1\|_{X_{\beta'}^{0,3}}\, 2^{k_2-(1+\beta) (m-k_1+k_2)} \|F_2\|_{X_\beta^{0}}  \nonumber \\
	&\lesssim 2^{\frac52  k_1 - 20 k_1^+ -\beta (k_2-k_1) - (10-a-b)k^+} 2^{-(2+\beta)m} \varepsilon_1^2,
\end{align}
and
\begin{align} \label{eq:case3-I3-L2-lh}
	\|I_3\|_{L^2} &\lesssim \| Z_{\nsim\nsim} e^{\mu_1 i s\Lambda} D_1 f_1 \|_{L^\infty} \|\wt{Z}_{\nsim\nsim} e^{\mu_2 i s\Lambda} |\nabla| D_2 f_2 \|_{L^2} \nonumber \\
	&\lesssim 2^{\frac32 k_1 -(1+\beta')m} \|F_1\|_{X^{0,3}_{\beta'}}  2^{k_2 - (1+\beta)m} \|F_2\|_{X^{0}_\beta} \nonumber\\
	&\lesssim 2^{\frac32 k_1 - 20 k_1^+ +k -(10-a-b)k^+} 2^{-(2+\beta+\beta')m} \varepsilon_1^2,
\end{align}
both of which are acceptable contributions to \eqref{eq:l-h-goal-ii'}. For $I_2$, a similar analysis would lead to
\begin{align} \label{eq:case3-I2-L2-lh}
	\|I_2\|_{L^2} &\lesssim \|Z_{\sim m-k_2} e^{\mu_1 i s\Lambda} D_1 f_1\|_{L^2} \|\wt{Z}_{\sim m-k_2} e^{\mu_2 i  s\Lambda} |\nabla| D_2 f_2\|_{L^\infty} \nonumber \\
	&\lesssim 2^{-(1+\beta)m} \|F_1\|_{X_\beta^{0}} 2^{\frac52 k_2 - m} \|F_2\|_{X^{0,3}_{\beta'}} \nonumber \\
	&\lesssim 2^{\frac52 k - 20 k^+ - (10-a-b) k_1^+} 2^{-(2+\beta)m} \varepsilon_1^2
\end{align}
However, we point out that \eqref{eq:case3-I2-L2-lh} is not summable in $k_1 \ll k$, thus a refined estimate is required. By Lemma \ref{lem:slab}(i) we have the pointwise bound
\begin{align}
	|\wt{Z}_{\sim m-k_2} e^{\mu_2 i  s\Lambda} |\nabla| D_2 f_2| &\lesssim 2^{\frac52 k_2 - m} \langle |x_h| \rangle^{-\frac12} \|F_2\|_{X^{0,3}_{\beta'}} \nonumber \\
	&\lesssim 2^{\frac52 k_2 - 20 k_2^+ - m} \langle |x_h| \rangle^{-\frac12} \varepsilon_1,
\end{align}
and by Lemma \ref{lem:slab}(iii) we have
\begin{align}
	\|{Z}_{\sim m-k_2} e^{\mu_1 i  s\Lambda} D_1 f_1\|_{L^\infty_{x_h} L^2_{x_3}} &\lesssim 2^{-(1+\beta)m+k_1} \|F_1\|_{X^0_\beta} \\
	&\lesssim 2^{-(1+\beta)m+k_1} 2^{-(10-a-b)k_1^+} \varepsilon_1.
\end{align}
Then, we estimate $\|I_2\|_{L^2}$ on two sets $\mathcal{C}_1 = \{x \in \R^3: |x_h|<2^{-\frac{2k_1}{3}}\}$ and
$\mathcal{C}_2 = \R^3 \setminus \mathcal{C}_1$ separately as
\begin{align} \label{eq:416-01}
	\|I_2\|_{L^2(\mathcal{C}_1)} &\lesssim 2^{-\frac23 k_1} \|Z_{\sim m-k_2} e^{\mu_1 i s\Lambda} D_1 f_1\|_{L^\infty_{x_h} L^2_{x_3}(\mathcal{C}_1)} \|\wt{Z}_{\sim m-k_2} e^{\mu_2 i  s\Lambda} |\nabla| D_2 f_2\|_{L^\infty(\mathcal{C}_1)} \nonumber \\
	&\lesssim 2^{-\frac{2k_1}{3}} 2^{-(1+\beta)m+k_1} 2^{-(10-a-b)k_1^+} 2^{\frac52 k_2 - 20 k_2^+ - m} \varepsilon_1^2 \nonumber \\
	&\lesssim 2^{\frac{k_1}{3} - (10-a-b) k_1^+ } 2^{\frac52 k -20 k^+ } 2^{-(2+\beta)m} \varepsilon_1^2,
\end{align}
and
\begin{align} \label{eq:416-02}
	\|I_2\|_{L^2(\mathcal{C}_2)} &\lesssim  \|Z_{\sim m-k_2} e^{\mu_1 i s\Lambda} D_1 f_1\|_{L^2(\mathcal{C}_2)} \|\wt{Z}_{\sim m-k_2} e^{\mu_2 i  s\Lambda} |\nabla| D_2 f_2\|_{L^\infty(\mathcal{C}_2)} \nonumber \\
	&\lesssim  2^{-(1+\beta)m} 2^{-(10-a-b)k_1^+} 2^{\frac52 k_2 - 20 k_2^+ - m} 2^{\frac{k_1}{3}}\varepsilon_1^2 \nonumber \\
	&\lesssim 2^{\frac{k_1}{3} - (10-a-b) k_1^+} 2^{\frac52 k - 20 k^+ } 2^{-(2+\beta)m} \varepsilon_1^2.
\end{align}
Due to the common factor $2^{\frac{k_1}{3}}$, we can sum \eqref{eq:416-01} and \eqref{eq:416-02} over $k_1 \ll k$ and both of them are acceptable contributions to \eqref{eq:l-h-goal-ii'}.

\subsection{high-high interactions, $\mu_1 \neq \mu_2$} \label{sec:511}

Consider the case $|k_1 - k_2| \le 11$ and $\mu_1 = - \mu_2$. In this case, there holds $k_1 \ge k-12$ and $k_2 \ge k-12$. The method here is in fact simpler than the previous two types of interactions (high-low, low-high). We apply $L^2$--$L^\infty$ or $L^\infty$--$L^2$ estimates in two separate cases. 

\medskip
\emph{Case 1}. $l+k \ge m + 16$.
\smallskip

Similar to Case 1 of Section \ref{sec:high-low-hp}, we may consider, for $l_1, l_2 \in [l-2, l+2]$, the following bound.
\begin{align} \label{eq:hh-hh-case1-1}
	&\quad \|(Z_{l_1}e^{\mu_1 i s \Lambda} D_1 f_1) \cdot (Z_{l_2}e^{\mu_2 i  s \Lambda} |\nabla| D_2 f_2)\|_{L^2} \nonumber \\
	&\le \|Z_{l_1} e^{\mu_1 i s \Lambda} D_1 f_1\|_{L^2}  \|Z_{l_2} e^{\mu_2 i  s \Lambda} |\nabla| D_2 f_2\|_{L^\infty} \nonumber \\
	&\lesssim 2^{-(1+\beta) (l_1+k_1)} \|F_1\|_{X^{0}_{\beta}} 2^{\frac52 k_2 -(1+\beta') m} \|F_2\|_{X^{0,3}_{\beta'}} \nonumber \\
	&\lesssim 2^{-(1+\beta) (l_1+k_1)-(10-a-b)k_1^+} 2^{\frac52 k_2 -(1+\beta') m - 20k_2^+} \varepsilon_1^2 \nonumber \\
	&\lesssim 2^{-(1+\beta')m} 2^{-(10-a-b)k^+ - (1+\beta)(l+k)}  \varepsilon_1^2.
	\end{align}
This gives an acceptable contribution.

\medskip
\emph{Case 2}.  $l + k \le m + 15 $.
\smallskip

We make the decomposition
\begin{align} \label{eq:prod-decomp-2}
	&\quad (e^{\mu_1 i s\Lambda} D_1 f_1) \cdot (e^{\mu_2 i  s\Lambda} |\nabla| D_2 f_2) \nonumber \\
	&= \left[Z^{-\mu_1}_{[m-k_1-4, m-k_1+1]} + (\mbox{Id}-Z^{-\mu_1}_{[m-k_1-4, m-k_1+1]})\right] \left[ ( e^{\mu_1 i s\Lambda} D_1 f_1) \cdot ( e^{\mu_2 i  s\Lambda} |\nabla| D_2 f_2) \right] \nonumber\\
	&=:I_1+I_2.
\end{align}
Using Lemma \ref{lem:slab}, there holds that
\begin{align}
	I_1 &\le  \|Z_{[m-k_1-4, m-k_1+1]}^{-\mu_1} e^{\mu_1 i s \Lambda} D_1 f_1\|_{L^\infty}  \|Z_{[m-k_1-5, m-k_1+2]}^{-\mu_1} e^{\mu_2 i  s \Lambda} |\nabla| D_2 f_2\|_{L^2} \nonumber \\
	&\lesssim 2^{\frac32 k_1 - m} \|F_1\|_{X^{0,3}_{\beta'}} 2^{ k_2 - (1+\beta) m} \|F_2\|_{X^{0}_{\beta}} \nonumber \\
	&\lesssim 2^{\frac32 k_1 - m - 20 k_1^+} 2^{k_2 - (1+\beta) m - (10-a-b)k_2^+} \varepsilon_1^2 \nonumber \\
	&\lesssim 2^{-|k_1|-(10-a-b)k^+ - (2+\beta)m}  \varepsilon_1^2,
\end{align}
and
\begin{align}
	I_2 &\le  \|(\mbox{Id}-Z^{-\mu_1}_{[m-k_1-4, m-k_1+1]}) e^{\mu_1 i s \Lambda} D_1 f_1\|_{L^2}  \|e^{\mu_2 i  s \Lambda} |\nabla| D_2 f_2\|_{L^\infty} \nonumber \\
	&\lesssim 2^{-(1+\beta) m} \|F_1\|_{X^{0}_{\beta}} 2^{\frac52 k_2 - m} \|F_2\|_{X^{0,3}_{\beta'}} \nonumber \\
	&\lesssim 2^{-(1+\beta) m -(10-a-b)k_1^+} 2^{\frac52 k_2 - m - 20k_2^+} \varepsilon_1^2 \nonumber \\
	&\lesssim 2^{-|k_2| -(10-a-b)k^+ - (2+\beta)m}  \varepsilon_1^2.
\end{align}
%
Both contributions are acceptable.

\subsection{high-high interactions, $\mu_1 = \mu_2$} \label{sec:high-high}

Consider the case $|k_1 - k_2| \le 11$ (hence $k_1 \ge k-12$ and $k_2 \ge k-12$) and $\mu_1 = \mu_2$. We can take $\mu_1 = \mu_2 = +$ as the opposite case can be treated in the same way up to a reflection with respect to the $x_h$ plane. 

First of all, the case $l+k \ge m + 16$ can be treated similarly as in Case 1 of Section \ref{sec:511}. Hence, in the sequel we may assume that $l+k\le m+15$.

Then, we make use of the following decompositions
\begin{align}
		\Q[\m](f_1, f_2) = 	\Q[\m^{(1)}+\m^{(2)}](f_1, f_2)+	\Q[\m^{(3)}](f_1, f_2),
\end{align}
and
\begin{align} \label{eq:422-1}
	\Q[\m^{(1)}+\m^{(2)}](f_1, f_2) &= \sum_{q_1> - m\alpha} \Q[\m^{(1)}+\m^{(2)}](P^{\hp, q_1} f_1, P^{\hp, \le q_1-2} f_2) \nonumber \\ 
	&\quad + \sum_{q_2 > - m\alpha } \Q[\m^{(1)}+\m^{(2)}](P^{\hp, \le q_2-2} f_1, P^{\hp, q_2} f_2) \nonumber \\
	&\quad + \sum_{\substack{|q_1 - q_2| \le 1, \\ \max\{q_1, q_2\}>-m \alpha}} \Q[\m^{(1)}+\m^{(2)}](P^{\hp, q_1} f_1, P^{\hp, q_2} f_2) \nonumber \\
	&\quad + \Q[\m^{(1)}+\m^{(2)}](P^{\hp,\le -m \alpha} f_1, P^{\hp, \le -m \alpha} f_2) \nonumber \\
	&=: S_1 + S_2 + S_3+ S_4,
\end{align}
and
\begin{align} \label{eq:422-1'}
	\Q[\m^{(3)}](f_1, f_2) &= \sum_{q_1> - \frac{9m}{10}} \Q[\m^{(3)}](P^{\hp, q_1} f_1, P^{\hp, \le q_1-2} f_2) \nonumber \\ 
	&\quad + \sum_{q_2 > - \frac{9m}{10} } \Q[\m^{(3)}](P^{\hp, \le q_2-2} f_1, P^{\hp, q_2} f_2) \nonumber \\
	&\quad + \sum_{\substack{|q_1 - q_2| \le 1, \\ \max\{q_1, q_2\}>- \frac{9m}{10}}} \Q[\m^{(3)}](P^{\hp, q_1} f_1, P^{\hp, q_2} f_2) \nonumber \\
	&\quad + \Q[\m^{(3)}](P^{\hp,\le - \frac{9m}{10}} f_1, P^{\hp, \le - \frac{9m}{10}} f_2) \nonumber \\
	&=: S_1' + S_2' + S_3'+ S_4',
\end{align}
in which $\m^{(1)}, \m^{(2)}$ and $\m^{(3)}$ are given by Proposition \ref{prop:nonlinear}. Here, the parameters $q_1,q_2$ are always taking values in $\Z_-$.

\subsubsection{Estimates for $S_i$ and $S_i', \ i=1,2,3$ via spacetime resonances}

In this section, we treat $S_1, S_2$ and $S_3$ defined in \eqref{eq:422-1}, as well as   $S_1', S_2'$ and $S_3'$ defined in \eqref{eq:422-1'}. We shall first work on the case $q_1, q_2 \gtrsim -m\alpha$ for $\Q[\m]$. Then we show the improved bounds for $\Q[\m^{(3)}]$ based on the GPW null condition on $\m^{(3)}$ (see Proposition \ref{prop:nonlinear}) which allows us to treat more negative $q_1, q_2$.

\medskip
{\bf $S_1, S_2$ and $S_1', S_2'$}
\smallskip

We consider the estimates for $S_1$ and $S_1'$, and the method for $S_2$ and $S_2'$ is similar by exchanging the roles of the two inputs.  

Decompose $f_1 =  Z_{\le l_1} f_1 +  Z_{> l_1} f_1$ and $f_2 =  Z_{\le l_2} f_2 +  Z_{> l_2} f_2$ with 
\begin{equation} \label{eq:422-2}
	l_1 = (1-\delta)m-k_1, \quad l_2 = (1-\delta)m-k_2.
\end{equation}
To treat the $Z_{\le l_1} f_1$ part, we use integration by parts along vector fields $S_\eta, \Omega_\eta$.  Note that on the support of $\chi_{k_1}^{\hp, q_1}(\xi-\eta) \chi_{k_2}^{\hp, \le q_1-2}(\eta)$, there holds that
\begin{equation}
	|\bar\sigma| = |(\xi-\eta)_h| |\eta_h| \bigg|\frac{(\xi-\eta)_3}{|(\xi-\eta)_h|} \frac{\eta_h}{|\eta_h|} - \frac{\eta_3}{|\eta_h|} \frac{(\xi-\eta)_h}{|(\xi-\eta)_h|}\bigg| \sim 2^{q_1+2k_1},
\end{equation}
and
\begin{equation}
	|S_\eta \Phi| + |\Omega_\eta \Phi| \sim \frac{|(\xi-\eta)_h||\bar{\sigma}|}{|\xi-\eta|^3} \sim 2^{q_1}.
\end{equation}
Based on \eqref{eq:429-001}, if $f_1$ is replaced by $Z_{\le l_1} f_1$, integration by parts in $V_\eta \in \{S_\eta, \Omega_\eta\}$ gains a factor bounded by 
\begin{equation} \label{eq:813-uu1}
	s^{-1} (2^{-q_1} +2^{l_1+k_1}) \lesssim s^{-\delta}.
\end{equation}
Note that here we do not need to localize further in the size of $\Lmd(\xi)$. To see \eqref{eq:813-uu1}, we analyze the main terms arising from \eqref{eq:429-001} (at time $s \in \mathcal{I}_m$ and with $L = 2^{q_1+2k_1}$):
\begin{enumerate}[(a)]
	\item There is always a gain of $s^{-1} 2^{-q_1}$ due to the denominator $s V_\eta \Phi$. Other factors are generated according to where $V_\eta$ lands.
	\item If $V_\eta$ lands on $(\chi^{\hp,\le q_1 -2}\wh{f}_2)(\eta)$, this produces a bounded factor  as $V_\eta$ is a good derivative for $\chi^{\hp,\le q_1 -2} \wh{f}_2$.
	\item If $V_\eta$ lands on $\chi^{\hp, q_1}(\xi-\eta) \F (Z_{\le l_1} f_1)(\xi-\eta)$, this produces a factor $\lesssim 1 + 2^{l_1+q_1+k_1}$ using the identities \eqref{eq:805-bb1}--\eqref{eq:805-bb2}.
	\item If $V_\eta$ lands on $\chi_{_{V_\eta}}$ or $\frac{1}{V_\eta\Phi}$, this produces a factor of size $\left|\frac{V_\eta V_\eta' \Phi}{V_\eta \Phi}\right|  \lesssim 1$ ($V_\eta' \in \{S_\eta, \Omega_\eta\}$) due to $k_1 \sim k_2$ and $|\Omega_\eta \bar{\sigma}|=|\xi_3 \eta_h| \le |(\xi-\eta)_3 \eta_h| + |\eta_3 \eta_h| \lesssim L$. 
	\item If $V_\eta$ lands on the multiplier or the coefficients in \eqref{eq:805-bb1}--\eqref{eq:805-bb2} that arise from previous steps of integration by parts, this leads to a bounded factor using the identities \eqref{eq:805-bb4}--\eqref{eq:805-bb7}.
\end{enumerate}
Then, repeated integration by parts in $V_\eta$ gives fast pointwise decay in time for the Fourier transform of $\Q[\m](P^{\hp, q_1} Z_{\le l_1} f_1, P^{\hp, \le q_1-2} f_2)$ which is more than needed. Similarly, repeated integration by parts in $V_{\xi-\eta}$ also gives fast decay in time if $f_2$ is replaced by $ Z_{\le l_2} f_2$.

To treat the $Z_{> l_i} f_i$ parts, we use the $X$ norm bounds and set size gain (cf. \cite[Section A.2]{GuoInvent}) to get
\begin{align} \label{eq:422-3}
	\|\Q[\m]({P}^{\hp, q_1} Z_{> l_1} f_1, {P}^{\hp, \le q_1-2} Z_{> l_2} f_2)\|_{L^2} &\lesssim 2^{k+\frac{3k}{2} }\|Z_{> l_1} f_1\|_{L^2} \|Z_{> l_2} f_2\|_{L^2} \nonumber \\
	 &\lesssim 2^{\frac{5k}{2}  - 2(1+\beta) (1-\delta)m } \|F_1\|_{X^0_\beta} \|F_2\|_{X^0_\beta} \nonumber \\
	 &\lesssim 2^{\frac{5k}{2}  - 2(1+\beta) (1-\delta)m - (20-a-b)k_1^+  } \varepsilon_1^2,
\end{align}
which gives an acceptable contribution. 

Note that in this step we only require that $q_1 \gtrsim -(1-\delta) m$, and it also works if $\m$ is replaced by $\m^{(3)}$.

\medskip
{\bf $S_3$ and $S_3'$}
\smallskip

Next, we consider the estimates for $S_3$ and $S_3'$. In order to utilize the multiplier bounds in \cite[Lemma A.8]{GuoInvent}, we also localize the output frequency in parameter $q$ as
\begin{align} \label{422-2-1}
P_k^{\hp} \Q[\m](P^{\hp, q_1}f_1, P^{\hp, q_2}f_2) &= \sum_{ q_1+5 \le q \le -1} P^{\hp} P^{\hp,q}_k \Q[\m](P^{\hp, q_1}f_1, P^{\hp, q_2}f_2) \nonumber \\
&\quad + P^{\hp} P^{\hp,\le q_1+4}_k \Q[\m](P^{\hp, q_1}f_1, P^{\hp, q_2}f_2).
\end{align}
(If $q_1+4 \ge 0$, we take $P^{\hp, \le q_1+4} = P^{\hp, \le -1}$.) For the terms with $q\ge q_1+5$ there holds that 
\begin{equation}
	|\Phi| = \big|-\mu \Lmd(\xi) + \mu_1 \Lmd(\xi-\eta) + \mu_2 \Lmd (\eta)\big| \gtrsim 2^q,
\end{equation}
and thus we consider such terms as ``nonresonant". For the last term in \eqref{422-2-1}, we make a further decomposition of the phase space towards normal form:
\begin{equation}
	1 \equiv  \chi^{\res}(\xi,\eta) +\chi^{\nr}(\xi, \eta) := \psi(2^{-q_1+5} \Phi) + (1-\psi)(2^{-q_1+5} \Phi).
\end{equation}
Let us first treat the resonant part, \emph{i.e.},
\begin{equation} \label{eq:51-1}
	P^{\hp,\le q_1+4}_k \Q[\m \chi^\res](P^{\hp, q_1} f_1, P^{\hp, q_2} f_2).
\end{equation} 
 Note that on the support of $\chi^{\res} \chi_{k_1}^{\hp, q_1}(\xi-\eta) \chi_{k_2}^{\hp, q_2}(\eta)$, we have (cf. \cite[Proposition 6.2]{GuoCpam} or \cite[Proposition 5.2]{GuoInvent})
\begin{equation}
	\Lmd(\xi) \lesssim 2^{q_1}, \quad |\bar{\sigma}| \sim  2^{k+k_1+q_1},
\end{equation}
and
\begin{equation}
	|S_\eta \Phi| + |\Omega_\eta \Phi| \sim  2^{k - k_1+q_1}.
\end{equation}
Consider the following two cases:

\medskip
\emph{Case 1. $q_1 + k - k_1 \gtrsim -(1-\delta)m$.} Decompose $f_1 =  Z_{\le l_1} f_1 +  Z_{> l_1} f_1$ and $f_2 =  Z_{\le l_2} f_2 +  Z_{> l_2} f_2$ where $l_1, l_2$ are taken as in \eqref{eq:422-2}. For the $Z_{\le l_1} f_1$ part, each step of integration by parts in $V_\eta \in \{S_\eta, \Omega_\eta\}$  gains a factor bounded by 
\begin{equation} \label{eq:factor-001}
	s^{-1} (2^{-q_1 -k +k_1} + 2^{-q_1} + 2^{l_1+k_1}) \lesssim s^{-\delta}.
\end{equation} 
To see this, we analyze the main terms arising from \eqref{eq:429-001} (with $L = 2^{k+k_1+q_1}$):
\begin{enumerate}[(a)]
	\item There is always a gain of $s^{-1} 2^{-q_1-k+k_1}$ due to the denominator $s V_\eta \Phi$. Other factors are generated according to where $V_\eta$ lands.
	\item If $V_\eta$ lands on $(\chi^{\hp,q_2}\wh{f}_2)(\eta)$, this produces a bounded factor  as $V_\eta$ is a good derivative for $\chi^{\hp,q_2} \wh{f}_2$.
	\item If $V_\eta$ lands on $\chi^{\res}(\xi) \chi^{\hp, q_1}(\xi-\eta) \F (Z_{\le l_1} f_1)(\xi-\eta)$, this produces a factor $\lesssim 1 + 2^{k-k_1}+ 2^{l_1+k+q_1}$ using the identities \eqref{eq:805-bb1}--\eqref{eq:805-bb2}.
	\item If $V_\eta$ lands on $\chi_{_{V_\eta}}$ or $\frac{1}{V_\eta\Phi}$, this produces a factor of size $\left|\frac{V_\eta V_\eta' \Phi}{V_\eta \Phi}\right|  \lesssim 1$ ($V_\eta' \in \{S_\eta, \Omega_\eta\}$) due to $k_1 \sim k_2$ and $|\Omega_\eta \bar{\sigma}|=|\xi_3 \eta_h| \lesssim L$. 
	\item If $V_\eta$ lands on the multiplier or the coefficients in \eqref{eq:805-bb1}--\eqref{eq:805-bb2} that arise from previous steps of integration by parts, this leads to a bounded factor using the identities \eqref{eq:805-bb4}--\eqref{eq:805-bb7}.
\end{enumerate}
Similarly, for the $Z_{\le l_2} f_2$ part, each step of integration by parts in $V_{\xi-\eta}$  gains a factor bounded by 
\begin{equation}
	s^{-1} (2^{-q_2 -k + k_2} + 2^{-q_2} + 2^{l_2+k_2}) \lesssim s^{-\delta}.
\end{equation} 
Hence, repeated integration by parts gives fast decay of \eqref{eq:51-1} if $f_1$ is replaced by $ Z_{\le l_1} f_1$ or if $f_2$ is replaced by $ Z_{\le l_2} f_2$. Then, we treat the  $Z_{> l_i} F_i$ parts in a similar way as in \eqref{eq:422-3},
\begin{align}
	&\quad \|P^{\hp, \le q_1+4}_k \Q[\m \chi^\res]({P}^{\hp, q_1} Z_{> l_1} f_1, {P}^{\hp, q_2} Z_{> l_2} f_2)\|_{L^2} \nonumber \\ 
	&\lesssim 2^{k+\frac{3k}{2} + \frac{q_1}{2}} \|Z_{> l_1} f_1\|_{L^2} \|Z_{> l_2} f_2\|_{L^2} \nonumber \\
	&\lesssim 2^{\frac{5k}{2} + \frac{q_1}{2} - 2(1+\beta) (1-\delta)m } \|F_1\|_{X^0_\beta} \|F_2\|_{X^0_\beta}  \nonumber \\
	&\lesssim 2^{\frac{5k}{2} + \frac{q_1}{2} - 2(1+\beta) (1-\delta)m - (20-a-b)k_1^+  } \varepsilon_1^2.
\end{align}
Here, we do not rely on the condition that $q_1 \gtrsim -\alpha m$, hence the argument works equally well for $S_3$ and $S_3'$.

\smallskip
\emph{Case 2. $q_1 + k - k_1 \lesssim -(1-\delta)m$.} In this case, repeated integration by parts in $V_\eta$ or $V_{\xi - \eta}$ is not beneficial. We directly use set size gain and the $L^\infty$ bounds for $\wh{f_{1}}$ and $\wh{f_2}$ as follows.
\begin{align} \label{eq:422-4}
	&\quad \|P^{\hp, \le q_1+4}_k \Q[\m \chi^\res](P^{\hp, q_1} f_1, P^{\hp, q_2} f_2)\|_{L^2} \nonumber \\ 
	&\lesssim 2^{k+\frac{3k}{2} + \frac{q_1}{2}} 2^{\frac{3k_1}{2} +  \frac{q_1}{2} + \frac{3k_2}{2} + \frac{q_2}{2}} \|\wh{f_1}\|_{L^\infty} \|\wh{ f_2}\|_{L^\infty} \nonumber \\
	&\lesssim 2^{\frac{5k}{2} + \frac{3q_1}{2}} \|F_1\|_{X^{0,3}_{\beta'}} \|F_2\|_{X^{0,3}_{\beta'}} \nonumber \\
	&\lesssim 2^{\frac{5(k-k_1)}{2} + \frac{3q_1}{2} + \frac{5k_1}{2} - 20 k_1^+ - 20 k_2^+  } \varepsilon_1^2.
\end{align}
In the case that $q_1 \gtrsim - \alpha m$, we have that $k-k_1\lesssim -(1 - \alpha -\delta) m$, hence
\begin{equation} \label{eq:422-5}
	2^{\frac{5(k-k_1)}{2} + \frac{3q_1}{2}} = 2^{\frac{3(k-k_1+q_1)}{2} + {k-k_1}} \lesssim 2^{(-\frac52 + \alpha + \frac52 \delta)m} \lesssim 2^{-(2+\beta+10\delta)m}.
\end{equation}
For the last inequality, we rely on the facts that $\alpha < \frac12$ (see \eqref{eq:614-aa1}--\eqref{eq:802-001}) and $\beta \ll 1$. Using \eqref{eq:422-5}, we see that \eqref{eq:422-4} gives an acceptable contribution. If $\m$ is replaced by $\m^{(3)}$ and $q_1 \ge -\frac{9m}{10}$, we receive an additional factor $2^{q_1}$ in \eqref{eq:422-4} due to the GPW null condition (see \eqref{eq:GPW-null}). Then, using the fact that 
\begin{equation} \label{eq:422-5'}
	2^{\frac{5(k-k_1)}{2} + \frac{5q_1}{2}} \lesssim 2^{-\frac52(1-\delta)m},
\end{equation}
we still get an acceptable contribution.

\medskip

Let us now treat the nonresonant parts, \emph{i.e.},
\begin{equation} \label{eq:422-2-2}
	P^{\hp,\le q_1+4}_k \Q[\m \chi^\nr](P^{\hp, q_1} f_1, P^{\hp, q_2} f_2)
\end{equation}
and
\begin{equation} \label{eq:422-2-3}
	P^{\hp, q}_k \Q[\m](P^{\hp, q_1} f_1, P^{\hp, q_2} f_2), \quad q \ge q_1+5.
\end{equation} 
Here, we can apply a normal form as $|\Phi| \gtrsim 2^{q}$ (for \eqref{eq:422-2-2} we think of $q = q_1$).  We just demonstrate the treatment of \eqref{eq:422-2-3}, while the estimate for \eqref{eq:422-2-2} is similar by taking $q = q_1$ in all the estimates below (cf. \cite[Lemma A.8]{GuoInvent}). Consider the following two cases in which we obtain a bilinear estimate and a trilinear estimate respectively. 

\smallskip
\emph{Case 1.} $l+k \le m + 4$. In this case we are following the approach of Duhamel's formula. Performing an integration by parts in time, we receive that
\begin{align} \label{eq:51-01}
 \int_{2^m}^{2^{m+1} \wedge t} P^{\hp, q}_k \Q[\m](P^{\hp, q_1} f_1, P^{\hp, q_2} f_2) ds
	& = -iP^{\hp, q}_k \Q[\Phi^{-1} \m ](P^{\hp, q_1} f_1, P^{\hp, q_2} f_2)\big|_{s = 2^{m+1} \wedge t } \nonumber\\
	&\quad + i P^{\hp, q}_k \Q[\Phi^{-1} \m ](P^{\hp, q_1} f_1, P^{\hp, q_2} f_2)\big|_{s = 2^{m}} \nonumber\\
	&\quad + i \int_{2^m}^{2^{m+1} \wedge t} P^{\hp, q}_k \Q[\Phi^{-1} \m ](P^{\hp, q_1} \partial_t f_1, P^{\hp, q_2} f_2) ds \nonumber\\
	&\quad + i \int_{2^m}^{2^{m+1} \wedge t} P^{\hp, q}_k \Q[\Phi^{-1} \m ](P^{\hp, q_1}  f_1, P^{\hp, q_2} \partial_t f_2) ds.
\end{align} 
For the first two term on the right of \eqref{eq:51-01}, we invoke Lemma \ref{lem:hp-disp} (ii) and \cite[Lemma A.8]{GuoInvent} to obtain that
\begin{align} \label{eq:51-2}
	&\quad \|P^{\hp, q}_k \Q[\Phi^{-1} \m ](P^{\hp, q_1} f_1, P^{\hp, q_2} f_2)\|_{L^2} \nonumber \\
	&\lesssim 2^{k-q} \left(\|e^{\mu_1 i s \Lmd} P^{\hp, q_1} f_1^I\|_{L^\infty} 2^{ \frac{3k_2}{2} + \frac{q_2}{2}}  \|\wh{ f_2}\|_{L^\infty} + \|e^{\mu_1 i s \Lmd} P^{\hp, q_1} f_1^{II}\|_{L^2}   \|e^{\mu_2 i s \Lmd} P^{\hp, q_2} f_2\|_{L^\infty}\right) \nonumber \\
	&\lesssim 2^{k-q}(2^{\frac32 k_1 - \frac{3m}{2}}+2^{\frac32 k_2-(2+\beta')m}) \|F_1\|_{X^{0,3}_{\beta'}} \|F_2\|_{X^{0,3}_{\beta'}} \nonumber \\
	&\lesssim 2^{k+\frac32 k_1-q- \frac{3m}{2}} 2^{- 20 k_1^+ - 20 k_2^+} \varepsilon_1^2.
\end{align} 
This is acceptable for $q_1 \gtrsim -\alpha m$ since we have $q \gtrsim q_1$ and $\alpha < \frac12$. In the case that $\m$ is replaced by $\m^{(3)}$ and the range of $q_1$ is replaced by $q_1 \gtrsim -\frac{9 m}{10}$, we receive an additional factor $2^q$ in \eqref{eq:51-2}, which makes the estimates  acceptable again. For the third term on the right of \eqref{eq:51-01}, we use Lemma \ref{lem:528-01} to deduce that
\begin{align} \label{eq:614-002}
	&\quad \ \|P^{\hp, q}_k \Q[\Phi^{-1} \m ](P^{\hp, q_1} \partial_t f_1, P^{\hp, q_2} f_2)\|_{L^2} \nonumber \\
	&\lesssim 2^{k-q} \|\partial_t f_1\|_{L^2} \|e^{\mu_2 i s \Lmd} f_2\|_{L^\infty} \nonumber \\
	&\lesssim 2^{k-q} 2^{k_1-30k_1^+ + \frac32 k_2- 20 k_2^+ - (\frac52-\delta_0) m}  \varepsilon_1^3,
\end{align}
which is similar to \eqref{eq:51-2} upon integration in time. The last term in \eqref{eq:51-01} is similarly treated as in \eqref{eq:614-002} by exchanging the roles of the inputs.

\medskip
\emph{Case 2.} $l+k \ge m + 5$. In this case we are following the approach of energy method. Performing an integration by parts in time for the trilinear form, we receive that
\begin{align} \label{eq:51-3}
	&\quad \int_{2^m}^{2^{m+1} \wedge t} \langle P^{\hp, q}_k \Q[\m ](P^{\hp, q_1} f_1, P^{\hp, q_2} f_2),  Z_l^2 F \rangle_{L^2}  ds \nonumber \\
	& = -i \langle  P^{\hp, q}_k \Q[\Phi^{-1} \m ](P^{\hp, q_1} f_1, P^{\hp, q_2} f_2), Z_l^2 F \rangle_{L^2}  \big|_{s = 2^{m+1} \wedge t } \nonumber\\
	&\quad + i \langle P^{\hp, q}_k \Q[\Phi^{-1} \m ](P^{\hp, q_1} f_1, P^{\hp, q_2} f_2) , Z_l^2 F \rangle_{L^2} \big|_{s = 2^{m}} \nonumber\\
	&\quad + i \int_{2^m}^{2^{m+1} \wedge t} \langle P^{\hp, q}_k \Q[\Phi^{-1} \m ](P^{\hp, q_1} \partial_t f_1, P^{\hp, q_2} f_2), Z_l^2 F \rangle_{L^2} ds \nonumber\\
	&\quad + i \int_{2^m}^{2^{m+1} \wedge t} \langle P^{\hp, q}_k \Q[\Phi^{-1} \m ](P^{\hp, q_1}  f_1, P^{\hp, q_2} \partial_t f_2), Z_l^2 F \rangle_{L^2} ds \nonumber\\
		&\quad + i \int_{2^m}^{2^{m+1} \wedge t} \langle P^{\hp, q}_k \Q[\Phi^{-1} \m ](P^{\hp, q_1}  f_1, P^{\hp, q_2}  f_2), Z_l^2 \partial_t F \rangle_{L^2} ds.
\end{align} 
All terms in \eqref{eq:51-3} except the last one can be treated
by applying Cauchy-Schwarz inequality to the inner product in $L^2$ and reducing to \eqref{eq:51-2} and \eqref{eq:614-002}. For the last term, we invoke Lemma \ref{lem:51-1} and \eqref{eq:51-2} to get that
\begin{align} \label{eq:51-5}
	&\quad \langle P^{\hp, q}_k \Q[\Phi^{-1} \m ](P^{\hp, q_1}  f_1, P^{\hp, q_2}  f_2), Z_l^2 \partial_t F \rangle_{L^2}  \nonumber \\
	&\lesssim \|P^{\hp, q}_k \Q[\Phi^{-1} \m ](P^{\hp, q_1}  f_1, P^{\hp, q_2}  f_2)\|_{L^2} \|\partial_t Z_l F\|_{L^2} \nonumber \\
	&\lesssim 2^{k+\frac32 k_1 -30 k^+ - 20 k_1^+ - 20 k_2^+} 2^{-q- \frac{7m}{2}} \varepsilon_1^4 \nonumber\\
	&\lesssim 2^{k+\frac32 k_1 -30 k^+ - 20 k_1^+ - 20 k_2^+} 2^{-q- (\frac{3}{2} - 3\beta)  m } 2^{-2(1+\beta)(l+k)} \varepsilon_1^4.
\end{align}
As in Case 1, \eqref{eq:51-5} is acceptable for $q_1 \gtrsim - \alpha m$. If $\m$ is replaced by $\m^{(3)}$ and $q_1 \gtrsim -\frac{9}{10}m$, then we receive an additional factor $2^q$ in \eqref{eq:51-5}, which makes it  acceptable again.

\subsubsection{Estimate for $S_4$ via a  multiscale decomposition} \label{sec:742}

The treatment of $S_4$ is one of the most challenging tasks of this work. For simplicity, we denote
\begin{equation}
	g_1 = P^{\hp, \le -\alpha m} f_1, \quad g_2 = P^{\hp, \le -\alpha m} f_2.
\end{equation}

\medskip
\textbf{Step 1.} Symmetrization.
\smallskip

Since $\mu_1 = \mu_2 = +$ and the constraints 
\begin{equation}
	|k_1-k_2| \le 11, a_1+a_2 \le a, b_1+b_2 \le b
\end{equation}
are all symmetric in the indices $1$ and $2$, a symmetrization procedure can be applied. Recall that by definition $\L_\mua[\m_\mua] = e^{\mu i s \Lmd} Q_\mua[\m_\mua]$ and note that
\begin{equation}
	\L_\mua[\m_\mua](g_2, g_1) = \L_\mua[\wt{\m}_\mua](g_1, g_2)
\end{equation}
with $\wt{\m}_\mua(\xi,  \eta) = \m_\mua(\xi, \xi-\eta)$. Hence, it suffices to effectively bound
\begin{equation}
	\L_\mua[\m_\mua^{(1), \mathrm{Sym}} + \m_\mua^{(2), \mathrm{Sym}}](g_1, g_2) = \L_\mua[\m_\mua^{(1)}+\m_\mua^{(2)}](g_1, g_2) + \L_\mua[\m_\mua^{(1)}+\m_\mua^{(2)}](g_2, g_1)
\end{equation}
in $L^2$, where $\m_\mua^{(1), \mathrm{Sym}}$ and $\m_\mua^{(2), \mathrm{Sym}}$ are given by Remark \ref{rem:null-sym}. 

\smallskip

 Define the bilinear multipliers
\begin{equation}
	\mathfrak{n}_1 = \frac{  (\xi-\eta)_1 }{|(\xi-\eta)_h|}(|\eta_h|-|(\xi-\eta)_h|),\quad 	\mathfrak{n}_2 = \frac{  (\xi-\eta)_2 }{|(\xi-\eta)_h|} (|\eta_h|-|(\xi-\eta)_h|),
\end{equation}
\begin{equation}
	\mathfrak{n}_3 = \frac{  \eta_1 }{|\eta_h|} (|\eta_h|-|(\xi-\eta)_h|), \quad 	\mathfrak{n}_4 = \frac{  \eta_2 }{|\eta_h|} (|\eta_h|-|(\xi-\eta)_h|),
\end{equation}
and
\begin{equation}
	\mathfrak{n}_5 = \frac{  (\xi-\eta)_h \cdot \eta_h^\perp }{|(\xi-\eta)_h||\eta_h|} (|\eta_h|-|(\xi-\eta)_h|).
\end{equation}
Due to the identity
\begin{equation}
	2(\m_\mua^{(1), \mathrm{Sym}} + \m_\mua^{(2), \mathrm{Sym}}) = \frac{\xi_1}{|\xi_h|} \mathfrak{n}_2  -   \frac{\xi_2}{|\xi_h|} \mathfrak{n}_1  -\frac{\xi_1}{|\xi_h|} \mathfrak{n}_4 +\frac{\xi_2}{|\xi_h|} \mathfrak{n}_3 + \mu \mu_1\mathfrak{n}_5,
\end{equation}
there holds that
\begin{equation}
	\|	\L_\mua[\m_\mua^{(1), \mathrm{Sym}} + \m_\mua^{(2), \mathrm{Sym}}](g_1, g_2)\|_{L^2} \lesssim \sum_{i=1}^5 \|\L_\mua[\mathfrak{n}_i](g_1, g_2)\|_{L^2}.
\end{equation}
For simplicity, in the sequel we will write $ \mathcal{B}$ for any one of  $\L_\mua[\mathfrak{n}_i], \ i=1,2,3,4,5$.
 
 \medskip
 \textbf{Step 2.} Reduction to $k_1 \sim k_2 \sim 0$. 
 \smallskip
 
 Our next goal is to show that
 \begin{align} \label{eq:804-b1}
 	\| \mathcal{B}(g_1, g_2)\|_{L^2} &\lesssim   2^{\frac52 k_1 -(2+\beta)m} \big(\|F_1\|_{X^{0,10}_{\beta'} } \|F_2\|_{X^{0}_{\beta}} + \|F_1\|_{X^{0}_{\beta}} \|F_2\|_{X^{0,10}_{\beta'} }  + \|F_1\|_{X^{0,10}_{\beta'} } \|F_2\|_{X^{0,10}_{\beta'} }\big) \nonumber \\
 	&\lesssim 2^{\frac52 k_1 -(2+\beta)m - 20 k_1^+} \varepsilon_1^2,
 \end{align}
which gives an acceptable contribution to \eqref{eq:511-1} and \eqref{eq:511-2}.
 By scaling, it is sufficient to prove the first inequality in \eqref{eq:804-b1} for the case $k_1, k_2 \in [-10, 10]$, which will be considered in the following.
 
\medskip
\textbf{Step 3.} Far field bounds.
\smallskip

	Note that the conclusion of Lemma \ref{lem:slab} (ii) and (iii) also applies to $P_k^{\hp,  \le -\alpha m} F$, namely, if we let $\mathcal{R} = Z^{-\mu}_{[m-k-4, m-k+1]^c} e^{\mu i t \Lmd} P_k^{\hp,  \le -\alpha m} F$, then the estimates \eqref{eq:803-001}--\eqref{eq:803-003} are still valid. The proof is the same, only with $\chi^\hp_k$ replaced by $\chi_k^{\hp, \le -\alpha m}$ when necessary (cf. Remark \ref{rem:413}). Hence, by the above remark and a direct $L^\infty$--$L^2$ estimate, we obtain that
\begin{align} \label{eq:803-a1}
	&\quad  \ \| \mathcal{B}(g_1, g_2)\|_{L^2( x_3 \notin [-C_2 s, -C_2^{-1} s] )}  \nonumber\\
	&\lesssim  \|e^{i s\Lmd} D_1 g_1\|_{L^\infty ( x_3 \notin [-C_2 s, -C_2^{-1} s] )} \|e^{i s\Lmd} D_2 g_2\|_{L^2 ( x_3 \notin [-C_2 s, -C_2^{-1} s] )} \nonumber \\
	&\lesssim  2^{-(2+2 \beta')m} \|F_1\|_{X^0_{\beta'}} \|F_2\|_{X^0_{\beta'}}.
\end{align}
Here and in the sequel, we write $D_1, D_2$ for smooth multipliers from the set $\left\{\frac{\partial_1}{|\nabla_h|}, \frac{\partial_2}{|\nabla_h|}, |\nabla_h|, \partial_1, \partial_2\right\}$. On the other hand, by Lemma \ref{lem:51-3} (ii) we have the bound
\begin{align} \label{eq:51-z}
	&\quad  \ \| \mathcal{B}(g_1, g_2)\|_{L^2( \{x_3 \in [-C_2 s, -C_2^{-1} s]\} \setminus \mathcal{C} )}  \nonumber\\
	&\lesssim \|e^{i s\Lmd} D_1 g_1\|_{L^\infty ( \{x_3 \in [-C_2 s, -C_2^{-1} s]\} \setminus \mathcal{C} )} \|e^{i s\Lmd} D_2 g_2\|_{L^2 } \nonumber \\
	&\lesssim  2^{-3m} \|F_1\|_{X^{0,10}_{0}} \|F_2\|_{X^0_{0}}. 
\end{align}
Both  \eqref{eq:803-a1} and \eqref{eq:51-z} give acceptable contributions.   Hence, in the following steps it is sufficient to consider
$$\| \mathcal{B}(g_1, g_2)\|_{L^2(\mathcal{C})}.$$

\medskip
\textbf{Step 4.} Multiscale decomposition of the bilinear form via $Q^J_j$ operators.
\smallskip

Let us take the numbers
\begin{equation}
	J_0 = 5, \quad J_{\max} = \lfloor \frac{m\gamma}{2} \rfloor.
\end{equation}
and denote, for $J \in [J_0, J_{\max}] \cap \Z$,
\begin{equation}
	N_J = [-10\times 4^J, 10\times 4^J] \cap \Z.
\end{equation}
By definition of the $Q^J_j$ localization operators in Section \ref{sec:52}, we have the following two simple but important facts. Consider $j_1, j_2 \in N_J$ and $j_1', j_2' \in N_{J+1}$ with $J \in [J_0, J_{\max}-1] \cap \Z$.
\begin{enumerate} [(i)]
	\item Suppose that $|j_1' - j_2'| \ge 6$ and $|j_1 - j_2| \le 1$, then there holds either	$ Q^{J}_{j_1} Q^{J+1}_{j_1'} = 0$ or	$ Q^{J}_{j_2} Q^{J+1}_{j_2'} = 0$.
	
	\item Suppose that $|j_1' - j_2'| \le 1$ and $ Q^{J}_{j_2} Q^{J+1}_{j_2'} \neq 0$, then there holds that
	\begin{equation}
		Q^{J+1}_{j_1'} \sum_{j_1 \in N_{J}: |j_1-j_2| \le 1} Q^{J}_{j_1} = Q^{J+1}_{j_1'}.
	\end{equation}
\end{enumerate}
Based on (i) and (ii), we can make a decomposition of the bilinear form in order to exploit the hidden null type structure revealed in Remark \ref{rem:null-sym}. 

As a first step, by writing $g_i = \sum_{j_i \in N_{J_0}} Q^{J_0}_{j_i} g_i, \ i=1,2$, we have
\begin{align} \label{eq:51-6}
	 \mathcal{B}(g_1, g_2) &= \sum_{j_1, j_2 \in N_{J_0}}  \mathcal{B}(Q^{J_0}_{j_1} g_1, Q^{J_0}_{j_2} g_2) \nonumber \\
	&= \sum_{\substack{j_1, j_2\in N_{J_0} \\ |j_1-j_2|>1}}  \mathcal{B}(Q^{J_0}_{j_1} g_1, Q^{J_0}_{j_2} g_2)  +\sum_{\substack{j_1, j_2\in N_{J_0} \\ |j_1-j_2|\le 1}}  \mathcal{B}(Q^{J_0}_{j_1} g_1, Q^{J_0}_{j_2} g_2).
\end{align}
The second sum  in \eqref{eq:51-6} can be further decomposed as
\begin{align} \label{eq:51-9}
	\sum_{\substack{j_1, j_2\in N_{J_0} \\ |j_1-j_2| \le 1}}  \mathcal{B}(Q^{J_0}_{j_1} g_1, Q^{J_0}_{j_2} g_2) &=	\sum_{\substack{j_1, j_2\in N_{J_0} \\ |j_1-j_2| \le 1}} \sum_{\substack{j_1', j_2'\in N_{J_0+1} \\ |j_1'-j_2'| > 1}}  \mathcal{B}(Q^{J_0}_{j_1} Q^{J_0+1}_{j_1'}g_1, Q^{J_0}_{j_2} Q^{J_0+1}_{j_2'} g_2) \nonumber \\
	&+\sum_{\substack{j_1, j_2\in N_{J_0} \\ |j_1-j_2| \le 1}} \sum_{\substack{j_1', j_2'\in N_{J_0+1} \\ |j_1'-j_2'| \le 1}}  \mathcal{B}(Q^{J_0}_{j_1} Q^{J_0+1}_{j_1'}g_1, Q^{J_0}_{j_2} Q^{J_0+1}_{j_2'} g_2).
\end{align}
Using the fact (i), in the first double sum on the right of \eqref{eq:51-9}, we can further make the restriction that $|j_1'-j_2'|<6$. Using the fact (ii), the last double sum in \eqref{eq:51-9} can be simplified as
\begin{equation} \label{eq:51-10}
	\sum_{\substack{j_1', j_2'\in N_{J_0+1} \\ |j_1'-j_2'| \le 1}}  \mathcal{B}( Q^{J_0+1}_{j_1'}g_1,  Q^{J_0+1}_{j_2'} g_2),
\end{equation}
which has the same form as the second sum in \eqref{eq:51-6} only with $J_0$ replaced by $J_0+1$. Hence the process in \eqref{eq:51-9}--\eqref{eq:51-10} can be iterated, until we reach the level $J_{\max}$. As a result, we obtain the full telescoping series given by
\begin{align} \label{eq:tele}
	 \mathcal{B}(g_1, g_2) &= \sum_{\substack{j_1, j_2\in N_{J_0} \\ |j_1-j_2| > 1}}  \mathcal{B}(Q^{J_0}_{j_1}g_1, Q^{J_0}_{j_2}g_2) \nonumber \\
	& + \sum_{J_0<J \le J_{\max} } \ \sum_{\substack{j_1, j_2\in N_J \\ 1<|j_1-j_2|<6}} \ \sum_{\substack{j_1', j_2'\in N_{J-1} \\ |j_1'-j_2'|\le 1}}   \mathcal{B}(Q^{J}_{j_1} Q^{J-1}_{j_1'} g_1, Q^{J}_{j_2} Q^{J-1}_{j_2'} g_2) \nonumber \\
	& + \sum_{\substack{j_1, j_2\in N_{J_{\max}} \\ |j_1-j_2|\le 1}}   \mathcal{B}(Q^{J_{\max}}_{j_1} g_1, Q^{J_{\max}}_{j_2}  g_2).
\end{align}
(For convenience, the roles of $j_i$ and $j_i'$ have been exchanged in the triple sum.) We point out that the above multiscale decomposition works for general bilinear forms. In the next steps, we shall omit the ranges like $j_1, j_2 \in N_{J}$ or $j_1', j_2' \in N_{J-1}$ in the summations for simplicity.

\medskip
\textbf{Step 5.} Bilinear estimate for non-overlapping case.
\smallskip

Here, we treat the triple sum term in \eqref{eq:tele}. Note that the estimate for the first term in \eqref{eq:tele} is similar by taking $J=J_0 = 5$ and removing the constraint $|j_1-j_2|<6$ as well as the $Q^{J-1}$ localizations. Take any $J \in [J_0+1, J_{\max}] \cap \Z$. Given $j_1, j_2 \in N_J$ such that $1<|j_1 - j_2| <6$, we write
\begin{align}
	|\eta_h|-|(\xi-\eta)_h| &= (|\eta_h| - \r_{j_1,1}^J) -(|(\xi-\eta)_h| - \r_{j_1,1}^J) \nonumber \\
	&= 2^{-2J} \left(2^{2J}(|\eta_h| - \r_{j_1,1}^J) -2^{2J}(|(\xi-\eta)_h| - \r_{j_1,1}^J)\right).
\end{align}
(See the definition of $\rho_{{j_1},1}^J$ in Section \ref{sec:52}.) The smallness factor $2^{-2J}$ will be crucial for the estimates below. For short, we denote $Q_1$ for one of
\begin{equation}
	Q^J_{j_1} Q_{j_1'}^{J-1} D \quad \mathrm{or} \quad   2^{2J} Q^J_{j_1} Q_{j_1'}^{J-1} D \left(|\nabla_h|- \rho_{{j_1},1}^J\right),
\end{equation}
where $D \in \left\{1, \frac{\partial_1}{|\nabla_h|}, \frac{\partial_2}{|\nabla_h|}\right\}$, and similarly denote $Q_2$ for one of
\begin{equation}
	Q^J_{j_2} Q_{j_2'}^{J-1} D   \quad \mathrm{or} \quad 2^{2J} Q^J_{j_2} Q_{j_2'}^{J-1} D \left(|\nabla_h|- \rho_{{j_1},1}^J\right).
\end{equation}
Moreover, in the estimates below we implicitly take supremum over all possible versions of $Q_1$ or $Q_2$ if they show up. Note that the multipliers represented by $Q_{2}$ enjoy similar support properties on the Fourier side as the cut-off function $\phi_{j_2}^J(\ln \r)$. Indeed, for $\eta$ from the support of $Q^J_{j_2}$, there holds that 
\begin{equation}\label{eq:51-12}
		\left||\eta_h|- \rho_{{j_1},1}^J\right| \lesssim 2^{-2J}, \quad \left|S_\eta \left(|\eta_h| - \rho_{{j_1},1}^J\right)\right| \lesssim 1,
\end{equation}
and moreover,
\begin{equation} \label{eq:51-11}
 \left|S_\eta^2 \left(|\eta_h| - \rho_{{j_1},1}^J\right)\right| \lesssim 1.
\end{equation}
(The last property \eqref{eq:51-11} is in fact stronger than what we need.) Using \eqref{eq:51-12}--\eqref{eq:51-11}, we see that the proof of Lemma \ref{lem:51-3} can also be applied to $Q_2 g_2$ instead of $Q^J_{j_2} g_2$. By the same reason, Lemma \ref{lem:51-3} also applies to $Q_1 g_1$ instead of $Q^J_{j_1} g_1$. 

Next, we make a decomposition $\mathcal{C} = \mathcal{C}_1 \cup \mathcal{C}_2$ with $\mathcal{C}_1 = \{x \in \mathcal{C}: |x_h| < 2^{\frac{2J}{3}}\}$ and $\mathcal{C}_2 = \mathcal{C}\, \setminus \, \mathcal{C}_1$. Due to $|j_1-j_2|>1$, we have
\begin{equation} \label{eq:804-001}
	B^J_{j_1}(0) \cap B^J_{j_2}(0) = \emptyset.
\end{equation}
Using \eqref{eq:804-001} and Lemma \ref{lem:51-3} (i)\&(iii), on the set $B^J_{j_1}(0)$ there holds that
\begin{align} \label{eq:51-21}
	&\quad \| \mathcal{B}( Q^J_{j_1} Q^{J-1}_{j_1'} g_1,  Q^J_{j_2} Q^{J-1}_{j_2'} g_2)\|_{L^2(B^J_{j_1}(0))} \nonumber\\ 
	&\lesssim 2^{-2J} \Big\{ 2^{\frac{2J}{3}}\|e^{is\Lmd} Q_1 g_1\|_{L^\infty(B^J_{j_1}(0) \cap \mathcal{C}_1)} \|e^{is\Lmd} Q_2 g_2\|_{L^\infty_{x_h}L^2_{x_3}(B^J_{j_1}(0) \cap \mathcal{C}_1 ) } \nonumber \\
	&\quad \quad \quad \quad  + \|e^{is\Lmd} Q_1 g_1\|_{L^\infty(B^J_{j_1}(0) \cap \mathcal{C}_2)} \|e^{is\Lmd} Q_2 g_2\|_{L^2(B^J_{j_1}(0) \cap \mathcal{C}_2 ) } \Big\} \nonumber\\
	&\lesssim 2^{-2J}  2^{-\frac{J}{3}-m}   \|F_1\|_{X^{0,3}_{\beta'}} \big(\|Z_{\sim m-2J}  Q^J_{j_2} F_2\|_{L^2} + C_K 2^{-mK} \|F_2\|_{X^0_0}\big) \nonumber\\
	&\lesssim 2^{-m-\frac{7J}{3}}  \|F_1\|_{X^{0,3}_{\beta'}}  \big(\|Z_{\sim m-2J} Q^J_{j_2} F_2\|_{L^2}+ C_K 2^{-mK} \|F_2\|_{X^0_0}\big),
\end{align}
and on the set $B^J_{j_1}(l)$ with $l \ge 1$ we have

\begin{align} \label{eq:51-22}
	&\quad \| \mathcal{B}( Q^J_{j_1} Q^{J-1}_{j_1'} g_1,  Q^J_{j_2} Q^{J-1}_{j_2'} g_2)\|_{L^2(B^J_{j_1}(l))} \nonumber\\ 
	&\lesssim 2^{-2J} \Big\{ 2^{\frac{2J}{3}} \|e^{i s\Lmd}Q_1 g_1\|_{L^\infty_{x_h}L^2_{x_3}(B^J_{j_1}(l) \cap \mathcal{C}_1)} \|e^{i s \Lmd}Q_2 g_2\|_{L^\infty(B^J_{j_1}(l) \cap \mathcal{C}_1 ) } \nonumber\\
	& \quad \quad \quad + \|e^{i s \Lmd} Q_1 g_1\|_{L^2(B^J_{j_1}(l) \cap \mathcal{C}_2)} \|e^{i s \Lmd} Q_2 g_2\|_{L^\infty(B^J_{j_1}(l) \cap \mathcal{C}_2 ) } \Big\} \nonumber\\
	&\lesssim 2^{-2J}  2^{-\frac{J}{3}-m}   \|F_2\|_{X^{0,3}_{\beta'}} \big(\|Z_{\sim m-2J+l}  Q^J_{j_1} F_1\|_{L^2} + C_K 2^{-mK} \|F_1\|_{X^0_0}\big) \nonumber\\
	&\lesssim 2^{-m-\frac{7J}{3}}  \|F_2\|_{X^{0,3}_{\beta'}}   \big(\|Z_{\sim m-2J+l} Q^J_{j_1} F_1\|_{L^2} + C_K 2^{-mK} \|F_1\|_{X^0_0}\big).
\end{align}
Here, $K$ stands for an arbitrary positive integer and $C_K$ depends only on $K$. 

For each $j_1$, the  number of possible choices of $(j_2, j_1', j_2')$ is bounded by $32$. Moreover, for each $j_1$ and $l$, the length in the $x_3$ direction of $B^J_{j_1}(l)$ is comparable with $2^{m-2J+l}$, hence the number of $j \in N_J$ such that $B^J_j(l) \cap B^J_{j_1}(l) \neq \emptyset$ is controlled by $2^l$. Applying Cauchy-Schwarz inequality and using \eqref{eq:51-22}, for each $1 \le l \lesssim m-2J$ we get that
\begin{align} \label{eq:51-31}
	&\quad \ \Bigg\| \sum_{ 1<|j_1-j_2|<6} \ \sum_{|j_1'-j_2'|\le 1}  \mathbf{1}_{B^J_{j_1}(l)} \cdot  \mathcal{B}(Q^{J}_{j_1} Q^{J-1}_{j_1'} g_1, Q^{J}_{j_2} Q^{J-1}_{j_2'} g_2) \Bigg\|_{L^2} \nonumber \\
	&\lesssim \left\|\left(\sum_{ 1<|j_1-j_2|<6} \ \sum_{ |j_1'-j_2'|\le 1} \bigg|   \mathbf{1}_{B^J_{j_1}(l)} \cdot  \mathcal{B}(Q^{J}_{j_1} Q^{J-1}_{j_1'} g_1, Q^{J}_{j_2} Q^{J-1}_{j_2'} g_2) \bigg|^2 \right)^\frac12 \right.  \nonumber \\
	&\quad \quad \quad  \cdot \left. \left(\sum_{ 1<|j_1-j_2|<6} \ \sum_{ |j_1'-j_2'|\le 1} \mathbf{1}_{B^J_{j_1}(l)} \right)^\frac12 \right\|_{L^2} \nonumber \\
	&\lesssim 2^{\frac{l}{2}} \left( \sum_{ 1<|j_1-j_2|<6} \ \sum_{ |j_1'-j_2'|\le 1} \left\|   \mathbf{1}_{B^J_{j_1}(l)} \cdot  \mathcal{B}(Q^{J}_{j_1} Q^{J-1}_{j_1'} g_1, Q^{J}_{j_2} Q^{J-1}_{j_2'} g_2) \right\|_{L^2}^2  \right)^\frac12  \nonumber \\
	&\lesssim 2^{\frac{l}{2}} \left( \sum_{ 1<|j_1-j_2|<6} \ \sum_{|j_1'-j_2'|\le 1}  \left( 2^{-m-\frac{7J}{3}}  \|F_2\|_{X^{0,3}_{\beta'}}  \|Z_{\sim m-2J+l} Q^J_{j_1} F_1\|_{L^2} \right)^2 \right)^\frac12 \nonumber \\
	&\quad \ + C_K 2^{-mK} \|F_2\|_{X_{\beta'}^{0,3}} \|F_1\|_{X^0_0} \nonumber \\
	&\lesssim   2^{\frac{l}{2}-m-\frac{7J}{3}}  \|F_2\|_{X^{0,3}_{\beta'}}  \|Z_{\sim m-2J+l} F_1\|_{L^2} + C_K 2^{-mK} \|F_2\|_{X_{\beta'}^{0,3}} \|F_1\|_{X^0_0}.
\end{align}
In the last inequality, we rely on the orthogonality of the $Q^J_j$ operators in $j$. Similarly, for $l = 0$, using \eqref{eq:51-21} we have
\begin{align} \label{eq:51-41}
	&\quad \ \Bigg\| \sum_{ 1<|j_1-j_2|<6} \ \sum_{ |j_1'-j_2'|\le 1}  \mathbf{1}_{B^J_{j_1}(0)} \cdot  \mathcal{B}(Q^{J}_{j_1} Q^{J-1}_{j_1'} g_1, Q^{J}_{j_2} Q^{J-1}_{j_2'} g_2) \Bigg\|_{L^2} \nonumber \\
	&\lesssim   2^{-m-\frac{7J}{3}}   \|F_1\|_{X^{0,3}_{\beta'}} \|Z_{\sim m-2J} F_2\|_{L^2} + C_K 2^{-mK} \|F_1\|_{X_{\beta'}^{0,3}} \|F_2\|_{X^0_0}.
\end{align}
 Then, summing \eqref{eq:51-31} and \eqref{eq:51-41} over $0 \le l \lesssim m-2J$, we get
\begin{align} \label{eq:51-51}
		&\quad \ \Bigg\| \sum_{ 1<|j_1-j_2|<6} \ \sum_{ |j_1'-j_2'|\le 1 }    \mathcal{B}(Q^{J}_{j_1} Q^{J-1}_{j_1'} g_1, Q^{J}_{j_2} Q^{J-1}_{j_2'} g_2) \Bigg\|_{L^2} \nonumber \\
		&\lesssim \sum_{l \lesssim m-2J} 2^{\frac{l}{2}-m-\frac{7J}{3}}  \big(\| F_2\|_{X^{0,3}_{\beta'}} \|Z_{\sim m-2J+l} F_1\|_{L^2} + \| F_1\|_{X^{0,3}_{\beta'}} \|Z_{\sim m-2J+l} F_2\|_{L^2}\big)  \nonumber \\
		&\quad \ + C_K 2^{-mK} \| F_1\|_{X^{0,3}_{\beta'}}  \| F_2\|_{X^{0,3}_{\beta'}} \nonumber \\
		&\lesssim \sum_{l \lesssim m-2J} 2^{\frac{l}{2}}    2^{-(1+\beta) (m-2J+l)} 2^{-m-\frac{7J}{3}}  \big( \| F_2\|_{X^{0,3}_{\beta'}} \| F_1\|_{X^{0}_{\beta}} + \| F_1\|_{X^{0,3}_{\beta'}} \| F_2\|_{X^{0}_{\beta}}\big) \nonumber \\
		&\quad \ + C_K 2^{-mK} \| F_1\|_{X^{0,3}_{\beta'}}  \| F_2\|_{X^{0,3}_{\beta'}} \nonumber \\
	&\lesssim 2^{-(2+\beta) m} 2^{(2\beta-\frac13) J} \big( \| F_2\|_{X^{0,3}_{\beta'}} \| F_1\|_{X^{0}_{\beta}} + \| F_1\|_{X^{0,3}_{\beta'}} \| F_2\|_{X^{0}_{\beta}} + \| F_1\|_{X^{0,3}_{\beta'}}  \| F_2\|_{X^{0,3}_{\beta'}} \big).
\end{align}
Summing the above estimate over $J_0 \le J \le J_{\max}$, we receive an acceptable contribution to \eqref{eq:804-b1}.


\medskip
\textbf{Step 6.} The final term in \eqref{eq:tele}.
\smallskip

Consider $J = J_{\max}$ and $|j_1-j_2| \le 1$. The $L^2$ estimate on $B_{j_1}^J(l)$ with $l \ge 1$ is the same as in \eqref{eq:51-22}, \emph{i.e.}, there holds that
\begin{align}
	 \| \mathcal{B}( Q^J_{j_1} g_1,  Q^J_{j_2} g_2)\|_{L^2(B^J_{j_1}(l))} \lesssim 2^{-m-\frac{7J}{3}}  \|F_2\|_{X^{0,3}_{\beta'}}  (\|Z_{\sim m-2J+l} Q^J_{j_1} F_1\|_{L^2} + C_K 2^{-mK} \|F_1\|_{X^0_0}).
\end{align}
Consequently, similar to \eqref{eq:51-31} and \eqref{eq:51-51}, here we have
\begin{align*}
	\sum_{1 \le l \lesssim m-2J} \Bigg\|\sum_{ |j_1 - j_2| \le 1} \mathbf{1}_{B^J_{j_1}(l)} \cdot  \mathcal{B}( Q^J_{j_1} g_1, Q^J_{j_2}  g_2)\Bigg\|_{L^2} \lesssim   2^{-(2+\beta) m} 2^{(2\beta-\frac13) J} \| F_2\|_{X^{0,3}_{\beta'}} \|F_1\|_{X^0_\beta}. 
\end{align*}
However, the estimate on the set $B^J_{j_1}(0)$ requires a different method. By Lemma \ref{lem:51-3},  we have the following pointwise bound for $x \in B^J_{j_1}(0)$:
\begin{align} \label{eq:51-52}
	| \mathcal{B}( Q^J_{j_1} g_1,  Q^J_{j_2} g_2)| &\lesssim 2^{-2J}|e^{i t \Lmd} Q_1 g_1| |e^{i t \Lmd} Q_2 g_2| \nonumber  \\
	&\lesssim  2^{-2m-2J} \langle |x_h| \rangle^{-1} \|F_1\|_{X^{0,3}_{\beta'}} \|F_2\|_{X^{0,3}_{\beta'}}.
\end{align}
Hence, directly integrating the square of \eqref{eq:51-52} in $B^J_{j_1}(0)$ and using that $J = J_{\max} \sim \frac{\gamma m}{2}$, we obtain that
\begin{align} \label{eq:804-a1}
	\| \mathcal{B}( Q^J_{j_1} g_1,  Q^J_{j_2} g_2)\|_{L^2(B^J_{j_1}(0))} &\lesssim m \,2^{-2m-2J} 2^{\frac12(m-2J)} \|F_1\|_{X^{0,3}_{\beta'}} \|F_2\|_{X^{0,3}_{\beta'}}   \nonumber \\
	&\lesssim 2^{-\frac32 m (1+\gamma-\delta)} \|F_1\|_{X^{0,3}_{\beta'}} \|F_2\|_{X^{0,3}_{\beta'}}.
\end{align}
Since the family of sets $\{B_{j_1}^J(0)\}_{j_1 \in N_J}$ is locally finite, we deduce that
\begin{align}
 \Big\| \sum_{ |j_1-j_2|\le 1} \mathbf{1}_{B^J_{j_1}(0)} \cdot  \mathcal{B}( Q^J_{j_1} g_1,  Q^J_{j_2} g_2)\Big\|_{L^2} &\lesssim |N_J|^\frac12 2^{-\frac32 m (1+\gamma-\delta)} \|F_1\|_{X^{0,3}_{\beta'}} \|F_2\|_{X^{0,3}_{\beta'}} \nonumber \\
	&\lesssim 2^{-\frac32 m (1+\frac23 \gamma-\delta)} \|F_1\|_{X^{0,3}_{\beta'}} \|F_2\|_{X^{0,3}_{\beta'}}.
\end{align}
Due to the constraint $\gamma > \frac12$ and by taking $\delta$ and $\beta$ small, we receive an acceptable contribution to \eqref{eq:804-b1}.

\subsubsection{Estimate for $S_4'$} \label{sec:743}

For simplicity, here we denote 
\begin{equation}
	h_1 = P^{\hp, \le -\frac{9m}{10}}f_1, \quad h_2 = P^{\hp, \le -\frac{9m}{10}}f_2.
\end{equation}
As in the study of $S_3'$, we also localize the output frequency in parameter $q$ as follows.
\begin{align} \label{eq:423-1}
	\Q[\m^{(3)}](h_1, h_2) &= \sum_{q \ge -\frac{9m}{10}+5} P^{\hp,q}_k \Q[\m^{(3)}](h_1, h_2) \nonumber \\
	&\quad + P^{\hp,\le -\frac{9m}{10}+4}_k \Q[\m^{(3)}](h_1, h_2).
\end{align}
For $q \ge -\frac{9m}{10}+5$, we perform a normal form as in 
\eqref{eq:51-01} or \eqref{eq:51-3} (depending on whether $l+k \le m+4$ or $l+k \ge m+5$), and obtain acceptable contributions using the GPW null condition  (see Proposition \ref{prop:nonlinear}) and set size gain (see \cite[Section A.2]{GuoInvent}):
\begin{align*}
	\|P^{\hp, q}_k \Q[\Phi^{-1} \m^{(3)}](h_1, h_2)\|_{L^2}  	&\lesssim 2^{k+\frac{3k_1}{2}-\frac{9m}{20}} \|h_1\|_{L^2}  \|h_2\|_{L^2} \\
	&\lesssim 2^{k+\frac{3k_1}{2}-\frac{27m}{20}} \|F_1\|_{X^{0,3}_{\beta'}} \|F_2\|_{X^{0,3}_{\beta'}} \\
	&\lesssim 2^{k+\frac{3k_1}{2}-\frac{27m}{20}} 2^{-20 k_1^+ - 20 k_2^+} \varepsilon_1^2,
\end{align*} 
\begin{align*}
	\|P^{\hp, q}_k \Q[\Phi^{-1} \m^{(3)}](\partial_t h_1, h_2)\|_{L^2}  	&\lesssim 2^{k+\frac{3k_1}{2}-\frac{9m}{20}} \|\partial_t h_1\|_{L^2}  \|h_2\|_{L^2} \\
	&\lesssim 2^{k+\frac{3k_1}{2} - 20 k_1^+ -\frac{9m}{20}} 2^{(-\frac32 +\delta_0)m - \frac{9}{20} m} \varepsilon_1^2 \|F_2\|_{X^{0,3}_{\beta'}} \\
	&\lesssim 2^{k+\frac{3k_1}{2}-\frac{9m}{4}} 2^{-20 k_1^+ - 20 k_2^+} \varepsilon_1^3,
\end{align*} 
and in the case $l+k \ge m+5$,
\begin{align*} 
	\langle P^{\hp, q}_k \Q[\Phi^{-1} \m^{(3)}](h_1, h_2), Z_l^2 \partial_t f \rangle_{L^2}  
	&\lesssim \|P^{\hp, q}_k \Q[\Phi^{-1} \m^{(3)}](h_1, h_2)\|_{L^2} \|Z_l \partial_t  f\|_{L^2} \nonumber \\
	&\lesssim 2^{k+\frac{3k_1}{2}-\frac{67m}{20}} 2^{-30 k^+ -20 k_1^+ - 20 k_2^+} \varepsilon_1^4 \nonumber \\
	&\lesssim 2^{k+\frac{3k_1}{2}-\frac{5m}{4}} 2^{-30 k^+ -20 k_1^+ - 20 k_2^+} 2^{-2(1+\beta)(l+k)} \varepsilon_1^4. 
\end{align*}
For the last term on the right of \eqref{eq:423-1}, we directly use the GPW null condition  and set size gain to get that
\begin{align*}
	\|P^{\hp, \le -\frac{9m}{10}+4}_k \Q[\m^{(3)}](h_1, h_2)\|_{L^2}
	&\lesssim 2^{k+\frac{3k_1}{2}-\frac{27m}{20}} \|h_1\|_{L^2}  \|h_2\|_{L^2} \\
	&\lesssim 2^{k+\frac{3k_1}{2}-\frac{9m}{4}} \|F_1\|_{X^{0,3}_{\beta'}} \|F_2\|_{X^{0,3}_{\beta'}} \\
	&\lesssim 2^{k+\frac{3k_1}{2}-\frac{9m}{4}} 2^{-20 k_1^+ - 20 k_2^+} \varepsilon_1^2.
\end{align*}

\subsection{The propagation of $X_{\beta'}$ norm} \label{sec:Xbetap}

In this section, we relax the constraint $a+b \le 10$ to $a+b \le 20$. By interpolation between \eqref{eq:BA} and \eqref{eq:slow-growth}, we have
\begin{equation} \label{eq:511-999}
	\|F_1\|_{X^{30, 30}_{\beta-\delta_0}} + \|F_2\|_{X^{30, 30}_{\beta-\delta_0}} \lesssim \langle t \rangle^{C\varepsilon_1} \varepsilon_1.
\end{equation}
Repeating the arguments in Sections \ref{sec:high-low-hp}--\ref{sec:high-high}, but using \eqref{eq:511-999} instead of the uniformly bounded bootstrap norms, we obtain that
\begin{align} \label{eq:511-1-1}
	\sum_{k_1, k_2}    \int_{\mathcal{I}_m} \left\|  Z_l^{(k)} P_k^\hp \Q_\mua[\m_\mua] (f_1, f_2)  \right\|_{L^2}  ds   \lesssim 2^{-30 k^+ -(1+\beta-\delta_0 - C \varepsilon_1)m }  \, \varepsilon_1^2
\end{align}
in the case of $l+k \le m+4$, and
	\begin{align} \label{eq:511-2-2}
	& \left|\sum_{k_1, k_2} \sum_{\mua \in \{\pm\}^3} \sum_{\substack{a_1+a_2 \le a \\ b_1+b_2 \le b}} C_{\bar{a}, \bar{b}}  \int_{\mathcal{I}_m} \langle  Z_l P_k^\hp \Q_\mua[\m_\mua] (f_1, f_2)    ,Z_l f\rangle_{L^2} ds \right| \nonumber\\
	& \lesssim 2^{-60k^+ - 2(1+\beta-\delta_0)(l+k)} \left(2^{-\delta m} + 2^{- |m-l| } + \mathbf{1}_{l+k \le m+15}\right) 2^{C\varepsilon_1 m}  \varepsilon_1^2
\end{align}
in the case of $m+5 \le l+k \le (1+\delta) m$, which imply the the desired bounds \eqref{eq:511-901} and \eqref{eq:511-902} respectively.

\section{Other types of interactions} \label{sec:9}

Throughout  this section, we assume that \eqref{eq:BA} holds and work with $a+b \le 20, \  l+k < (1+\delta)m$. Write $F_i$ for some component of $S^{a_i} \bar{\Omega}^{b_i} \f_{\mu_i}, \mu_i \in \{\pm\}, \ i=1,2$.  The goal is to show that, for $\bar{\iota} = (\iota, \iota_1, \iota_2)\neq (\bighp, \bighp, \bighp)$,
\begin{align} \label{eq:808-001}
	\left\|\int_{\mathcal{I}_m} P_k^\iota \Q_\mua[\m] (P^{\iota_1} F_1, P^{\iota_2} F_2)  ds \right\|_{L^2} &\lesssim  2^{-20 k^+ - (1+\beta+3\delta)m} \varepsilon_1^2 \nonumber \\
	&\lesssim 2^{-20 k^+ - (1+\beta)(l+k) - \delta m} \varepsilon_1^2,
\end{align}
in which we assume that $\m= |\xi|\m(|\xi|^{-1}\xi)$ is smooth away from the origin.  In the case when $\iota=\bighp$, $(\iota_1, \iota_2) \neq (\bighp, \bighp)$ and $l+k \in [m+5, (1+\delta)m)$, we also prove that
\begin{align}
	\left\|\int_{\mathcal{I}_m} \langle Z_l P_k^\hp \Q_\mua[\m] (P^{\iota_1} F_1, P^{\iota_2} F_2), Z_l P_k^\hp S^a \Omega^b \Fp_\mu \rangle_{L^2}  ds \right\|_{L^2} &\lesssim 2^{-40 k^+ - (2+2\beta+5\delta) m}  \varepsilon_1^3 \nonumber \\ 
	&\lesssim 2^{-40 k^+ - 2(1+\beta)(l+k) - \delta m}  \varepsilon_1^3.
\end{align}
Since the signs in the phase play no role in the proof blow, for simplicity we shall work with $\mua = (+,+,+)$ and drop it from the notations. As usual, we dyadically localize the inputs into $P^{\iota_1}_{k_1}$ and $P^{\iota_2}_{k_2}$ pieces. Similar to \eqref{eq:806-aa1}--\eqref{eq:not-too-large-k}, in all types of interactions considered below we assume that $k_{\max } < \delta_0 m$ and $k_{\min} > -3m$ hold. Moreover, we shall often make use of the estimate 
\begin{equation}
	\|\f_\pm\|_{(X\cap Y)^{30, 30}_{\beta-\delta_0}} \lesssim \langle t \rangle^{C\varepsilon_1} \varepsilon_1,
\end{equation}
which follows from interpolation between \eqref{eq:BA} and \eqref{eq:slow-growth}.




\subsection{$\bighp + \bighp \to \bigvp$ interactions} \label{sec:8-1}

Here, we give the bilinear $L^2$ estimate on $ \int_{\mathcal{I}_m} {P}_k^\vp \Q[\m](f_1, f_2) ds$
with $f_i = P_{k_i}^\hp F_i, \ i =1,2$. 

\medskip
We decompose the first input as 
\begin{equation}
	f_1 = P^{\hp, \le -5} f_1 +  P^{\hp, > -5} f_1
\end{equation}
where $P^{\hp, >-5} = \sum_{-5<q\le -1} P^{\hp, q}$, and also introduce the following  decomposition of the phase space towards a normal form:
\begin{equation}
	1 \equiv  \chi^{\res}(\xi, \eta) +\chi^{\nr}(\xi, \eta) := \psi(2^{5} \Phi) + (1-\psi)(2^{5} \Phi).
\end{equation}
Hence, we can write
\begin{align} \label{eq:805-a2}
	{P}_k^\vp \Q_\mua[\m](f_1, f_2) &= 	{P}_k^\vp \Q_\mua[\m](P^{\hp, \le -5} f_1, f_2)  + {P}_k^\vp \Q_\mua[\m \chi^\res](P^{\hp, > -5} f_1, f_2) \nonumber \\
	&\quad + {P}_k^\vp \Q_\mua[\m \chi^\nr](P^{\hp, > -5} f_1, f_2).
\end{align}

First, we study the the first two terms on the right of \eqref{eq:805-a2}. Note that on the support of $\chi^\vp(\xi)\chi^{\hp, \le -5}(\xi-\eta)\chi^\hp(\eta)$ or $\chi^\res \chi^\vp(\xi)\chi^{\hp,>-5}(\xi-\eta)\chi^\hp(\eta)$, the triangle formed by $\xi, \xi-\eta, \eta$ is non-degenerate in the sense of Lemma \ref{lem:805-c1}. Using Lemma \ref{lem:805-c1}, there holds that
\begin{equation}
	|\bar{\sigma}| \sim 2^{k_{\max} + k_{\min}},
\end{equation}
and due to \eqref{eq:805-a1},  we have
\begin{equation} \label{eq:805-d1}
	|S_\eta \Phi| + |\Omega_\eta \Phi| \sim 2^{k_{\max} + k_{\min} - 2k_1}.
\end{equation}
Here, we   just show the estimate for $${P}_k^\vp \Q_\mua[\m \chi^\res](P^{\hp, > -5} f_1, f_2),$$ while the argument for $${P}_k^\vp \Q_\mua[\m](P^{\hp, \le -5} f_1, f_2)$$ is entirely similar. Make the decompositions
$$f_1 =  Z_{\le l_1} f_1 +  Z_{> l_1} f_1, \quad f_2 =  \wt{P}^\hp Z_{\le l_2} f_2 +  \wt{P}^\hp Z_{> l_2} f_2,$$  
with 
\begin{equation} 
	l_1 = (1-\delta)m-k_1, \quad l_2 = (1-\delta)m-k_2.
\end{equation}
To treat the $Z_{\le l_1} f_1$ part, we use integration by parts along vector fields $S_\eta$ and $\Omega_\eta$. By \eqref{eq:805-d1} and \eqref{eq:429-001}, each integration by parts in $V_\eta \in \{S_\eta, \Omega_\eta\}$ in the expression
$${P}_k^\vp \Q_\mua[\m \chi^\res](P^{\hp, > -5} Z_{\le l_1} f_1, f_2),$$
gains a factor of
\begin{equation} \label{eq:423-01}
	s^{-1}(2^{2k_1-k_{\max} - k_{\min}} (1+2^{k_2-k_1}) + 2^{l_1+k_1}).
\end{equation}
The factor $2^{k_2 - k_1}$ here is due to the ``cross terms", \emph{e.g.}, $\frac{V_\eta^2 \Phi}{V_\eta \Phi}$. The analysis here is similar to the derivation of \eqref{eq:factor-001}, and we have the following major terms arising from repeated integration by parts using \eqref{eq:429-001} (with $L = 2^{k_{\max} + k_{\min}}$):
\begin{enumerate}[(a)]
	\item There is always a gain of $s^{-1} 2^{2k_1-k_{\max}-k_{\min}}$ due to the denominator $s V_\eta \Phi$.
	\item $V_\eta$ landing on  $\chi^{\res}(\xi) {\chi}^{\hp, >-5}(\xi-\eta) \F (Z_{\le l_1} f_1)(\xi-\eta)$ gives a factor $\lesssim 2^{k_2-k_1}+2^{l_1-k_1+k_{\max}+k_{\min}}$ using \eqref{eq:805-bb1}--\eqref{eq:805-bb2}.
	\item $V_\eta$ landing on the second input $\wh{f}_2(\eta)$ gives a bounded factor.
	\item If $V_\eta$ lands on $\chi_{_{V_\eta}}$ or $\frac{1}{V_\eta\Phi}$ or coefficients from previous steps of integration by parts, this produces a factor $\lesssim 1 + 2^{k_2-k_1}$  using the identities \eqref{eq:805-bb4}--\eqref{eq:805-bb7}. Note that in this case we have
	\begin{equation}
		|\Omega_\eta \bar\sigma| = |\xi_3 \eta_h| \lesssim 2^{k+k_2} \lesssim L (1+2^{k_2-k_1}).
	\end{equation}
\end{enumerate}
Note that if $k_2 \gg k_1$, then
$$2^{2k_1-k_{\max} - k_{\min}} (1+2^{k_2-k_1}) \sim 2^{k_1+k_2 -k_{\max} - k_{\min} } \sim 1. $$
Hence, there holds that
\begin{equation} \label{eq:423-02}
	\eqref{eq:423-01} \lesssim s^{-1}	(1+2^{2k_1-k_{\max} - k_{\min}} + 2^{l_1+k_1}) \lesssim 	s^{-1} 2^{2k_1-k_{\max} - k_{\min}} + s^{-\delta}.
\end{equation}
Similarly, for the $Z_{\le l_2} f_2$ part, each integration by parts along $V_{\xi-\eta}$ gains a factor bounded by
$$	s^{-1} 2^{2k_2-k_{\max} - k_{\min}} + s^{-\delta}.$$

\smallskip
\emph{Case 1.} $k_{\max} + k_{\min} - 2k_1 \lesssim -(1-\delta)m$ \ or \ $k_{\max} + k_{\min} - 2k_2 \lesssim -(1-\delta)m$. 
\smallskip

In this case, repeated integration by parts in $V_\eta$ or $V_{\xi-\eta}$ is not beneficial. Instead, we rely on set size gain, multiplier bounds and linear dispersive estimates. Let us only treat the case $k_{\max} + k_{\min} - 2k_1 \lesssim -(1-\delta)m$ as the other case is similar. If $k = k_{\min}$, then $k_1 \sim k_2 \sim k_{\max}$ and we have
\begin{align*}
	\|	P^\vp_k \Q[\m \chi^\res](P^{\hp, >-5} f_1, f_2)\|_{L^2} &\lesssim 2^{k+\frac{3k}{2}} \|f_1\|_{L^2} \|f_2\|_{L^2} \\
	&\lesssim 2^{\frac{5k}{2} - 30k_1^+ - 30 k_2^+} \varepsilon_1^2 \\
	&\lesssim 2^{-\frac{5}{2}(1-\delta)m + \frac{5k_1}{2} - 30k_1^+ - 30 k_2^+} \varepsilon_1^2, 
\end{align*}
which is acceptable.  If  $k_2 = k_{\min}$, then $k_1 \sim k \sim k_{\max}$ and using \cite[Lemma A.8]{GuoInvent} (with direct adaptations) we have 
\begin{align*}
	\|	P^\vp_k \Q[\m \chi^\res](P^{\hp, >-5}  f_1, f_2)\|_{L^2} &\lesssim 2^{k} \|e^{i s \Lmd} f_2\|_{L^\infty} \|f_1\|_{L^2} \\
	&\lesssim 2^{k + \frac{3k_2}{2} - 30k_1^+ - 30 k_2^+} 2^{-m}  \varepsilon_1^2 \\
	&\lesssim 2^{k +  \frac{3}{2} k_1 - 30k_1^+ - 30 k_2^+} 2^{-m - \frac{3}{2} (1-\delta) m}  \varepsilon_1^2,
\end{align*}
which is again acceptable. If $k_1 = k_{\min}$, then $k_2 \sim k \sim k_{\max}$, but this is excluded by the assumption $k_{\max} + k_{\min} - 2k_1 \lesssim -(1-\delta)m$.

\medskip
\emph{Case 2.} $k_{\max} + k_{\min} - 2k_1 \gtrsim -(1-\delta)m$ and $k_{\max} + k_{\min} - 2k_2 \gtrsim -(1-\delta)m$.
\smallskip

In this case, the factor \eqref{eq:423-02} is bounded by $s^{-\delta}$, and repeated integration by parts in $V_\eta$ leads to arbitrary fast decay in time if the first input is replaced by $P^{\hp, >-5}  Z_{\le l_1} f_1$. Integration by parts in $V_{\xi - \eta}$ is similarly effective if the second input is replaced by $\wt{P}^\hp Z_{\le l_2} f_2$.  Hence, it is sufficient to treat
the $Z_{> l_i} f_i, i=1,2$ interactions, for which we receive an acceptable contribution as
\begin{align*}
	\|	P^\vp_k \Q[\m \chi^\res](P^{\hp, >-5}  Z_{> l_1} f_1, \wt{P}^{\hp} Z_{> l_2} f_2)\|_{L^2} &\lesssim 2^{k+\frac{3k_{\min}}{2}} \|Z_{> l_1} f_1\|_{L^2} \|Z_{> l_2} f_2\|_{L^2} \\
	&\lesssim 2^{k+\frac{3k_{\min}}{2}} 2^{-2(1+\beta')(1-\delta) m} \|F_1\|_{X^{0}_{\beta'}} \|F_2\|_{X^{0}_{\beta'}} \\
	&\lesssim 2^{k+\frac{3k_{\min}}{2} - 30 k_1^+ - 30 k_2^+} 2^{-2(1+\beta')(1-\delta) m + C \varepsilon_1 m} \varepsilon_1^2. 
\end{align*}

\medskip

Next, we study the non-resonant part in \eqref{eq:805-a2}, namely, $P_k^\vp \Q[\m \chi^\nr](P^{\hp, > -5} f_1, f_2)$. Performing a normal form, we get
\begin{align}
	&\quad \int_{2^m}^{2^{m+1}\wedge t} P_k^\vp \Q[\Phi^{-1} \m \chi^\nr](P^{\hp, > -5} f_1, f_2) ds  \nonumber\\
	& = - i P_k^\vp \Q[\Phi^{-1} \m \chi^\nr](P^{\hp, > -5}  f_1, f_2)|_{s = 2^{m+1} \wedge t} \nonumber \\
	& \quad + i P_k^\vp \Q[\Phi^{-1} \m \chi^\nr](P^{\hp, > -5}  f_1, f_2)|_{s = 2^{m}} \nonumber \\
	& \quad + i \int_{2^m}^{2^{m+1}\wedge t} P_k^\vp \Q[\Phi^{-1} \m \chi^\nr](P^{\hp, > -5} \partial_t f_1, f_2) ds \nonumber \\
	&\quad + i \int_{2^m}^{2^{m+1}\wedge t} P_k^\vp \Q[\Phi^{-1} \m \chi^\nr](P^{\hp, > -5} f_1, \partial_t f_2) ds.
\end{align}
For the first two terms, applying \cite[Lemma A.8]{GuoInvent}  and Lemma \ref{lem:hp-disp} there holds that
\begin{align} \label{eq:512-1}
	&\quad \ \|P_k^\vp \Q[\Phi^{-1} \m \chi^\nr](P^{\hp, > -5} f_1, f_2)\|_{L^2} \nonumber \\
	&\lesssim 2^{k} \|e^{i s \Lmd} P^{\hp, > -5} f_1^I\|_{L^\infty} \|f_2\|_{L^2} + 2^k \|P^{\hp, > -5} f_1^{II}\|_{L^2} \|e^{ i s \Lmd} f_2\|_{L^\infty} \nonumber \\
	&\lesssim 2^{k+ \frac32 k_1 -\frac32 m} \|F_1\|_{X^{0,3}_{\beta'}} \|F_2\|_{X^0_{\beta'}}  + 2^{k + \frac32 k_2 - (2+\beta') m} \|F_1\|_{X^0_{\beta'}} \|F_2\|_{X^{0,3}_{\beta'}}. \nonumber \\
	&\lesssim 2^{k+ \frac32 \max \{k_1, k_2\} - 30 k_1^+ - 30 k_2^+} 2^{-(\frac32-\delta) m} \varepsilon_1^2.
\end{align}
Hence, the estimate  \eqref{eq:512-1} gives an acceptable contribution.  For the second term in \eqref{eq:512-1}, we have
\begin{align} 
	\|P_k^\vp \Q[\Phi^{-1} \m \chi^\nr](P^{\hp, > -5} \partial_t f_1, f_2)\|_{L^2} &\lesssim 2^k \|\partial_t f_1\|_{L^2} \|e^{ i s \Lmd} f_2\|_{L^\infty} \nonumber \\
	&\lesssim 2^{k+k_1 + \frac32 k_2 - 30 k_1^+} 2^{-(\frac52 - \delta_0) m} \varepsilon_1^2 \|F_2\|_{X^{0,3}_{\beta'}} \nonumber \\
	 &\lesssim 2^{k+k_1 + \frac32 k_2 -30 k_1^+ - 30 k_2^+} 2^{-(\frac52 - \delta) m} \varepsilon_1^3.
\end{align}
The third term in \eqref{eq:512-1} is similar. Finally, in the case $l+k \ge m+4$, we are following the energy method approach, hence in addition we need
\begin{align} \label{eq:512-hahaha}
	&\quad \ \langle P_k^\vp \Q[\Phi^{-1} \m \chi^\nr](P^{\hp, > -5}  f_1, f_2), Z_l^2 \partial_t f \rangle_{L^2}  \nonumber \\
	&\lesssim  \|P_k^\vp \Q[\Phi^{-1} \m \chi^\nr](P^{\hp, > -5}  f_1, f_2)\|_{L^2}  \|\partial_t f\|_{L^2} \nonumber \\
	&\lesssim 2^{k+ \frac32 \max \{k_1, k_2\} - 30 k^+ - 30 k_1^+ - 30 k_2^+} 2^{-(\frac72-2\delta) m} \varepsilon_1^4,
\end{align}
where $f$ stands for $P_k^{\vp} S^a \Omega^b \Fp_\mu$ or a component of $P_k^{\vp} S^a \bar\Omega^b \f_\mu$ and we used \eqref{eq:512-1} and Lemma \ref{lem:51-1} in the last line.

\subsection{$\bigvp + \bigvp \to \bigvp$ interactions} \label{sec:9-2}

Here, we give the bilinear $L^2$ estimate on $ \int_{\mathcal{I}_m}	{P}_k^\vp \Q[\m](f_1, f_2) ds$ with $f_i = P_{k_i}^\vp F_i, i =1,2$. 

\medskip

In this case, there holds that $|\Phi| \gtrsim 1$, hence we can directly perform a normal form as
\begin{align} \label{eq:512-41}
   \int_{\mathcal{I}_m} {P}_k^\vp \Q[\m](f_1, f_2) ds 
	& = - i P_k^\vp \Q[\Phi^{-1} \m ](  f_1, f_2)|_{s = 2^{m+1} \wedge t} \nonumber \\
	& \quad + i P_k^\vp \Q[\Phi^{-1} \m ](  f_1, f_2)|_{s = 2^{m}} \nonumber \\
	& \quad + i \int_{\mathcal{I}_m} P_k^\vp \Q[\Phi^{-1} \m ]( \partial_t f_1, f_2) ds \nonumber \\
	&\quad + i \int_{\mathcal{I}_m} P_k^\vp \Q[\Phi^{-1} \m ]( f_1, \partial_t f_2) ds.
\end{align}
To treat the first two terms in \eqref{eq:512-41}, we further decompose the inputs as
\begin{equation} \label{eq:512-11}
	f_1 = \sum_{p_1 > -\frac{m}{2}} P^{\vp, p_1} f_1 + P^{\vp, \le -\frac{m}{2}} f_1
\end{equation}
and
\begin{equation} \label{eq:512-12}
	f_2 = \sum_{p_2 > -\frac{m}{2}} P^{\vp, p_2} f_2 + P^{\vp, \le -\frac{m}{2}} f_2.
\end{equation}
For the last terms in \eqref{eq:512-11} and \eqref{eq:512-12}, we think of $p_1 = -\frac{m}{2}$ and  $p_2 = - \frac{m}{2}$ respectively. For simplicity, we shall write for $j=1,2$,
\begin{equation}
	P^{\vp,(\le)p_j} := \begin{cases} P^{\vp, p_j}, \quad p_j > - \frac{m}{2}, \\
		P^{\vp, \le -\frac{m}{2}}, \quad p_j = - \frac{m}{2}.
	\end{cases}
\end{equation}
In the case that $p_1 \le p_2$, an $L^2$--$L^\infty$ bound with the help of  \cite[Lemma A.8]{GuoInvent} (with direct adaptation) and Lemma \ref{lem:vp-disp}  gives
\begin{align} \label{eq:806-71}
	&\quad \ \|P_k^\vp \Q[\Phi^{-1} \m](P^{\vp, (\le)p_1} f_1, P^{\vp,(\le)p_2} f_2)\|_{L^2} \nonumber \\	
	&\lesssim 2^k \| P^{\vp, (\le)p_1} f_1\|_{L^2} \|e^{ i s \Lmd } P^{\vp, (\le)p_2} f_2^I\|_{L^\infty} + 2^k \| e^{ i s \Lmd } P^{\vp, (\le)p_1} f_1\|_{L^\infty} \| P^{\vp,(\le)p_2} f_2^{II}\|_{L^2} \nonumber \\
	&\lesssim (2^{k+ p_1 + \frac32k_2 - p_2} t^{-\frac32} + 2^{k+\frac32 k_1 - (1+\beta')p_2} 2^{(-2-\beta')m} ) \|F_1\|_{Y^{0,3}_{\beta'}} \|F_2\|_{Y^{0,3}_{\beta'}} \nonumber \\
	&\lesssim 2^{k + \frac32 \max\{k_1,k_2\} - 30k_1^+ - 30 k_2^+} 2^{-(\frac32-\delta) m} \varepsilon_1^2. 
\end{align} 
	(For $p_2 = -\frac{m}{2}$, we simply take $f_2^I = f_2$ and $f_2^{II} = 0$.) The case $p_1 > p_2$ can be similarly treated by exchanging the role of $f_1$ and $f_2$. Summing \eqref{eq:806-71} over possible $p_{1,2}$, we receive an acceptable contribution. To treat the third terms in \eqref{eq:512-41}, we also use an $L^2$--$L^\infty$ bound to get
\begin{align} \label{eq:806-72}
	 \|P_k^\vp \Q[\Phi^{-1} \m ]( \partial_t f_1, f_2)\|_{L^2} &\lesssim 2^k \|\partial_t f_1\|_{L^2} \|e^{ i  s \Lmd} f_2\|_{L^\infty} \nonumber \\
	 &\lesssim 2^{k+k_1 + \frac32 k_2 -30k_1^+ -30k_2^+}2^{ - (\frac52 - \delta) m} \varepsilon_1^3.  
\end{align}
The estimate for the fourth term in \eqref{eq:512-41} is similar to \eqref{eq:806-72}. 


\subsection{$\bigvp + \bigvp \to \bighp$ interactions} \label{sec:93}

Here, we give the bilinear $L^2$ estimate on $\int_{\mathcal{I}_m}	{P}_k^\hp \Q[\m](f_1, f_2) ds$ with $f_i = P_{k_i}^\vp F_i, \ i=1,2$.

\medskip
Let $\bar{P}^{\hp} = 1- \wt{P}^{\vp}$, and further decompose the output using
\begin{equation}
	{P}^\hp = ({P}^\hp - \bar{P}^{\hp})+\bar{P}^{\hp}.
\end{equation} 
For the first part, due to $P^\hp - \bar{P}^\hp = \wt{P}^\vp P^\hp$, we can apply the method for $\vp+\vp \to \vp$ interaction. Indeed, the arguments in Section \ref{sec:9-2} still apply if $P^\vp$ is replaced by $\wt{P}^\vp$. (In the case when $l+k \ge m+5$, the energy estimate approach is compatible with normal forms due to Lemma \ref{lem:51-1}.)

%


Next, we localize the input in $p$ as
\begin{equation}\label{eq:53-1}
	f_1 = \sum_{p_1 > -\frac{(1-\delta)m}{2}} P^{\vp, p_1} f_1 + P^{\vp, \le -\frac{(1-\delta)m}{2}} f_1,
\end{equation} 
and
\begin{equation} \label{eq:53-2}
	f_2 = \sum_{p_2 > -\frac{(1-\delta)m}{2}} P^{\vp, p_2} f_2 + P^{\vp, \le -\frac{(1-\delta)m}{2}} f_2. 
\end{equation} 
For the last terms in \eqref{eq:53-1} and  \eqref{eq:53-2}, we think of $p_1 = -\frac{(1-\delta)m}{2}$ and $p_2 = -\frac{(1-\delta)m}{2}$. Without loss of generality, let us assume that $p_1 \le p_2$ (the opposite case is similar by exchanging the roles of $f_1$ and $f_2$). On the support of $\bar{\chi}_k^\hp(\xi) \chi_{k_1}^{\vp,p_1}(\xi-\eta) \chi_{k_2}^{\vp,p_2}(\eta)$, we have
\begin{equation}
	-p_2 + k \lesssim k_1 \sim k_2.
\end{equation} 
Consider the non-endpoint case $p_1> -\frac{(1-\delta)m}{2}$ first. Take two numbers
\begin{equation}
	l_1 = (1-\delta) m - k_1 + p_1,\quad l_2 = (1-\delta) m - k_2 + p_2.
\end{equation}
By Lemma \ref{lem:805-c1}, we have
\begin{equation}
	|\bar{\sigma}| \sim 2^{k_1 + k}
\end{equation}
and hence, by Lemma \ref{lem:53-1}, there holds that
\begin{equation} \label{eq:806-91}
	|S_\eta \Phi| + |\Omega_\eta \Phi| \sim 2^{p_1+k-k_1}.
\end{equation}
Using \eqref{eq:806-91} and \eqref{eq:429-001}, each integration by parts in $V_\eta \in \{S_\eta, \Omega_\eta\}$ gains  a factor of size
\begin{equation} \label{eq:614-01}
	s^{-1}(2^{-p_1-k+k_1}(1+ 2^{p_2-p_1}) + 2^{-2p_1} + 2^{-p_1+ l_1+k_1}) \lesssim s^{-1} 2^{-p_1-k+k_1}(1+ 2^{p_2-p_1}) +  s^{-\delta},
\end{equation}
if the input $f_1$ is replaced by $H_{\le l_1 }f_1$. Indeed, we have the following major terms arising from repeated integration by parts using \eqref{eq:429-001} (with $L = 2^{k_1+k}$):
\begin{enumerate}[(a)]
	\item There is always a gain of $s^{-1} 2^{-p_1-k+k_1}$ due to the denominator $s V_\eta \Phi$.
	\item $V_\eta$ landing on  ${\chi}^{\vp,p_1}(\xi-\eta) \F (H_{\le l_1} f_1)(\xi-\eta)$ gives a factor $\lesssim 1 + 2^{-p_1 -k_1 +k} +2^{l_1+k}$ using \eqref{eq:805-bb3} and $k_1 \sim k_2$.
	\item $V_\eta$ landing on $(\chi^{\vp, p_2} \wh{f}_2)(\eta)$ gives a bounded factor.
	\item If $V_\eta$ lands on $\chi_{_{V_\eta}}$ or $\frac{1}{V_\eta\Phi}$ or coefficients from previous steps of integration by parts, this produces a factor $\lesssim 1 + 2^{p_2-p_1}$  using the identities \eqref{eq:805-bb4}--\eqref{eq:805-bb7}. Note that in this case we have
	\begin{equation}
		|\Omega_\eta \bar\sigma| = |\xi_3 \eta_h| \lesssim L, \quad \frac{|V_\eta (\xi-\eta)_h|}{|(\xi-\eta)_h|} \lesssim 2^{p_2-p_1}.
	\end{equation}
\end{enumerate}
 Similarly, each integration by parts in $V_{\xi - \eta} \in \{S_{\xi-\eta}, \Omega_{\xi- \eta}\}$ gains  a factor of size
\begin{equation} \label{eq:614-02}
	s^{-1}(2^{-p_2-k+k_2} + 2^{-2p_2} + 2^{-p_2+ l_2+k_2})	\lesssim s^{-1} 2^{-p_2-k+k_2} + s^{-\delta},
\end{equation}
if $f_2$ is replaced by $H_{\le l_2 }f_2$. 
If $k_1-k \gg -p_2$ and $k_1-k-p_1 \gg (1-\delta)m$, then $p_1 \sim p_2$ and we have an acceptable bound via set size gain:
	\begin{align*}
 \|\bar{P}_k^\hp \Q[\m](P^{\vp, p_1} f_1, P^{\vp, p_2}  f_2)\|_{L^2}
	& \lesssim 2^k 2^{\frac32 k} \|P^{\vp, p_1} f_1\|_{L^2} \|P^{\vp, p_2}  f_2\|_{L^2} \\
	&\lesssim 2^{k} 2^{\frac32 k} 2^{p_1} 2^{p_2} 2^{-30k_1^+ - 30 k_2^+} 2^{C \varepsilon_1 m} \varepsilon_1^2 \\
	&\lesssim 2^{\frac52 k_1 - \frac94 (1-\delta) m + C \varepsilon_1 m} 2^{-30k_1^+ - 30 k_2^+} \varepsilon_1^2.
\end{align*}
In the last inequality, we used that
\begin{equation}
	\frac{5}{2} (k-k_1) \lesssim -\frac52 p_1 - \frac52(1-\delta) m \lesssim -2p_1 - \frac94 (1-\delta) m. 
\end{equation}
If  $k_1-k \lesssim -p_2$ or $k_1-k-p_1 \lesssim (1-\delta)m$, the factors \eqref{eq:614-01} and \eqref{eq:614-02} are both bounded by $s^{-\delta}$, hence repeated integration by parts in $V_\eta$ as well as $V_{\xi-\eta}$ yields acceptable contributions for the $H_{\le l_i} f_i, i=1,2$ inputs. Thus, it remains to treat   the $H_{> l_i} f_i$ inputs, \emph{i.e.}, to estimate
\begin{equation} \label{eq:806-011}
	\|\bar{P}_k^\hp \Q[\m](P^{\vp,p_1} H_{>l_1} f_1, P^{\vp,p_2}  H_{>l_2} f_2)\|_{L^2}.
\end{equation}

\medskip
\emph{Case 1.} $p_2 \le -10 \delta m$.
\smallskip

In this case, we use set size gain and $Y$ norm bounds to deduce that
\begin{align}
\eqref{eq:806-011}	&\lesssim 2^{k} 2^{\frac{k}{2} + k_1 + p_1} \|H_{>l_1} f_1\|_{L^2} \|H_{>l_2} f_2\|_{L^2} \nonumber \\
	&\lesssim 2^{\frac{5k_1}{2} + \frac{3p_2}{2} + p_1} 2^{-(p_1+(1-\delta)m)} 2^{-(1+\beta-\delta_0)(p_2+(1-\delta)m)} \|F_1\|_{Y^0_0} \|F_2\|_{Y^0_{\beta-\delta_0}} \nonumber \\
	&\lesssim 2^{\frac{5k_1}{2} + (\frac12 - \beta)p_2} 2^{-(2+\beta-\delta_0)(1-\delta)m} 2^{-30k_1^+-30k_2^+} t^{C \varepsilon_1} \varepsilon_1^2 \nonumber \\
	&\lesssim 2^{\frac{5k_1}{2}} 2^{-(2+\beta+3 \delta)m} 2^{-30k_1^+-30 k_2^+} \varepsilon_1^2.
\end{align}

\medskip
\emph{Case 2.} $p_2 > -10 \delta m$ and $p_1 \sim p_2$.
\smallskip

Here, we also use  set size gain and $Y$ norm bounds to deduce that
\begin{align}
	\eqref{eq:806-011}&\lesssim 2^{k} 2^{\frac{k}{2} + k_1 + p_1} \|H_{>l_1} f_1\|_{L^2} \|H_{>l_2} f_2\|_{L^2} \nonumber \\
	&\lesssim 2^{\frac{5k_1}{2} + \frac{3p_2}{2} + p_1} 2^{-(1+\beta-\delta_0)(p_1+(1-\delta)m)} 2^{-(1+\beta-\delta_0)(p_2+(1-\delta)m)} \|F_1\|_{Y^0_{\beta-\delta_0}} \|F_2\|_{Y^0_{\beta-\delta_0}} \nonumber \\
	&\lesssim 2^{\frac{5k_1}{2}} 2^{-2(1+\beta-\delta_0)(1-\delta)m + C \varepsilon_1 m} 2^{-30 k_1^+ - 30 k_2^+} \varepsilon_1^2,
\end{align}
which is acceptable since $2(1+\beta-\delta_0)(1-\delta) > 2+ \beta + \delta$.

\medskip
\emph{Case 3.} $p_2 > -10 \delta m$ and $p_1 \ll p_2$.
\smallskip

As observed in \cite{GuoInvent}, in this case we have
\begin{equation}
	\partial_{\eta_3} \Lmd(\eta) \sim 2^{2p_2 - k_2}, \quad -\partial_{\eta_3} \Lmd(\xi - \eta) \sim 2^{2p_1-k_1}
\end{equation}
and combined with $p_1 \ll p_2$, there holds that
\begin{equation} \label{eq:53-11}
	|\partial_{\eta_3} \Phi|  \gtrsim 2^{2 p_2 - k_2} \gtrsim 2^{-20 \delta m - k_2}.
\end{equation}
Make the decomposition of phase space
\begin{equation}
	1 \equiv \chi^{\res}(\xi, \eta) + \chi^{\nr}(\xi, \eta)
\end{equation}
with
\begin{equation}
	\chi^{\res} = \psi(2^{10\beta m} \Phi), \quad \chi^{\nr} = (1-\psi)(2^{10\beta m} \Phi).
\end{equation}
For the resonant part, the lower bound \eqref{eq:53-11} helps us to gain a crucial set size bounded by $2^{-8 \beta m + k_2}$ in the vertical direction ({cf.} \cite[Lemma A.4]{GuoInvent}), which leads to an acceptable contribution as
\begin{align}
	&\quad \ \|\bar{P}_k^\hp \Q[\m \chi^\res](P^{\vp, p_1} H_{>l_1} f_1, P^{\vp, p_2} H_{>l_2} f_2)\|_{L^2} \nonumber \\
	&\lesssim 2^{k} 2^{k_1+p_1+\frac{k_2}{2} - 4 \beta m} \|H_{>l_1} f_1\|_{L^2} \|H_{>l_2} f_2\|_{L^2} \nonumber \\
	&\lesssim 2^{\frac52 k_1 + p_1 + p_2 - 4 \beta m} 2^{-(l_1+k_1)-(l_2+k_2)} \|F_1\|_{Y^0_0} \|F_2\|_{Y^0_{0}} \nonumber \\
	&\lesssim 2^{\frac52 k_1 -4 \beta m- 2(1-\delta)m + C \varepsilon_1 m} 2^{-30 k_1^+ -30k_2^+}  \varepsilon_1^2.
\end{align}
Then, we consider the non-resonant part. Performing a normal form, we receive acceptable contributions due to the estimates
\begin{align} \label{eq:531047}
	&\quad \ \|\bar{P}_k^\hp \Q[\m \chi^\nr \Phi^{-1}](P^{\vp, p_1} H_{>l_1} f_1, P^{\vp, p_2} H_{>l_2} f_2)\|_{L^2} \nonumber \\
	&\lesssim 2^{k+10\beta m} 2^{k_1+p_1+\frac{k}{2} } \|H_{>l_1} f_1\|_{L^2} \|H_{>l_2} f_2\|_{L^2} \nonumber \\
	&\lesssim 2^{\frac12 k + 2 k_1 +10\beta m} 2^{-2(1-\delta)m} \|F_1\|_{Y^0_0} \|F_2\|_{Y^0_0} \nonumber \\
	&\lesssim 2^{\frac12 k + 2 k_1 - 30 k_1^+ - 30 k_2^+} 2^{-(2-11\beta)m} \varepsilon_1^2,
\end{align}
and
\begin{align}
	&\quad \ \|\bar{P}_k^\hp \Q[\m \chi^\nr \Phi^{-1}](P^{\vp, p_1} H_{>l_1} \partial_t f_1, P^{\vp, p_2} H_{>l_2} f_2)\|_{L^2} \nonumber \\
	&\lesssim 2^{k+10\beta m} 2^{k_1+p_1+\frac{k}{2} } \|\partial_t f_1\|_{L^2} \|H_{>l_2} f_2\|_{L^2} \nonumber \\
	&\lesssim 2^{\frac32 k + k_1 - 30 k_1^+ +10\beta m} 2^{(-\frac32+\delta_0)m-(1-\delta)m} \varepsilon_1^2  \|F_2\|_{Y^0_0} \nonumber \\
	&\lesssim 2^{\frac32 k + k_1 - 30 k_1^+ - 30 k_2^+} 2^{-(\frac52-11\beta)m} \varepsilon_1^2
\end{align}
(the case $\partial_t$ landing on $f_2$ is similar), and in the case of $m+5 \le l+k <(1+\delta)m$,
\begin{align} \label{eq:527-a1}
	&\quad \ \langle \bar{P}_k^\hp \Q[\m \chi^\nr \Phi^{-1}](P^{\vp, p_1} H_{>l_1} f_1, P^{\vp, p_2} H_{>l_2} f_2), Z_l^2 \partial_t f \rangle_{L^2} \nonumber \\
	&\stackrel{\eqref{eq:531047}}{\lesssim} 2^{\frac12 k + 2k_1 - 30 k_1^+ - 30 k_2^+} 2^{-(2-11 \beta)m} \varepsilon_1^2 \cdot 2^{-2m} \varepsilon_1^2 \nonumber \\
	&\ \lesssim 2^{\frac12 k + 2k_1 - 30 k_1^+ - 30 k_2^+} 2^{-2(1+\beta)(l+k) -\frac32 m} \varepsilon_1^4.
\end{align}

\medskip

Finally, we discuss the endpoint cases for $p_1$ and $p_2$. If $p_1 = -\frac{(1-\delta) m}{2}$, we avoid using the $H$-decomposition for $f_1$ and integration by parts in $V_\eta$. Following the approach in Case 1 and Case 3, instead of using $Y$ norm bounds for $\|H_{>l_1} f_1\|_{L^2}$, we simply use the crude bound 
\begin{equation} \label{eq:512-841}
	2^{p_1} \|P^{\vp, \le p_1} f_1\|_{L^2} \lesssim 2^{2p_1} \|F_1\|_{Y^{0,3}_{\beta'}} \lesssim 2^{-(1-\delta)m} \|F_1\|_{Y^{0,3}_{\beta'}}.
\end{equation}
Similarly, if $p_2 = p_1 = -\frac{(1-\delta) m}{2}$, we also avoid using the $H$-decomposition for $f_2$ and integration by parts in $V_{\xi-\eta}$. Following the approach in Case 1, instead of using $Y$ norm bounds for $\|H_{>l_2} f_2\|_{L^2}$, we use the crude bound 
\begin{equation}
	2^{p_2} \|P^{\vp, \le p_2} f_2\|_{L^2}\lesssim 2^{2p_2} \|F_2\|_{Y^{0,3}_{\beta'}} \lesssim 2^{-(1-\delta)m} \|F_2\|_{Y^{0,3}_{\beta'}}.
\end{equation}
It is straightforward to check that the resulting estimates are still acceptable.

%
%
%
%
%
%
%
%

\subsection{$\bigvp + \bighp$ interactions} \label{sec:84}

Here, we give the bilinear $L^2$ estimate on $\int_{\mathcal{I}_m} P_k\Q[\m](f_1, f_2) ds$ with $f_1 = P_{k_1}^\vp F_1$ and $f_2 = P_{k_2}^\hp F_2$. (The $\bighp$+$\bigvp$ case is similar by exchanging the roles of $f_1$ and $f_2$.)

\medskip

Decompose the second input into
\begin{equation} \label{eq:53-a1}
	f_2 = (P^{\hp} - \bar{P}^{\hp}) f_2 +  \bar{P}^{\hp} f_2.
\end{equation}
For the first part in \eqref{eq:53-a1}, since $P^\hp - \bar{P}^\hp = \wt{P}^\vp P^\hp$, we can apply the methods for $\bigvp + \bigvp \to \bigvp$ and $\bigvp + \bigvp \to \bighp$ cases directly. Indeed, the arguments in Section \ref{sec:9-2} and Section \ref{sec:93} still work if $P^\vp$ is replaced by $\wt{P}^\vp$. Note that  on the support of $\chi_k(\xi) \chi_{k_1}^{\vp}(\xi-\eta) \bar{\chi}_{k_2}^{\hp}(\eta)$, there holds that	$k \sim k_{\max}$.


\smallskip
Decompose the first input into
\begin{equation}
		f_1 = \sum_{p_1 > -\frac{(1-\delta)m}{2}} P^{\vp, p_1} f_1 + P^{\vp, \le -\frac{(1-\delta)m}{2}} f_1.
\end{equation}
We first consider the non-endpoint case $2p_1 > -(1-\delta)m$, while the endpoint case will be treated at the end of the proof. Take two numbers
\begin{equation}
	l_1 = (1-\delta)m - k_1 +p_1, \quad l_2 = (1-\delta)m - k_2.
\end{equation}
As in previous sections, we use integration by parts in $V_\eta, V_{\xi-\eta}$ to treat the cases when $f_1$ is replaced by $H_{\le l_1} f_1$ or when $f_2$ is replaced by $Z_{\le l_2} f_2$. By Lemma \ref{lem:805-c1}, on the support of $\chi^{\vp, p_1}_{k_1}(\xi - \eta) \bar{\chi}^{\hp}_{k_2}(\eta)$, there holds that
\begin{equation}
  |\bar{\sigma}| \sim 2^{k_1+k_2},
\end{equation}
and
\begin{equation} \label{eq:806-f1}
	|S_\eta \Phi| + |\Omega_\eta \Phi| \sim 2^{k_2-k_1+p_1}.
\end{equation}
By \eqref{eq:806-f1} and \eqref{eq:429-001}, each integration by parts in $V_\eta \in \{S_\eta, \Omega_\eta\}$ gains a factor  of size
\begin{equation}
	s^{-1}(2^{-p_1+k_1-k_2} + 2^{-2p_1} + 2^{-p_1+l_1+k_1}) \lesssim s^{-1} 2^{-p_1+k_1-k_2} + s^{-\delta}.
\end{equation}
Indeed, we have the following major terms arising from repeated integration by parts using \eqref{eq:429-001} (with $L = 2^{k_1+k_2}$):
\begin{enumerate}[(a)]
	\item There is always a gain of $s^{-1} 2^{-p_1+k_1-k_2}$ due to the denominator $s V_\eta \Phi$.
	\item $V_\eta$ landing on  ${\chi}^{\vp,p_1}(\xi-\eta) \F (H_{\le l_1} f_1)(\xi-\eta)$ gives a factor $\lesssim 1 + 2^{k_2-k_1}+2^{-p_1+k_2-k_1} +2^{l_1+k_2}$ using \eqref{eq:805-bb3}.
	\item $V_\eta$ landing on the second input gives a bounded factor.
	\item If $V_\eta$ lands on $\chi_{_{V_\eta}}$ or $\frac{1}{V_\eta\Phi}$ or coefficients from previous steps of integration by parts, this produces a factor $\lesssim 1 + 2^{k_2-k_1-p_1}$  using the identities \eqref{eq:805-bb4}--\eqref{eq:805-bb7}. Note that in this case we have
	\begin{equation}
		|\Omega_\eta \bar\sigma| = |\xi_3 \eta_h| \lesssim L(1+2^{k_2-k_1}), \quad \frac{|V_\eta (\xi-\eta)_h|}{|(\xi-\eta)_h|} \lesssim 2^{k_2-k_1-p_1}.
	\end{equation}
\end{enumerate}
On the other hand (using \eqref{eq:805-bb1}--\eqref{eq:805-bb2} instead of \eqref{eq:805-bb3}), each integration by parts in $V_{\xi-\eta} \in \{S_{\xi- \eta}, \Omega_{\xi-\eta}\}$ gains a factor  of size
\begin{equation}
	s^{-1}(2^{k_2-k_1} (1+ 2^{k_1-k_2})+2^{l_2+k_2}) \lesssim s^{-1}2^{k_2-k_1} + s^{-\delta}.
\end{equation}

\medskip
\emph{Case 1.} $-p_1 + k_1 - k_2 \ge (1-\delta)m$.
\smallskip

In this case, $k_2-k_1 \le -(1-\delta)m-p_1 \le -\frac{(1-\delta)m}{2}$, hence repeated integration by parts in $V_{\xi - \eta}$ is beneficial. Then, we use set size gain to deduce that
\begin{align}
	& \quad \ \| P_k \Q[\m](P^{\vp, p_1}  f_1, \bar{P}^\hp Z_{> l_2} f_2)\|_{L^2} \nonumber \\
	&\lesssim 2^{k} \min\{2^{k_1+p_1+\frac12 k_2}, 2^{\frac32 k_2}\} \|P^{\vp, p_1} f_1\|_{L^2} \|Z_{>l_2} f_2\|_{L^2} \nonumber \\
	&\lesssim  2^{k+  \frac14 k_1 + \frac54 k_2 + \frac14 p_1} 2^{p_1-(1+\beta-\delta_0)(l_2+k_2)} \|F_1\|_{Y^{0,3}_{\beta'}} \|F_2\|_{X^0_{\beta-\delta_0}} \nonumber \\
	&\lesssim 2^{k+ \frac32 k_1 - 30 k_1^+ - 30 k_2^+} 2^{-(\frac94+\beta-\delta_0)(1-\delta)m} 2^{C \varepsilon_1 m} \varepsilon_1^2.
\end{align}

\medskip
\emph{Case 2.} $k_2 - k_1 \ge (1-\delta)m$.
\smallskip

In this case, $-p_1 + k_1-k_2 \le -\frac{(1-\delta)m}{2}$, hence repeated integration by parts in $V_\eta$ is beneficial. Then, we use set size gain to deduce that
\begin{align}
	&\quad \ \|P_k \Q[\m]( P^{\vp, p_1} H_{> l_1} f_1, \bar{P}^\hp  f_2)\|_{L^2} \nonumber \\
	&\lesssim 2^{k}  2^{\frac32 k_1 + p_1} \|H_{>l_1} f_1\|_{L^2} \|\bar{P}^\hp f_2\|_{L^2} \nonumber \\
	&\lesssim  2^{k+  \frac32 k_2 - \frac32 (1-\delta)m  + p_1} 2^{-(l_1+k_1)} \|F_1\|_{Y^{0}_{0}} \|F_2\|_{X^0_0} \nonumber \\
	&\lesssim 2^{k + \frac32 k_2 - 30 k_1^+ - 30 k_2^+ } 2^{-\frac52 (1-\delta)m + C \varepsilon_1 m} \varepsilon_1^2.
\end{align}

\medskip
\emph{Case 3.} $k_2 - k_1 < (1-\delta)m$ and $-p_1 + k_1 - k_2 < (1-\delta)m$.
\smallskip

In this case, repeated integration by parts in $V_\eta$ and in $V_{\xi-\eta}$ is beneficial. Hence, it suffices to consider
\begin{equation} \label{eq:806-yy1}
	\|P_k\Q[\m](P^{\vp, p_1} H_{> l_1} f_1, \bar{P}^\hp Z_{> l_2} f_2)\|_{L^2}.
\end{equation}

\medskip
\emph{Subcase 3.1.} $2p_1 -k_1 \gtrsim -k_2$.
\smallskip

We directly apply set size gain to achieve an acceptable contribution as  
\begin{align}
\eqref{eq:806-yy1}&\lesssim 2^{k} 2^{k_1+p_1+\frac12 k_2} \|H_{>l_1} f_1\|_{L^2} \|Z_{>l_2} f_2\|_{L^2} \nonumber \\
	&\lesssim  2^{k+  k_2 + \frac12 k_1 +2p_1} 2^{-(1+\beta-\delta_0)(l_1+k_1)-(1+\beta-\delta_0)(l_2+k_2)} \|F_1\|_{Y^0_{\beta-\delta_0}} \|F_2\|_{X^0_{\beta-\delta_0}} \nonumber \\
	&\lesssim 2^{k+ k_2+ \frac12 k_1 - 30 k_1^+ -30k_2^+} 2^{-2(1+\beta-\delta_0)(1-\delta)m + C \varepsilon_1 m} \varepsilon_1^2.
\end{align}

\medskip
\emph{Subcase 3.2.} $2p_1 -k_1 \ll -k_2$.
\smallskip

In this case, there holds that
\begin{equation}
	|\partial_{\eta_3} \Phi| \gtrsim 2^{-k_2}.
\end{equation}
Make the decomposition of phase space
\begin{equation}
	1 \equiv \chi^{\res}(\xi, \eta) + \chi^{\nr}(\xi, \eta)
\end{equation}
with
\begin{equation}
	\chi^{\res} = \psi(2^{100\delta m} \Phi), \quad \chi^{\nr} = (1-\psi)(2^{100\delta m} \Phi).
\end{equation}
For the resonant part, we use \cite[Lemma A.4]{GuoInvent} to get that
\begin{align}
		&\quad \ \|P_k\Q[\m \chi^{\res}](P^{\vp, p_1} H_{> l_1} f_1, \bar{P}^\hp Z_{> l_2} f_2)\|_{L^2} \nonumber \\
		&\lesssim 2^k 2^{k_1+p_1+\frac12 k_2 -50 \delta m} \|H_{>l_1} f_1\|_{L^2} \|Z_{>l_2} f_2\|_{L^2} \nonumber \\
		&\lesssim  2^{k+  k_1 + \frac12 k_2 +p_1 - 50 \delta m} 2^{-(l_1+k_1)-(1+\beta-\delta_0)(l_2+k_2)} \|F_1\|_{Y^0_0} \|F_2\|_{X^0_{\beta-\delta_0}} \nonumber \\
		&\lesssim 2^{k+ k_1+ \frac12 k_2-30 k_1^+ -30 k_2^+} 2^{-(2+\beta-\delta_0)(1-\delta)m - (50 \delta-C \varepsilon_1) m} \varepsilon_1^2.
\end{align}
For the non-resonant part, we perform a normal form and achieve an acceptable contribution due to the estimates
\begin{align}
	&\quad \ \|P_k\Q[\m \chi^{\nr} \Phi^{-1}](P^{\vp, p_1} H_{> l_1} f_1, \bar{P}^\hp Z_{> l_2} f_2)\|_{L^2} \nonumber \\
	&\lesssim 2^{k +100 \delta m} 2^{k_1 + p_1 + \frac12 k_2} \|H_{>l_1} f_1\|_{L^2} \|Z_{>l_2} f_2\|_{L^2} \nonumber \\
	&\lesssim 2^{k +100 \delta m} 2^{k_1 + p_1 + \frac12 k_2} 2^{-(l_1+k_1)-(l_2+k_2)} \|F_1\|_{Y^0_0} \|F_2\|_{X^0_0} \nonumber \\
	&\lesssim  2^{k+k_1  + \frac12 k_2 - 30 k_1^+ - 30 k_2^+} 2^{-2(1-\delta)m + (100 \delta +C\varepsilon_1) m} \varepsilon_1^2,
\end{align}
and
\begin{align} \label{eq:512-ff}
	&\quad \ \|P_k\Q[\m \chi^{\nr} \Phi^{-1}](P^{\vp, p_1} H_{> l_1} \partial_t f_1, \bar{P}^\hp Z_{> l_2} f_2)\|_{L^2} \nonumber \\
	&\lesssim 2^{k +100 \delta m} 2^{k_1 + p_1 + \frac12 k_2} \|\partial_t f_1\|_{L^2} \|Z_{>l_2} f_2\|_{L^2} \nonumber \\
	&\lesssim 2^{k +100 \delta m} 2^{k_1 + p_1 + \frac12 k_2} 2^{-\frac32m + \delta_0 m-(l_2+k_2)} \varepsilon_1^2 \, \|F_2\|_{X^0_0} \nonumber \\
	&\lesssim  2^{k+k_1  + \frac12 k_2 - 30 k_1^+ - 30 k_2^+} 2^{-\frac52 m + 120 \delta m} \varepsilon_1^3.
\end{align}
The case of $\partial_t$ landing on $f_2$ is similar to \eqref{eq:512-ff}, and in the case of energy approach the additional term can be treated similarly as in \eqref{eq:527-a1} using Lemma \ref{lem:51-1}.

Finally, for the endpoint case $p_1 = - \frac{(1-\delta)m}{2}$, we do not make the $H$-decomposition for $f_1$. Using the crude bound \eqref{eq:512-841}  instead of $Y$ norm bounds, the resulting estimate is still acceptable.

\section*{Acknowledgement} 

We would like to thank professors Zhen Lei, Yi Zhou, Dongyi Wei and Yuan Cai for some useful discussions and professor Taoufik Hmidi for a helpful suggestion. X. Ren is supported by the China Postdoctoral Science Foundation under Grant Number BX20230019 and the National Key R\&D Program of China (No.2023YFA1010700). G. Tian is supported in part
by NSFC No.11890660 \& No.12341105, MOST No.2020YFA0712800.

\appendix
\section{Appendix}

\subsection{Commutator estimates}

On the support of $\chi_k$ (see Section \ref{sec:515-22}), the group velocity has an upper bound given by
\begin{equation}
	|\nabla \Lmd| = \frac{|\xi_h|}{|\xi|^2} \le \frac{1}{|\xi|} \le \frac{10}{9} \times 2^{-k}.
\end{equation}
The following result can be seen as  a property of finite propagation speed  for the linear dispersion $e^{i t \Lmd}$.
\begin{lemma} \label{lem:a1-1}
	Let $D$ be a Fourier multiplier operator corresponding to a homogeneous symbol $\lambda(\xi) = \bar{\lambda}(\xi/|\xi|)$ with $\bar{\lambda} \in C^\infty(\mathbb{S}^2)$. Suppose that $t \in [2^m, 2^{m+1}) (m \ge 1)$, $l+k \ge m +3$, $\mu \in \{+,-\}$, then for any integer $K \ge 1$ there holds that
	\begin{equation} \label{eq:a1-est-1}
		\left\|Z_l \left(e^{\mu i t\Lambda} D P_k \right) Z_{[l-1, l+1]^c}\right\|_{L^2 \to L^2} + \left\|H_l \left(e^{\mu i t\Lambda} D P_k \right) H_{[l-1, l+1]^c}\right\|_{L^2 \to L^2} \lesssim_K  2^{-K(l+k)}.
	\end{equation}
\end{lemma}

\begin{proof}
We only consider the estimate involving $Z_l$ as the estimate involving $H_l$ can be shown in a similar way. Write for short $A= e^{\mu i t \Lambda} D P_k$ and $a = e^{\mu i \Lambda t} \lambda \chi_k $, and consider
\begin{align} \label{eq:F3-1-h}
	(\F_3^{-1}a)(\xi_h, x_3)	&= \frac{1}{2 \pi}\int e^{i(\mu t \Lambda +  x_3 \xi_3)} \lambda \chi_k d\xi_3 \nonumber \\
	&= \frac{-i}{2 \pi}\int \frac{\partial_{\xi_3}e^{i(\mu t \Lambda +  x_3 \xi_3)}}{\partial_{\xi_3} (\mu t\Lambda +x_3 \xi_3)} \lambda \chi_k d\xi_3 \nonumber \\
	&= \frac{-i}{2 \pi}\int e^{i(\mu t \Lambda +  x_3 \xi_3)} \left\{ \frac{\partial^2_{\xi_3} (\mu t\Lambda +x_3 \xi_3)}{[\partial_{\xi_3} (\mu t\Lambda +x_3 \xi_3)]^2}  \lambda \chi_k - \frac{\partial_{\xi_3} (\lambda \chi_k)}{\partial_{\xi_3} (\mu t\Lambda +x_3 \xi_3)}  \right\} d\xi_3.
\end{align}
Suppose that $|x_3| \ge t 2^{-k+1}$, then on the support of $\chi_k$  we have
\begin{align} \label{eq:730-201}
	|\partial_{\xi_3} (\mu t\Lambda +x_3 \xi_3)| \ge \frac{1}{2} |x_3|,
\end{align}
hence from \eqref{eq:F3-1-h}  we deduce that
\begin{align*}
	|(\F_3^{-1}a)(\xi_h, x_3)| \lesssim (t 2^{-k}|x_3|^{-1} + 1)|x_3|^{-1} \lesssim |x_3|^{-1}.
\end{align*}
Further repeated integration by parts based on \eqref{eq:F3-1-h} and \eqref{eq:730-201} gives
 \begin{align} \label{eq:F3-1-h-pw}
 	|(\F_3^{-1}a)(\xi_h, x_3)| \lesssim_K 2^{-(K-1)k} |x_3|^{-K}
 \end{align}
 for any $K \ge 1$. 
 
 Now, \eqref{eq:a1-est-1} follows from \eqref{eq:F3-1-h-pw}, Young's inequality for convolutions and the identity
 \begin{equation}
  \left[\F_h \left( A Z_{[l-1,l+1]^c} f \right) \right] (\xi_h, x_3) = \int_{\R} (\F_3^{-1} a)(\xi_h, x_3-y_3) (Z_{[l-1,l+1]^c} \F_h f)(\xi_h, y_3) dy_3.
 \end{equation}
\end{proof}

\begin{lemma} Under the assumptions of Lemma \ref{lem:a1-1}, there holds that
	\begin{equation} \label{eq:comm-2}
		\|[Z_l, e^{\mu i t \Lambda} D P_k^\hp]\|_{L^2 \to L^2} \lesssim  2^{-l-k+m}.
	\end{equation}
The constant in \eqref{eq:comm-2} depends only on $\bar{\lambda}$.
\end{lemma}

\begin{proof}
Using notations from the previous proof, we have
\begin{align}
	\{\F_h [Z_l, A] f\}(\xi_h,x_3) &= \int \big(\varphi_l(x_3)-\varphi_l(y_3)\big)(\F_3^{-1} a)(\xi_h, x_3-y_3) (\F_h f)(\xi_h, y_3) dy_3 \nonumber \\
	&= i \int \frac{ \varphi_l(x_3)-\varphi_l(y_3)}{x_3 - y_3} (\F_3^{-1} \partial_3 a)(\xi_h, x_3-y_3) (\F_h f)(\xi_h, y_3) dy_3.
\end{align}
Then, \eqref{eq:comm-2} is a consequence of the fact that
\begin{equation}
	\left\|\F_{x_3,y_3} \left(\frac{ \varphi_l(x_3)-\varphi_l(y_3)}{x_3 - y_3}\right)\right\|_{L^1} \lesssim 2^{-l}
\end{equation}
and the trilinear inequality
\begin{equation}
	\left\|\int g_1(x, y) g_2(x-y) g_3(y) dy \right\|_{L^2(dx)} \lesssim \|\F_{x, y} g_1\|_{L^1} \|\F g_2\|_{L^\infty} \|g_3\|_{L^2}.
\end{equation}
\end{proof}

\subsection{Interpolation estimates}

\begin{lemma}
	\label{lem:16-a}There holds that
	\begin{equation}
		\| S^a f \|_{L^{\frac{2 n}{a}}} \lesssim \| f \|_{L^{\infty}}^{1 -
			\frac{a}{n}} \left( \sum_{b = 0}^n \| S^b f \|_{L^2}
		\right)^{\frac{a}{n}}, \label{eq:16-a}
	\end{equation}
	\begin{equation}
		\| \Omega^a f \|_{L^{\frac{2 n}{a}}} \lesssim \| f \|_{L^{\infty}}^{1 -
			\frac{a}{n}} \| \Omega^n f \|_{L^2}^{\frac{a}{n}} . \label{eq:16-b}
	\end{equation}
\end{lemma}

\begin{proof}
	Inequality \eqref{eq:16-a} is proved in {\cite[Lemma 5.3]{GuoCpam}}, and
	\eqref{eq:16-b} can be proved using the same argument. For readers'
	convenience, we present the details for \eqref{eq:16-b} here. We define, for
	$0 \leqslant a \leqslant n$,
	\[ A (a, n) : = \| \Omega^a f \|_{L^{\frac{2 n}{a}}} . \]
	Using integration by parts and Holder's inequality, for $1 \leqslant a
	\leqslant n - 1$, we have
	\[ A (a, n)^{\frac{2 n}{a}} = \int \Omega^a f \cdot \Omega^a f | \Omega^a f
	|^{2 \left( \frac{n}{a} - 1 \right)} d x = - \int \Omega^{a - 1} f \cdot
	\Omega \left( \Omega^a f | \Omega^a f |^{2 \left( \frac{n}{a} - 1
		\right)} \right) d x \]
	\[ = - \left( \frac{2 n}{a} - 1 \right) \int \Omega^{a - 1} f \cdot
	\Omega^{a + 1} f | \Omega^a f |^{2 \left( \frac{n}{a} - 1 \right)} d x
	\lesssim A  (a - 1, n) A (a + 1, n) A (a, n)^{2 \left( \frac{n}{a} - 1
		\right)} \]
	Consequently, we get
	\[ A (a, n) \lesssim \sqrt{A (a - 1, n) A (a + 1, n)} . \]
	This shows that for some constant $C$, the sequence
	\[ b_a = C a^2 + \ln A (a, n) \]
	is discretely convex, that is,
	\[ b_{a - 1} - 2 b_a + b_{a + 1} \geqslant 0, \quad a \in [1, n - 1] . \]
	Hence, for $a \in [0, n]$ we have
	\[ b_a \leqslant \left( 1 - \frac{a}{n} \right) b_0 + \frac{a}{n} b_n \]
	which implies \eqref{eq:16-b} directly.
	
	\ 
\end{proof}

\begin{lemma} \label{lem:730-a}
	There holds that
	\begin{equation}
		\| S^a f \|_{H^{n - a}} \lesssim \| f \|_{H^n}^{1 - \frac{a}{n}} \left(
		\sum_{b = 1}^n \| S^b f \|_{L^2} \right)^{\frac{a}{n}} \label{eq:16-2b}
	\end{equation}
	\begin{equation}
		\| \Omega^a f \|_{H^{n - a}} \lesssim \| f \|_{H^n}^{1 - \frac{a}{n}}
		\left( \sum_{b = 1}^n \| \Omega^b f \|_{L^2} \right)^{\frac{a}{n}}
		\label{eq:16-3b}
	\end{equation}
	\begin{equation}
		\| S^a \Omega^{n - a} f \|_{L^2} \lesssim \| \Omega^{a + b} f \|_{L^2}^{1
			- \frac{a}{n}} \left( \sum_{k = 0}^n \| S^k f \|_{L^2}
		\right)^{\frac{a}{n}} \label{eq:26aa}
	\end{equation}
\end{lemma}

\begin{proof}
	We define, for $0 \leqslant a \leqslant n$,
	\[ A (a, n) = \sum_{b = 0}^a \| S^b f \|_{H^{n - a}} . \]
	Using integration by parts and Holder's inequality, we have, for $| \mu | =
	n - a$,
	\[ \| \partial^{\mu} S^a f \|_{L^2}^2 = - \int \partial^{\mu - 1} S^{a + 1}
	f \cdot \partial^{\mu + 1} S^{a - 1} f d x + l.o.t. \lesssim A (a + 1, n)
	A (a - 1, n), \]
	which show that
	\begin{equation}
		A (a, n) \leqslant \sqrt{A (a - 1, n) A (a + 1, n)} . \label{eq:16-2a}
	\end{equation}
	Hence, the argument from Lemma \ref{lem:16-a} can be applied to prove
	\eqref{eq:16-2b}.
	
	The proofs for \eqref{eq:16-3b} and \eqref{eq:26a} are similar.
\end{proof}




		\bibliographystyle{plain}
	\bibliography{Euler-Coriolis}

\begin{thebibliography}{10}

\bibitem{Angulo}
Vladimir Angulo-Castillo and Lucas C.~F. Ferreira.
\newblock On the 3{D} {E}uler equations with {C}oriolis force in borderline
  {B}esov spaces.
\newblock {\em Commun. Math. Sci.}, 16(1):145--164, 2018.

\bibitem{BBZD}
Jacob Bedrossian, Roberta Bianchini, Michele Coti~Zelati, and Michele Dolce.
\newblock Nonlinear inviscid damping and shear-buoyancy instability in the
  two-dimensional {B}oussinesq equations.
\newblock {\em Comm. Pure Appl. Math.}, 76(12):3685--3768, 2023.

\bibitem{JZV}
Jacob Bedrossian, Michele Coti~Zelati, and Vlad Vicol.
\newblock Vortex axisymmetrization, inviscid damping, and vorticity depletion
  in the linearized 2{D} {E}uler equations.
\newblock {\em Ann. PDE}, 5(1):Paper No. 4, 192, 2019.

\bibitem{Mas}
Jacob Bedrossian and Nader Masmoudi.
\newblock Inviscid damping and the asymptotic stability of planar shear flows
  in the 2{D} {E}uler equations.
\newblock {\em Publ. Math. Inst. Hautes \'{E}tudes Sci.}, 122:195--300, 2015.

\bibitem{BFMT}
Roberta {Bianchini}, Luca {Franzoi}, Riccardo {Montalto}, and Shulamit
  {Terracina}.
\newblock {Large amplitude quasi-periodic traveling waves in two dimensional
  forced rotating fluids}.
\newblock {\em arXiv e-prints}, page arXiv:2406.07099, June 2024.

\bibitem{Cai1}
Yuan Cai and Zhen Lei.
\newblock Global well-posedness of the incompressible magnetohydrodynamics.
\newblock {\em Arch. Ration. Mech. Anal.}, 228(3):969--993, 2018.

\bibitem{Cai2}
Yuan Cai, Zhen Lei, Fanghua Lin, and Nader Masmoudi.
\newblock Vanishing viscosity limit for incompressible viscoelasticity in two
  dimensions.
\newblock {\em Comm. Pure Appl. Math.}, 72(10):2063--2120, 2019.

\bibitem{Chemin1}
J.-Y. Chemin, B.~Desjardins, I.~Gallagher, and E.~Grenier.
\newblock Anisotropy and dispersion in rotating fluids.
\newblock In {\em Nonlinear partial differential equations and their
  applications. {C}oll\`ege de {F}rance {S}eminar, {V}ol. {XIV} ({P}aris,
  1997/1998)}, volume~31 of {\em Stud. Math. Appl.}, pages 171--192.
  North-Holland, Amsterdam, 2002.

\bibitem{Chen}
Jiajie Chen and Thomas~Y. Hou.
\newblock Stable nearly self-similar blowup of the 2d boussinesq and 3d euler
  equations with smooth data i: Analysis, 2023.

\bibitem{Choi}
Kyudong Choi.
\newblock Stability of {H}ill's spherical vortex.
\newblock {\em Comm. Pure Appl. Math.}, 77(1):52--138, 2024.

\bibitem{Dut}
Alexandre Dutrifoy.
\newblock Examples of dispersive effects in non-viscous rotating fluids.
\newblock {\em J. Math. Pures Appl. (9)}, 84(3):331--356, 2005.

\bibitem{Elgindi}
Tarek Elgindi.
\newblock Finite-time singularity formation for {$C^{1,\alpha}$} solutions to
  the incompressible {E}uler equations on {$\Bbb R^3$}.
\newblock {\em Ann. of Math. (2)}, 194(3):647--727, 2021.

\bibitem{EWsiam}
Tarek~M. Elgindi and Klaus Widmayer.
\newblock Sharp decay estimates for an anisotropic linear semigroup and
  applications to the surface quasi-geostrophic and inviscid {B}oussinesq
  systems.
\newblock {\em SIAM J. Math. Anal.}, 47(6):4672--4684, 2015.

\bibitem{EW}
Tarek~M. Elgindi and Klaus Widmayer.
\newblock Long time stability for solutions of a beta-plane equation.
\newblock {\em Communications on Pure and Applied Mathematics},
  70(8):1425--1471, 2017.

\bibitem{Galla}
Isabelle Gallagher and Laure Saint-Raymond.
\newblock Chapter 5 - on the influence of the earth's rotation on geophysical
  flows.
\newblock In S.~Friedlander and D.~Serre, editors, {\em Handbook of
  Mathematical Fluid Dynamics}, volume~4 of {\em Handbook of Mathematical Fluid
  Dynamics}, pages 201--329. North-Holland, 2007.

\bibitem{GMS2}
P.~Germain, N.~Masmoudi, and J.~Shatah.
\newblock Global solutions for the gravity water waves equation in dimension 3.
\newblock {\em Ann. of Math. (2)}, 175(2):691--754, 2012.

\bibitem{Germain}
Pierre Germain.
\newblock Space-time resonances.
\newblock {\em Journ\'ees \'equations aux d\'eriv\'ees partielles}, pages
  1--10, 2010.

\bibitem{GMS}
Pierre Germain, Nader Masmoudi, and Jalal Shatah.
\newblock Global solutions for 3d quadratic schr\"odinger equations.
\newblock {\em International Mathematics Research Notices}, 2009(3):414--432,
  12 2008.

\bibitem{Grafa}
Loukas Grafakos.
\newblock {\em Classical {F}ourier analysis}, volume 249 of {\em Graduate Texts
  in Mathematics}.
\newblock Springer, New York, third edition, 2014.

\bibitem{Gren}
E.~Grenier.
\newblock Oscillatory perturbations of the {N}avier-{S}tokes equations.
\newblock {\em J. Math. Pures Appl. (9)}, 76(6):477--498, 1997.

\bibitem{GuoCpam}
Yan Guo, Chunyan Huang, Benoit Pausader, and Klaus Widmayer.
\newblock On the stabilizing effect of rotation in the 3d {E}uler equations.
\newblock {\em Comm. Pure Appl. Math.}, 76(12):3553--3641, 2023.

\bibitem{GuoInvent}
Yan Guo, Benoit Pausader, and Klaus Widmayer.
\newblock Global axisymmetric {E}uler flows with rotation.
\newblock {\em Invent. Math.}, 231(1):169--262, 2023.

\bibitem{Jiacmp}
Alexandru~D. Ionescu and Hao Jia.
\newblock Inviscid damping near the {C}ouette flow in a channel.
\newblock {\em Comm. Math. Phys.}, 374(3):2015--2096, 2020.

\bibitem{Jia1}
Alexandru~D. Ionescu and Hao Jia.
\newblock Axi-symmetrization near point vortex solutions for the 2{D} {E}uler
  equation.
\newblock {\em Comm. Pure Appl. Math.}, 75(4):818--891, 2022.

\bibitem{Jia2}
Alexandru~D. Ionescu and Hao Jia.
\newblock Non-linear inviscid damping near monotonic shear flows.
\newblock {\em Acta Math.}, 230(2):321--399, 2023.

\bibitem{Kiselev}
Alexander Kiselev and Vladimir \v{S}ver\'{a}k.
\newblock Small scale creation for solutions of the incompressible
  two-dimensional {E}uler equation.
\newblock {\em Ann. of Math. (2)}, 180(3):1205--1220, 2014.

\bibitem{Klainerman}
Sergiu Klainerman.
\newblock Global existence of small amplitude solutions to nonlinear
  {K}lein-{G}ordon equations in four space-time dimensions.
\newblock {\em Comm. Pure Appl. Math.}, 38(5):631--641, 1985.

\bibitem{KLT}
Youngwoo Koh, Sanghyuk Lee, and Ryo Takada.
\newblock Strichartz estimates for the {E}uler equations in the rotational
  framework.
\newblock {\em J. Differential Equations}, 256(2):707--744, 2014.

\bibitem{Zhou}
Ning-An {Lai} and Yi~{Zhou}.
\newblock {Self-dual solution of 3D incompressible Navier-Stokes equations}.
\newblock {\em arXiv e-prints}, page arXiv:2403.16642, March 2024.

\bibitem{Lei}
Zhen Lei.
\newblock Global well-posedness of incompressible elastodynamics in two
  dimensions.
\newblock {\em Comm. Pure Appl. Math.}, 69(11):2072--2106, 2016.

\bibitem{LeiLinZhou}
Zhen Lei, Fang-Hua Lin, and Yi~Zhou.
\newblock Structure of helicity and global solutions of incompressible
  {N}avier-{S}tokes equation.
\newblock {\em Arch. Ration. Mech. Anal.}, 218(3):1417--1430, 2015.

\bibitem{LWZZ}
Zhiwu Lin, Dongyi Wei, Zhifei Zhang, and Hao Zhu.
\newblock The number of traveling wave families in a running water with
  {C}oriolis force.
\newblock {\em Arch. Ration. Mech. Anal.}, 246(2-3):475--533, 2022.

\bibitem{McWilliams}
James~C. McWilliams.
\newblock Fundamentals of geophysical fluid dynamics.
\newblock 2011.

\bibitem{Ped}
Joseph Pedlosky.
\newblock {\em Geophysical Fluid Dynamics, 2nd edn. Springer, Berlin (1987)}.

\bibitem{PW}
Fabio Pusateri and Klaus Widmayer.
\newblock {On the global stability of a beta-plane equation}.
\newblock {\em Anal. PDE}, 11(7):1587 -- 1624, 2018.

\bibitem{Rensq}
Siqi Ren.
\newblock A direct proof of linear decay rate for {E}uler-{C}oriolis equations.
\newblock {\em Acta Appl. Math.}, 188:Paper No. 13, 9, 2023.

\bibitem{RWDZ}
Siqi Ren, Luqi Wang, Dongyi Wei, and Zhifei Zhang.
\newblock Linear inviscid damping and vortex axisymmetrization via the vector
  field method.
\newblock {\em J. Funct. Anal.}, 285(1):Paper No. 109919, 41, 2023.

\bibitem{Shatah}
Jalal Shatah.
\newblock Normal forms and quadratic nonlinear {K}lein-{G}ordon equations.
\newblock {\em Comm. Pure Appl. Math.}, 38(5):685--696, 1985.

\bibitem{Takada}
Ryo Takada.
\newblock Strongly stratified limit for the 3{D} inviscid {B}oussinesq
  equations.
\newblock {\em Arch. Ration. Mech. Anal.}, 232(3):1475--1503, 2019.

\bibitem{WanChen}
Renhui Wan and Jiecheng Chen.
\newblock Decay estimate and well-posedness for the 3{D} {E}uler equations with
  {C}oriolis force.
\newblock {\em Monatsh. Math.}, 185(3):525--536, 2018.

\bibitem{WZZ}
Luqi Wang, Zhifei Zhang, and Hao Zhu.
\newblock Dynamics near {C}ouette flow for the {\it {$\beta $}}-plane equation.
\newblock {\em Adv. Math.}, 432:Paper No. 109261, 67, 2023.

\bibitem{Wei}
Dongyi Wei, Zhifei Zhang, and Weiren Zhao.
\newblock Linear inviscid damping for a class of monotone shear flow in
  {S}obolev spaces.
\newblock {\em Comm. Pure Appl. Math.}, 71(4):617--687, 2018.

\bibitem{WZZcmp}
Dongyi Wei, Zhifei Zhang, and Hao Zhu.
\newblock Linear inviscid damping for the {$\beta$}-plane equation.
\newblock {\em Comm. Math. Phys.}, 375(1):127--174, 2020.

\end{thebibliography}

\end{document}